\documentclass[a4paper,11pt]{amsart}

\usepackage{amsmath,amssymb,amsthm}
\usepackage{tikz}

\newtheorem{theorem}{Theorem}[section]

\newtheorem{corollary}[theorem]{Corollary}
\newtheorem{lemma}[theorem]{Lemma}
\newtheorem{proposition}[theorem]{Proposition}

\theoremstyle{definition}
\newtheorem{definition}[theorem]{Definition}
\newtheorem{remark}[theorem]{Remark}
\newtheorem{example}[theorem]{Example}
\newtheorem{construction}[theorem]{Construction}
\newtheorem{observation}[theorem]{Observation}

\DeclareMathOperator{\walk}{walk}

\newcommand{\Db}{\mathcal{D}^{\rm b}}
\newcommand{\Kb}{\mathcal{K}^{\rm b}}
\newcommand{\Z}{{\mathbb{Z}}}

\DeclareMathOperator{\stabadd}{\underline{add}}
\DeclareMathOperator{\costadd}{\overline{add}}

\newcommand{\iso}{\cong}

\renewcommand{\iff}{\ensuremath{\Longleftrightarrow}}

\DeclareMathAlphabet{\mathpzc}{OT1}{pzc}{m}{it}

\DeclareMathOperator{\modules}{mod} \renewcommand{\mod}{\modules}

\DeclareMathOperator{\proj}{proj}

\DeclareMathOperator{\inj}{inj}

\DeclareMathOperator{\add}{add}

\DeclareMathOperator{\ind}{ind}
\DeclareMathOperator{\stabmod}{\underline{mod}}

\DeclareMathOperator{\costmod}{\overline{mod}}

\DeclareMathOperator{\End}{End}
\DeclareMathOperator{\Hom}{Hom}
\DeclareMathOperator{\Ext}{Ext}
\DeclareMathOperator{\Tor}{Tor}

\DeclareMathOperator{\Tr}{Tr}

\DeclareMathOperator{\Ob}{\mathpzc{Ob}}
\DeclareMathOperator{\Rad}{Rad}
\DeclareMathOperator{\Soc}{Soc}
\DeclareMathOperator{\TopOfModule}{top} \renewcommand{\top}{\TopOfModule}

\DeclareMathOperator{\Ker}{Ker}
\DeclareMathOperator{\Cok}{Cok}
\DeclareMathOperator{\Img}{Im} \renewcommand{\Im}{\Img}

\DeclareMathOperator{\gld}{gl.\!dim}
\DeclareMathOperator{\pd}{pd}
\DeclareMathOperator{\id}{id}

\DeclareMathOperator{\op}{op}

\newenvironment{smallpmatrix}
	       {\left( \! \begin{smallmatrix}}
	       {\end{smallmatrix} \! \right)}

\newcommand{\spm}[1]{\begin{smallpmatrix} #1 \end{smallpmatrix}}

\def\mathclap{\mathpalette\mathclapinternal}
\def\mathclapinternal#1#2{\clap{$\mathsurround=0pt#1{#2}$}}
\def\clap#1{\hbox to 0pt{\hss#1\hss}}

\newcommand{\leftsub}[2]{{\vphantom{#2}}_{#1}{#2}}

\usetikzlibrary{arrows,decorations.pathmorphing,decorations.pathreplacing,positioning}

\tikzset{>=stealth',
         vertex/.style={circle,draw=black,inner sep=1.5pt,outer sep=2pt},
         tvertex/.style={inner sep=1pt,font=\scriptsize},
         gap/.style={fill=white,inner sep=1pt}}

\newcommand{\arrow}[2][20]
 {
  \hspace{-5pt}
  \begin{tikzpicture}
   \node (A) at (0,0) {};
   \node (B) at (#1pt,0) {};
   \draw [#2] (A) -- (B);
  \end{tikzpicture}
  \hspace{-5pt}
 }

\newcommand{\arrowl}[3][20]
 {
  \hspace{-5pt}
  \begin{tikzpicture}
   \node (A) at (0,0) {};
   \node (B) at (#1pt,0) {};
   \draw [#2] (A) -- node [above] {$#3$} (B);
  \end{tikzpicture}
  \hspace{-5pt}
 }

\renewcommand{\to}[1][20]{\arrow[#1]{->}}
\newcommand{\tol}[2][20]{\arrowl[#1]{->}{#2}}

\newcommand{\epi}[1][20]{\arrow[#1]{->>}}
\newcommand{\epil}[2][20]{\arrowl[#1]{->>}{#2}}

\newcommand{\mono}[1][20]{\arrow[#1]{>->}}
\newcommand{\monol}[2][20]{\arrowl[#1]{>->}{#2}}

\renewcommand{\mapsto}[1][20]{\arrow[#1]{|->}}

\renewcommand{\leadsto}[1][20]{\arrow[#1]{->, decorate, decoration={snake, segment length=4pt, amplitude=1pt}}}

\tikzset{admset/.style={ultra thick},
         svertex/.style={circle,draw=black,inner sep=1pt,outer sep=1pt},
         bgarrow/.style={->,ultra thin}}

\newcommand{\iteratedAprLinearAThree}{
\[ \begin{tikzpicture}[xscale=5,yscale=-2.8]
 \node (A) at (0,0)
  {
   \begin{tikzpicture}[scale=1.5,xscale=.29,yscale=-.5]
    \draw (-2.87,-1.5) rectangle (2.87,1.5);
    \node (0) at (0,-1) [vertex] {};
    \node (1) at (-1,0) [vertex] {};
    \node (2) at (1,0) [vertex] {};
    \node (3) at (-2,1) [tvertex] {T};
    \node (4) at (0,1) [vertex] {};
    \node (5) at (2,1) [tvertex] {C};
    \draw [->] (1) -- (0);
    \draw [->] (0) -- (2);
    \draw [loosely dotted, thick] (2) -- (1);
    \draw [->] (3) -- (1);
    \draw [->] (1) -- (4);
    \draw [->] (4) -- (2);
    \draw [->] (2) -- (5);
    \draw [loosely dotted, thick] (4) -- (3);
    \draw [loosely dotted, thick] (5) -- (4);
   \end{tikzpicture}
  };
 \node (B) at (1,0)
  {
   \begin{tikzpicture}[scale=1.5,xscale=.29,yscale=-.5]
    \draw (-2.87,-1.5) rectangle (2.87,1.5);
    \node (0) at (0,-1) [vertex] {};
    \node (1) at (-1,0) [tvertex] {T};
    \node (2) at (1,0) [tvertex] {C};
    \node (3) at (-2,1) [tvertex] {C};
    \node (4) at (0,1) [vertex] {};
    \node (5) at (2,1) [tvertex] {T};
    \draw [->] (1) -- (0);
    \draw [->] (0) -- (2);
    \draw [loosely dotted, thick] (2) -- (1);
    \draw [loosely dotted, thick] (3) -- (1);
    \draw [->] (1) -- (4);
    \draw [->] (4) -- (2);
    \draw [loosely dotted, thick] (2) -- (5);
    \draw [->] (4) -- (3);
    \draw [->] (5) -- (4);
   \end{tikzpicture}
  };
 \node (C) at (0,1)
  {
   \begin{tikzpicture}[scale=1.5,xscale=.29,yscale=-.5]
    \draw (-2.87,-1.5) rectangle (2.87,1.5);
    \node (0) at (0,-1) [vertex] {};
    \node (1) at (-1,0) [tvertex] {T};
    \node (2) at (1,0) [vertex] {};
    \node (3) at (-2,1) [tvertex] {C};
    \node (4) at (0,1) [vertex] {};
    \node (5) at (2,1) [tvertex] {C};
    \draw [->] (1) -- (0);
    \draw [->] (0) -- (2);
    \draw [loosely dotted, thick] (2) -- (1);
    \draw [loosely dotted, thick] (3) -- (1);
    \draw [->] (1) -- (4);
    \draw [->] (4) -- (2);
    \draw [->] (2) -- (5);
    \draw [->] (4) -- (3);
    \draw [loosely dotted, thick] (5) -- (4);
   \end{tikzpicture}
  };
 \node (D) at (1,1)
  {
   \begin{tikzpicture}[scale=1.5,xscale=.29,yscale=-.5]
    \draw (-2.87,-1.5) rectangle (2.87,1.5);
    \node (0) at (0,-1) [tvertex] {T};
    \node (1) at (-1,0) [tvertex] {C};
    \node (2) at (1,0) [vertex] {};
    \node (3) at (-2,1) [vertex] {};
    \node (4) at (0,1) [vertex] {};
    \node (5) at (2,1) [tvertex] {T};
    \draw [loosely dotted, thick] (1) -- (0);
    \draw [->] (0) -- (2);
    \draw [->] (2) -- (1);
    \draw [->] (3) -- (1);
    \draw [loosely dotted, thick] (1) -- (4);
    \draw [->] (4) -- (2);
    \draw [loosely dotted, thick] (2) -- (5);
    \draw [->] (4) -- (3);
    \draw [->] (5) -- (4);
   \end{tikzpicture}
  };
 \node (E) at (0,2)
  {
   \begin{tikzpicture}[scale=1.5,xscale=.29,yscale=-.5]
    \draw (-2.87,-1.5) rectangle (2.87,1.5);
    \node (0) at (0,-1) [tvertex] {T};
    \node (1) at (-1,0) [tvertex] {C};
    \node (2) at (1,0) [vertex] {};
    \node (3) at (-2,1) [vertex] {};
    \node (4) at (0,1) [tvertex] {T};
    \node (5) at (2,1) [tvertex] {C};
    \draw [loosely dotted, thick] (1) -- (0);
    \draw [->] (0) -- (2);
    \draw [->] (2) -- (1);
    \draw [->] (3) -- (1);
    \draw [loosely dotted, thick] (1) -- (4);
    \draw [->] (4) -- (2);
    \draw [->] (2) -- (5);
    \draw [->] (4) -- (3);
    \draw [loosely dotted, thick] (5) -- (4);
   \end{tikzpicture}
  };
 \node (F) at (1,2)
  {
   \begin{tikzpicture}[scale=1.5,xscale=.29,yscale=-.5]
    \draw (-2.87,-1.5) rectangle (2.87,1.5);
    \node (0) at (0,-1) [tvertex] {C};
    \node (1) at (-1,0) [vertex] {};
    \node (2) at (1,0) [vertex] {};
    \node (3) at (-2,1) [vertex] {};
    \node (4) at (0,1) [vertex] {};
    \node (5) at (2,1) [tvertex] {T};
    \draw [->] (1) -- (0);
    \draw [loosely dotted, thick] (0) -- (2);
    \draw [->] (2) -- (1);
    \draw [->] (3) -- (1);
    \draw [loosely dotted, thick] (1) -- (4);
    \draw [->] (4) -- (2);
    \draw [loosely dotted, thick] (2) -- (5);
    \draw [->] (4) -- (3);
    \draw [->] (5) -- (4);
   \end{tikzpicture}
  };
 \node (Eq1) at (-.5,0) [circle,draw=black,thick] {=};
 \node (Eq2) at (1.5,2) [circle,draw=black,thick,rotate=-60] {=};
 \draw [dashed] (-.7,0) -- (Eq1);
 \draw [dashed] (Eq1) -- (A);
 \draw [dashed] (A) -- (B);
 \draw [dashed] (B) -- (1.7,0);
 \draw [dashed] (-.7,2) -- (E);
 \draw [dashed] (E) -- (F);
 \draw [dashed] (F) -- (Eq2);
 \draw [dashed] (Eq2) -- (1.7,2);
 \draw (A) -- (C);
 \draw (B) -- (C);
 \draw (B) -- (D);
 \draw (C) -- (E);
 \draw (D) -- (E);
 \draw (D) -- (F);
\end{tikzpicture} \]}

\newcommand{\iteratedAprNonlinearAThree}{
\[ \begin{tikzpicture}[xscale=4,yscale=-2.8]
 \node (A) at (0,0)
  {
   \begin{tikzpicture}[scale=1.5,xscale=.29,yscale=-.5]
    \draw (-1.87,-1.5) rectangle (2.87,1.5);
    \node (0) at (0,-1) [vertex] {};
    \node (1) at (2,-1) [tvertex] {C};
    \node (2) at (-1,0) [tvertex] {T};
    \node (3) at (1,0) [vertex] {};
    \node (4) at (0,1) [vertex] {};
    \node (5) at (2,1) [tvertex] {C};
    \draw [loosely dotted, thick] (1) -- (0);
    \draw [->] (2) -- (0);
    \draw [->] (0) -- (3);
    \draw [->] (3) -- (1);
    \draw [loosely dotted, thick] (3) -- (2);
    \draw [->] (2) -- (4);
    \draw [->] (4) -- (3);
    \draw [->] (3) -- (5);
    \draw [loosely dotted, thick] (5) -- (4);
   \end{tikzpicture}
  };
 \node (B) at (0,1)
  {
   \begin{tikzpicture}[scale=1.5,xscale=.29,yscale=-.5]
    \draw (-1.87,-1.5) rectangle (2.87,1.5);
    \node (0) at (0,-1) [tvertex] {T};
    \node (1) at (2,-1) [tvertex] {X};
    \node (2) at (-1,0) [tvertex] {C};
    \node (3) at (1,0) [vertex] {};
    \node (4) at (0,1) [tvertex] {T};
    \node (5) at (2,1) [tvertex] {X};
    \draw [loosely dotted, thick] (1) -- (0);
    \draw [loosely dotted, thick] (2) -- (0);
    \draw [->] (0) -- (3);
    \draw [->] (3) -- (1);
    \draw [->] (3) -- (2);
    \draw [loosely dotted, thick] (2) -- (4);
    \draw [->] (4) -- (3);
    \draw [->] (3) -- (5);
    \draw [loosely dotted, thick] (5) -- (4);
   \end{tikzpicture}
  };
 \node (C) at (-1,1.75)
  {
   \begin{tikzpicture}[scale=1.5,xscale=.29,yscale=-.5]
    \draw (-1.87,-1.5) rectangle (2.87,1.5);
    \node (0) at (0,-1) [tvertex] {C};
    \node (1) at (2,-1) [vertex] {};
    \node (2) at (-1,0) [vertex] {};
    \node (3) at (1,0) [vertex] {};
    \node (4) at (0,1) [tvertex] {T};
    \node (5) at (2,1) [tvertex] {X};
    \draw [->] (1) -- (0);
    \draw [->] (2) -- (0);
    \draw [loosely dotted, thick] (0) -- (3);
    \draw [->] (3) -- (1);
    \draw [->] (3) -- (2);
    \draw [loosely dotted, thick] (2) -- (4);
    \draw [->] (4) -- (3);
    \draw [->] (3) -- (5);
    \draw [loosely dotted, thick] (5) -- (4);
   \end{tikzpicture}
  };
 \node (D) at (1,1.75)
  {
   \begin{tikzpicture}[scale=1.5,xscale=.29,yscale=-.5]
    \draw (-1.87,-1.5) rectangle (2.87,1.5);
    \node (0) at (0,-1) [tvertex] {T};
    \node (1) at (2,-1) [tvertex] {X};
    \node (2) at (-1,0) [vertex] {};
    \node (3) at (1,0) [vertex] {};
    \node (4) at (0,1) [tvertex] {C};
    \node (5) at (2,1) [vertex] {};
    \draw [loosely dotted, thick] (1) -- (0);
    \draw [loosely dotted, thick] (2) -- (0);
    \draw [->] (0) -- (3);
    \draw [->] (3) -- (1);
    \draw [->] (3) -- (2);
    \draw [->] (2) -- (4);
    \draw [loosely dotted, thick] (4) -- (3);
    \draw [->] (3) -- (5);
    \draw [->] (5) -- (4);
   \end{tikzpicture}
  };
 \node (E) at (0,2.5)
  {
   \begin{tikzpicture}[scale=1.5,xscale=.29,yscale=-.5]
    \draw (-1.87,-1.5) rectangle (2.87,1.5);
    \node (0) at (0,-1) [tvertex] {C};
    \node (1) at (2,-1) [vertex] {};
    \node (2) at (-1,0) [vertex] {};
    \node (3) at (1,0) [tvertex] {T};
    \node (4) at (0,1) [tvertex] {C};
    \node (5) at (2,1) [vertex] {};
    \draw [->] (1) -- (0);
    \draw [->] (2) -- (0);
    \draw [loosely dotted, thick] (0) -- (3);
    \draw [->] (3) -- (1);
    \draw [->] (3) -- (2);
    \draw [->] (2) -- (4);
    \draw [loosely dotted, thick] (4) -- (3);
    \draw [->] (3) -- (5);
    \draw [->] (5) -- (4);
   \end{tikzpicture}
  };
 \node (F) at (0,3.5)
  {
   \begin{tikzpicture}[scale=1.5,xscale=.29,yscale=-.5]
    \draw (-1.87,-1.5) rectangle (2.87,1.5);
    \node (0) at (0,-1) [vertex] {};
    \node (1) at (2,-1) [tvertex] {T};
    \node (2) at (-1,0) [tvertex] {X};
    \node (3) at (1,0) [tvertex] {C};
    \node (4) at (0,1) [vertex] {};
    \node (5) at (2,1) [tvertex] {T};
    \draw [->] (1) -- (0);
    \draw [->] (2) -- (0);
    \draw [->] (0) -- (3);
    \draw [loosely dotted, thick] (3) -- (1);
    \draw [loosely dotted, thick] (3) -- (2);
    \draw [->] (2) -- (4);
    \draw [->] (4) -- (3);
    \draw [loosely dotted, thick] (3) -- (5);
    \draw [->] (5) -- (4);
   \end{tikzpicture}
  };
 \node (G) at (-1,4.25)
  {
   \begin{tikzpicture}[scale=1.5,xscale=.29,yscale=-.5]
    \draw (-1.87,-1.5) rectangle (2.87,1.5);
    \node (0) at (0,-1) [vertex] {};
    \node (1) at (2,-1) [tvertex] {C};
    \node (2) at (-1,0) [tvertex] {X};
    \node (3) at (1,0) [vertex] {};
    \node (4) at (0,1) [vertex] {};
    \node (5) at (2,1) [tvertex] {T};
    \draw [loosely dotted, thick] (1) -- (0);
    \draw [->] (2) -- (0);
    \draw [->] (0) -- (3);
    \draw [->] (3) -- (1);
    \draw [loosely dotted, thick] (3) -- (2);
    \draw [->] (2) -- (4);
    \draw [->] (4) -- (3);
    \draw [loosely dotted, thick] (3) -- (5);
    \draw [->] (5) -- (4);
   \end{tikzpicture}
  };
 \node (H) at (1,4.25)
  {
   \begin{tikzpicture}[scale=1.5,xscale=.29,yscale=-.5]
    \draw (-1.87,-1.5) rectangle (2.87,1.5);
    \node (0) at (0,-1) [vertex] {};
    \node (1) at (2,-1) [tvertex] {T};
    \node (2) at (-1,0) [tvertex] {X};
    \node (3) at (1,0) [vertex] {};
    \node (4) at (0,1) [vertex] {};
    \node (5) at (2,1) [tvertex] {C};
    \draw [->] (1) -- (0);
    \draw [->] (2) -- (0);
    \draw [->] (0) -- (3);
    \draw [loosely dotted, thick] (3) -- (1);
    \draw [loosely dotted, thick] (3) -- (2);
    \draw [->] (2) -- (4);
    \draw [->] (4) -- (3);
    \draw [->] (3) -- (5);
    \draw [loosely dotted, thick] (5) -- (4);
   \end{tikzpicture}
  };
 \node (I) at (0,5)
  {
   \begin{tikzpicture}[scale=1.5,xscale=.29,yscale=-.5]
    \draw (-1.87,-1.5) rectangle (2.87,1.5);
    \node (0) at (0,-1) [vertex] {};
    \node (1) at (2,-1) [tvertex] {C};
    \node (2) at (-1,0) [tvertex] {T};
    \node (3) at (1,0) [vertex] {};
    \node (4) at (0,1) [vertex] {};
    \node (5) at (2,1) [tvertex] {C};
    \draw [loosely dotted, thick] (1) -- (0);
    \draw [->] (2) -- (0);
    \draw [->] (0) -- (3);
    \draw [->] (3) -- (1);
    \draw [loosely dotted, thick] (3) -- (2);
    \draw [->] (2) -- (4);
    \draw [->] (4) -- (3);
    \draw [->] (3) -- (5);
    \draw [loosely dotted, thick] (5) -- (4);
   \end{tikzpicture}
  };
 \node (Eq1) at (-1,0) [circle,draw=black,thick] {=};
 \node (Eq2) at (-1,5) [circle,draw=black,thick] {=};
 \draw [dashed] (-1.5,0) -- (Eq1);
 \draw [dashed] (Eq1) -- (A);
 \draw [dashed] (A) -- (1.5,0);
 \draw [dashed] (-1.5,5) -- (Eq2);
 \draw [dashed] (Eq2) -- (I);
 \draw [dashed] (I) -- (1.5,5);
 \draw (A) -- (B);
 \draw (B) -- (C);
 \draw (B) -- (D);
 \draw (C) -- (E);
 \draw (D) -- (E);
 \draw (E) -- (F);
 \draw (F) -- (G);
 \draw (F) -- (H);
 \draw (G) -- (I);
 \draw (H) -- (I);
\end{tikzpicture} \]}

\newcommand{\iteratedAprLinearAFour}{
\[ \begin{tikzpicture}[xscale=3,yscale=-2]
 \node (A1) at (0,0)
  {
   \begin{tikzpicture}[scale=.9,xscale=.29,yscale=-.5]
    \node (030) at (0,-1) [svertex] {};
    \node (120) at (-1,0) [svertex] {};
    \node (021) at (1,0) [svertex] {};
    \node (210) at (-2,1) [svertex] {};
    \node (111) at (0,1) [svertex] {};
    \node (012) at (2,1) [svertex] {};
    \node (300) at (-3,2) [svertex] {};
    \node (201) at (-1,2) [svertex] {};
    \node (102) at (1,2) [svertex] {};
    \node (003) at (3,2) [svertex] {};
    \draw [bgarrow] (120) -- (030);
    \draw [bgarrow] (030) -- (021);
    \draw [admset] (021) -- (120);
    \draw [bgarrow] (210) -- (120);
    \draw [bgarrow] (120) -- (111);
    \draw [bgarrow] (111) -- (021);
    \draw [bgarrow] (021) -- (012);
    \draw [admset] (012) -- (111);
    \draw [admset] (111) -- (210);
    \draw [bgarrow] (300) -- (210);
    \draw [bgarrow] (210) -- (201);
    \draw [bgarrow] (201) -- (111);
    \draw [bgarrow] (111) -- (102);
    \draw [bgarrow] (102) -- (012);
    \draw [bgarrow] (012) -- (003);
    \draw [admset] (003) -- (102);
    \draw [admset] (102) -- (201);
    \draw [admset] (201) -- (300);
   \end{tikzpicture}
  };
 \node (A2) at (1,0)
  {
   \begin{tikzpicture}[scale=.9,xscale=.29,yscale=-.5]
    \node (030) at (0,-1) [svertex] {};
    \node (120) at (-1,0) [svertex] {};
    \node (021) at (1,0) [svertex] {};
    \node (210) at (-2,1) [svertex] {};
    \node (111) at (0,1) [svertex] {};
    \node (012) at (2,1) [svertex] {};
    \node (300) at (-3,2) [svertex] {};
    \node (201) at (-1,2) [svertex] {};
    \node (102) at (1,2) [svertex] {};
    \node (003) at (3,2) [svertex] {};
    \draw [bgarrow] (120) -- (030);
    \draw [bgarrow] (030) -- (021);
    \draw [admset] (021) -- (120);
    \draw [bgarrow] (210) -- (120);
    \draw [bgarrow] (120) -- (111);
    \draw [bgarrow] (111) -- (021);
    \draw [bgarrow] (021) -- (012);
    \draw [admset] (012) -- (111);
    \draw [admset] (111) -- (210);
    \draw [admset] (300) -- (210);
    \draw [bgarrow] (210) -- (201);
    \draw [bgarrow] (201) -- (111);
    \draw [bgarrow] (111) -- (102);
    \draw [bgarrow] (102) -- (012);
    \draw [admset] (012) -- (003);
    \draw [bgarrow] (003) -- (102);
    \draw [admset] (102) -- (201);
    \draw [bgarrow] (201) -- (300);
   \end{tikzpicture}
  };
 \node (A3) at (2,0)
  {
   \begin{tikzpicture}[scale=.9,xscale=.29,yscale=-.5]
    \node (030) at (0,-1) [svertex] {};
    \node (120) at (-1,0) [svertex] {};
    \node (021) at (1,0) [svertex] {};
    \node (210) at (-2,1) [svertex] {};
    \node (111) at (0,1) [svertex] {};
    \node (012) at (2,1) [svertex] {};
    \node (300) at (-3,2) [svertex] {};
    \node (201) at (-1,2) [svertex] {};
    \node (102) at (1,2) [svertex] {};
    \node (003) at (3,2) [svertex] {};
    \draw [bgarrow] (120) -- (030);
    \draw [bgarrow] (030) -- (021);
    \draw [admset] (021) -- (120);
    \draw [admset] (210) -- (120);
    \draw [bgarrow] (120) -- (111);
    \draw [bgarrow] (111) -- (021);
    \draw [admset] (021) -- (012);
    \draw [bgarrow] (012) -- (111);
    \draw [bgarrow] (111) -- (210);
    \draw [bgarrow] (300) -- (210);
    \draw [admset] (210) -- (201);
    \draw [bgarrow] (201) -- (111);
    \draw [bgarrow] (111) -- (102);
    \draw [admset] (102) -- (012);
    \draw [bgarrow] (012) -- (003);
    \draw [bgarrow] (003) -- (102);
    \draw [admset] (102) -- (201);
    \draw [bgarrow] (201) -- (300);
   \end{tikzpicture}
  };
 \node (B1) at (.5,1)
  {
   \begin{tikzpicture}[scale=.9,xscale=.29,yscale=-.5]
    \node (030) at (0,-1) [svertex] {};
    \node (120) at (-1,0) [svertex] {};
    \node (021) at (1,0) [svertex] {};
    \node (210) at (-2,1) [svertex] {};
    \node (111) at (0,1) [svertex] {};
    \node (012) at (2,1) [svertex] {};
    \node (300) at (-3,2) [svertex] {};
    \node (201) at (-1,2) [svertex] {};
    \node (102) at (1,2) [svertex] {};
    \node (003) at (3,2) [svertex] {};
    \draw [bgarrow] (120) -- (030);
    \draw [bgarrow] (030) -- (021);
    \draw [admset] (021) -- (120);
    \draw [bgarrow] (210) -- (120);
    \draw [bgarrow] (120) -- (111);
    \draw [bgarrow] (111) -- (021);
    \draw [bgarrow] (021) -- (012);
    \draw [admset] (012) -- (111);
    \draw [admset] (111) -- (210);
    \draw [admset] (300) -- (210);
    \draw [bgarrow] (210) -- (201);
    \draw [bgarrow] (201) -- (111);
    \draw [bgarrow] (111) -- (102);
    \draw [bgarrow] (102) -- (012);
    \draw [bgarrow] (012) -- (003);
    \draw [admset] (003) -- (102);
    \draw [admset] (102) -- (201);
    \draw [bgarrow] (201) -- (300);
   \end{tikzpicture}
  };
 \node (B2) at (1.5,1)
  {
   \begin{tikzpicture}[scale=.9,xscale=.29,yscale=-.5]
    \node (030) at (0,-1) [svertex] {};
    \node (120) at (-1,0) [svertex] {};
    \node (021) at (1,0) [svertex] {};
    \node (210) at (-2,1) [svertex] {};
    \node (111) at (0,1) [svertex] {};
    \node (012) at (2,1) [svertex] {};
    \node (300) at (-3,2) [svertex] {};
    \node (201) at (-1,2) [svertex] {};
    \node (102) at (1,2) [svertex] {};
    \node (003) at (3,2) [svertex] {};
    \draw [bgarrow] (120) -- (030);
    \draw [bgarrow] (030) -- (021);
    \draw [admset] (021) -- (120);
    \draw [admset] (210) -- (120);
    \draw [bgarrow] (120) -- (111);
    \draw [bgarrow] (111) -- (021);
    \draw [bgarrow] (021) -- (012);
    \draw [admset] (012) -- (111);
    \draw [bgarrow] (111) -- (210);
    \draw [bgarrow] (300) -- (210);
    \draw [admset] (210) -- (201);
    \draw [bgarrow] (201) -- (111);
    \draw [bgarrow] (111) -- (102);
    \draw [bgarrow] (102) -- (012);
    \draw [admset] (012) -- (003);
    \draw [bgarrow] (003) -- (102);
    \draw [admset] (102) -- (201);
    \draw [bgarrow] (201) -- (300);
   \end{tikzpicture}
  };
 \node (C1) at (.5,2)
  {
   \begin{tikzpicture}[scale=.9,xscale=.29,yscale=-.5]
    \node (030) at (0,-1) [svertex] {};
    \node (120) at (-1,0) [svertex] {};
    \node (021) at (1,0) [svertex] {};
    \node (210) at (-2,1) [svertex] {};
    \node (111) at (0,1) [svertex] {};
    \node (012) at (2,1) [svertex] {};
    \node (300) at (-3,2) [svertex] {};
    \node (201) at (-1,2) [svertex] {};
    \node (102) at (1,2) [svertex] {};
    \node (003) at (3,2) [svertex] {};
    \draw [bgarrow] (120) -- (030);
    \draw [bgarrow] (030) -- (021);
    \draw [admset] (021) -- (120);
    \draw [admset] (210) -- (120);
    \draw [bgarrow] (120) -- (111);
    \draw [bgarrow] (111) -- (021);
    \draw [bgarrow] (021) -- (012);
    \draw [admset] (012) -- (111);
    \draw [bgarrow] (111) -- (210);
    \draw [bgarrow] (300) -- (210);
    \draw [admset] (210) -- (201);
    \draw [bgarrow] (201) -- (111);
    \draw [bgarrow] (111) -- (102);
    \draw [bgarrow] (102) -- (012);
    \draw [bgarrow] (012) -- (003);
    \draw [admset] (003) -- (102);
    \draw [admset] (102) -- (201);
    \draw [bgarrow] (201) -- (300);
   \end{tikzpicture}
  };
 \node (C2) at (1.5,2)
  {
   \begin{tikzpicture}[scale=.9,xscale=.29,yscale=-.5]
    \node (030) at (0,-1) [svertex] {};
    \node (120) at (-1,0) [svertex] {};
    \node (021) at (1,0) [svertex] {};
    \node (210) at (-2,1) [svertex] {};
    \node (111) at (0,1) [svertex] {};
    \node (012) at (2,1) [svertex] {};
    \node (300) at (-3,2) [svertex] {};
    \node (201) at (-1,2) [svertex] {};
    \node (102) at (1,2) [svertex] {};
    \node (003) at (3,2) [svertex] {};
    \draw [bgarrow] (120) -- (030);
    \draw [bgarrow] (030) -- (021);
    \draw [admset] (021) -- (120);
    \draw [admset] (210) -- (120);
    \draw [bgarrow] (120) -- (111);
    \draw [bgarrow] (111) -- (021);
    \draw [bgarrow] (021) -- (012);
    \draw [admset] (012) -- (111);
    \draw [bgarrow] (111) -- (210);
    \draw [bgarrow] (300) -- (210);
    \draw [bgarrow] (210) -- (201);
    \draw [admset] (201) -- (111);
    \draw [bgarrow] (111) -- (102);
    \draw [bgarrow] (102) -- (012);
    \draw [admset] (012) -- (003);
    \draw [bgarrow] (003) -- (102);
    \draw [bgarrow] (102) -- (201);
    \draw [admset] (201) -- (300);
   \end{tikzpicture}
  };
 \node (D1) at (0,3)
  {
   \begin{tikzpicture}[scale=.9,xscale=.29,yscale=-.5]
    \node (030) at (0,-1) [svertex] {};
    \node (120) at (-1,0) [svertex] {};
    \node (021) at (1,0) [svertex] {};
    \node (210) at (-2,1) [svertex] {};
    \node (111) at (0,1) [svertex] {};
    \node (012) at (2,1) [svertex] {};
    \node (300) at (-3,2) [svertex] {};
    \node (201) at (-1,2) [svertex] {};
    \node (102) at (1,2) [svertex] {};
    \node (003) at (3,2) [svertex] {};
    \draw [admset] (120) -- (030);
    \draw [bgarrow] (030) -- (021);
    \draw [bgarrow] (021) -- (120);
    \draw [bgarrow] (210) -- (120);
    \draw [admset] (120) -- (111);
    \draw [bgarrow] (111) -- (021);
    \draw [bgarrow] (021) -- (012);
    \draw [admset] (012) -- (111);
    \draw [bgarrow] (111) -- (210);
    \draw [bgarrow] (300) -- (210);
    \draw [admset] (210) -- (201);
    \draw [bgarrow] (201) -- (111);
    \draw [bgarrow] (111) -- (102);
    \draw [bgarrow] (102) -- (012);
    \draw [bgarrow] (012) -- (003);
    \draw [admset] (003) -- (102);
    \draw [admset] (102) -- (201);
    \draw [bgarrow] (201) -- (300);
   \end{tikzpicture}
  };
 \node (D2) at (1,3)
  {
   \begin{tikzpicture}[scale=.9,xscale=.29,yscale=-.5]
    \node (030) at (0,-1) [svertex] {};
    \node (120) at (-1,0) [svertex] {};
    \node (021) at (1,0) [svertex] {};
    \node (210) at (-2,1) [svertex] {};
    \node (111) at (0,1) [svertex] {};
    \node (012) at (2,1) [svertex] {};
    \node (300) at (-3,2) [svertex] {};
    \node (201) at (-1,2) [svertex] {};
    \node (102) at (1,2) [svertex] {};
    \node (003) at (3,2) [svertex] {};
    \draw [bgarrow] (120) -- (030);
    \draw [bgarrow] (030) -- (021);
    \draw [admset] (021) -- (120);
    \draw [admset] (210) -- (120);
    \draw [bgarrow] (120) -- (111);
    \draw [bgarrow] (111) -- (021);
    \draw [bgarrow] (021) -- (012);
    \draw [admset] (012) -- (111);
    \draw [bgarrow] (111) -- (210);
    \draw [bgarrow] (300) -- (210);
    \draw [bgarrow] (210) -- (201);
    \draw [admset] (201) -- (111);
    \draw [bgarrow] (111) -- (102);
    \draw [bgarrow] (102) -- (012);
    \draw [bgarrow] (012) -- (003);
    \draw [admset] (003) -- (102);
    \draw [bgarrow] (102) -- (201);
    \draw [admset] (201) -- (300);
   \end{tikzpicture}
  };
 \node (D3) at (2,3)
  {
   \begin{tikzpicture}[scale=.9,xscale=.29,yscale=-.5]
    \node (030) at (0,-1) [svertex] {};
    \node (120) at (-1,0) [svertex] {};
    \node (021) at (1,0) [svertex] {};
    \node (210) at (-2,1) [svertex] {};
    \node (111) at (0,1) [svertex] {};
    \node (012) at (2,1) [svertex] {};
    \node (300) at (-3,2) [svertex] {};
    \node (201) at (-1,2) [svertex] {};
    \node (102) at (1,2) [svertex] {};
    \node (003) at (3,2) [svertex] {};
    \draw [admset] (120) -- (030);
    \draw [bgarrow] (030) -- (021);
    \draw [bgarrow] (021) -- (120);
    \draw [bgarrow] (210) -- (120);
    \draw [admset] (120) -- (111);
    \draw [bgarrow] (111) -- (021);
    \draw [bgarrow] (021) -- (012);
    \draw [admset] (012) -- (111);
    \draw [bgarrow] (111) -- (210);
    \draw [bgarrow] (300) -- (210);
    \draw [bgarrow] (210) -- (201);
    \draw [admset] (201) -- (111);
    \draw [bgarrow] (111) -- (102);
    \draw [bgarrow] (102) -- (012);
    \draw [admset] (012) -- (003);
    \draw [bgarrow] (003) -- (102);
    \draw [bgarrow] (102) -- (201);
    \draw [admset] (201) -- (300);
   \end{tikzpicture}
  };
 \node (E1) at (0,4)
  {
   \begin{tikzpicture}[scale=.9,xscale=.29,yscale=-.5]
    \node (030) at (0,-1) [svertex] {};
    \node (120) at (-1,0) [svertex] {};
    \node (021) at (1,0) [svertex] {};
    \node (210) at (-2,1) [svertex] {};
    \node (111) at (0,1) [svertex] {};
    \node (012) at (2,1) [svertex] {};
    \node (300) at (-3,2) [svertex] {};
    \node (201) at (-1,2) [svertex] {};
    \node (102) at (1,2) [svertex] {};
    \node (003) at (3,2) [svertex] {};
    \draw [bgarrow] (120) -- (030);
    \draw [bgarrow] (030) -- (021);
    \draw [admset] (021) -- (120);
    \draw [admset] (210) -- (120);
    \draw [bgarrow] (120) -- (111);
    \draw [bgarrow] (111) -- (021);
    \draw [bgarrow] (021) -- (012);
    \draw [admset] (012) -- (111);
    \draw [bgarrow] (111) -- (210);
    \draw [admset] (300) -- (210);
    \draw [bgarrow] (210) -- (201);
    \draw [admset] (201) -- (111);
    \draw [bgarrow] (111) -- (102);
    \draw [bgarrow] (102) -- (012);
    \draw [bgarrow] (012) -- (003);
    \draw [admset] (003) -- (102);
    \draw [bgarrow] (102) -- (201);
    \draw [bgarrow] (201) -- (300);
   \end{tikzpicture}
  };
 \node (E2) at (1,4)
  {
   \begin{tikzpicture}[scale=.9,xscale=.29,yscale=-.5]
    \node (030) at (0,-1) [svertex] {};
    \node (120) at (-1,0) [svertex] {};
    \node (021) at (1,0) [svertex] {};
    \node (210) at (-2,1) [svertex] {};
    \node (111) at (0,1) [svertex] {};
    \node (012) at (2,1) [svertex] {};
    \node (300) at (-3,2) [svertex] {};
    \node (201) at (-1,2) [svertex] {};
    \node (102) at (1,2) [svertex] {};
    \node (003) at (3,2) [svertex] {};
    \draw [admset] (120) -- (030);
    \draw [bgarrow] (030) -- (021);
    \draw [bgarrow] (021) -- (120);
    \draw [bgarrow] (210) -- (120);
    \draw [admset] (120) -- (111);
    \draw [bgarrow] (111) -- (021);
    \draw [bgarrow] (021) -- (012);
    \draw [admset] (012) -- (111);
    \draw [bgarrow] (111) -- (210);
    \draw [bgarrow] (300) -- (210);
    \draw [bgarrow] (210) -- (201);
    \draw [admset] (201) -- (111);
    \draw [bgarrow] (111) -- (102);
    \draw [bgarrow] (102) -- (012);
    \draw [bgarrow] (012) -- (003);
    \draw [admset] (003) -- (102);
    \draw [bgarrow] (102) -- (201);
    \draw [admset] (201) -- (300);
   \end{tikzpicture}
  };
 \node (E3) at (2,4)
  {
   \begin{tikzpicture}[scale=.9,xscale=.29,yscale=-.5]
    \node (030) at (0,-1) [svertex] {};
    \node (120) at (-1,0) [svertex] {};
    \node (021) at (1,0) [svertex] {};
    \node (210) at (-2,1) [svertex] {};
    \node (111) at (0,1) [svertex] {};
    \node (012) at (2,1) [svertex] {};
    \node (300) at (-3,2) [svertex] {};
    \node (201) at (-1,2) [svertex] {};
    \node (102) at (1,2) [svertex] {};
    \node (003) at (3,2) [svertex] {};
    \draw [admset] (120) -- (030);
    \draw [bgarrow] (030) -- (021);
    \draw [bgarrow] (021) -- (120);
    \draw [bgarrow] (210) -- (120);
    \draw [bgarrow] (120) -- (111);
    \draw [admset] (111) -- (021);
    \draw [bgarrow] (021) -- (012);
    \draw [bgarrow] (012) -- (111);
    \draw [admset] (111) -- (210);
    \draw [bgarrow] (300) -- (210);
    \draw [bgarrow] (210) -- (201);
    \draw [bgarrow] (201) -- (111);
    \draw [admset] (111) -- (102);
    \draw [bgarrow] (102) -- (012);
    \draw [admset] (012) -- (003);
    \draw [bgarrow] (003) -- (102);
    \draw [bgarrow] (102) -- (201);
    \draw [admset] (201) -- (300);
   \end{tikzpicture}
  };
 \node (F1) at (.5,5)
  {
   \begin{tikzpicture}[scale=.9,xscale=.29,yscale=-.5]
    \node (030) at (0,-1) [svertex] {};
    \node (120) at (-1,0) [svertex] {};
    \node (021) at (1,0) [svertex] {};
    \node (210) at (-2,1) [svertex] {};
    \node (111) at (0,1) [svertex] {};
    \node (012) at (2,1) [svertex] {};
    \node (300) at (-3,2) [svertex] {};
    \node (201) at (-1,2) [svertex] {};
    \node (102) at (1,2) [svertex] {};
    \node (003) at (3,2) [svertex] {};
    \draw [bgarrow] (120) -- (030);
    \draw [admset] (030) -- (021);
    \draw [bgarrow] (021) -- (120);
    \draw [bgarrow] (210) -- (120);
    \draw [admset] (120) -- (111);
    \draw [bgarrow] (111) -- (021);
    \draw [bgarrow] (021) -- (012);
    \draw [admset] (012) -- (111);
    \draw [bgarrow] (111) -- (210);
    \draw [bgarrow] (300) -- (210);
    \draw [bgarrow] (210) -- (201);
    \draw [admset] (201) -- (111);
    \draw [bgarrow] (111) -- (102);
    \draw [bgarrow] (102) -- (012);
    \draw [bgarrow] (012) -- (003);
    \draw [admset] (003) -- (102);
    \draw [bgarrow] (102) -- (201);
    \draw [admset] (201) -- (300);
   \end{tikzpicture}
  };
 \node (F2) at (1.5,5)
  {
   \begin{tikzpicture}[scale=.9,xscale=.29,yscale=-.5]
    \node (030) at (0,-1) [svertex] {};
    \node (120) at (-1,0) [svertex] {};
    \node (021) at (1,0) [svertex] {};
    \node (210) at (-2,1) [svertex] {};
    \node (111) at (0,1) [svertex] {};
    \node (012) at (2,1) [svertex] {};
    \node (300) at (-3,2) [svertex] {};
    \node (201) at (-1,2) [svertex] {};
    \node (102) at (1,2) [svertex] {};
    \node (003) at (3,2) [svertex] {};
    \draw [admset] (120) -- (030);
    \draw [bgarrow] (030) -- (021);
    \draw [bgarrow] (021) -- (120);
    \draw [bgarrow] (210) -- (120);
    \draw [bgarrow] (120) -- (111);
    \draw [admset] (111) -- (021);
    \draw [bgarrow] (021) -- (012);
    \draw [bgarrow] (012) -- (111);
    \draw [admset] (111) -- (210);
    \draw [bgarrow] (300) -- (210);
    \draw [bgarrow] (210) -- (201);
    \draw [bgarrow] (201) -- (111);
    \draw [admset] (111) -- (102);
    \draw [bgarrow] (102) -- (012);
    \draw [bgarrow] (012) -- (003);
    \draw [admset] (003) -- (102);
    \draw [bgarrow] (102) -- (201);
    \draw [admset] (201) -- (300);
   \end{tikzpicture}
  };
 \node (G1) at (0,6)
  {
   \begin{tikzpicture}[scale=.9,xscale=.29,yscale=-.5]
    \node (030) at (0,-1) [svertex] {};
    \node (120) at (-1,0) [svertex] {};
    \node (021) at (1,0) [svertex] {};
    \node (210) at (-2,1) [svertex] {};
    \node (111) at (0,1) [svertex] {};
    \node (012) at (2,1) [svertex] {};
    \node (300) at (-3,2) [svertex] {};
    \node (201) at (-1,2) [svertex] {};
    \node (102) at (1,2) [svertex] {};
    \node (003) at (3,2) [svertex] {};
    \draw [bgarrow] (120) -- (030);
    \draw [admset] (030) -- (021);
    \draw [bgarrow] (021) -- (120);
    \draw [bgarrow] (210) -- (120);
    \draw [admset] (120) -- (111);
    \draw [bgarrow] (111) -- (021);
    \draw [bgarrow] (021) -- (012);
    \draw [admset] (012) -- (111);
    \draw [bgarrow] (111) -- (210);
    \draw [admset] (300) -- (210);
    \draw [bgarrow] (210) -- (201);
    \draw [admset] (201) -- (111);
    \draw [bgarrow] (111) -- (102);
    \draw [bgarrow] (102) -- (012);
    \draw [bgarrow] (012) -- (003);
    \draw [admset] (003) -- (102);
    \draw [bgarrow] (102) -- (201);
    \draw [bgarrow] (201) -- (300);
   \end{tikzpicture}
  };
 \node (G2) at (1,6)
  {
   \begin{tikzpicture}[scale=.9,xscale=.29,yscale=-.5]
    \node (030) at (0,-1) [svertex] {};
    \node (120) at (-1,0) [svertex] {};
    \node (021) at (1,0) [svertex] {};
    \node (210) at (-2,1) [svertex] {};
    \node (111) at (0,1) [svertex] {};
    \node (012) at (2,1) [svertex] {};
    \node (300) at (-3,2) [svertex] {};
    \node (201) at (-1,2) [svertex] {};
    \node (102) at (1,2) [svertex] {};
    \node (003) at (3,2) [svertex] {};
    \draw [bgarrow] (120) -- (030);
    \draw [admset] (030) -- (021);
    \draw [bgarrow] (021) -- (120);
    \draw [bgarrow] (210) -- (120);
    \draw [bgarrow] (120) -- (111);
    \draw [admset] (111) -- (021);
    \draw [bgarrow] (021) -- (012);
    \draw [bgarrow] (012) -- (111);
    \draw [admset] (111) -- (210);
    \draw [bgarrow] (300) -- (210);
    \draw [bgarrow] (210) -- (201);
    \draw [bgarrow] (201) -- (111);
    \draw [admset] (111) -- (102);
    \draw [bgarrow] (102) -- (012);
    \draw [bgarrow] (012) -- (003);
    \draw [admset] (003) -- (102);
    \draw [bgarrow] (102) -- (201);
    \draw [admset] (201) -- (300);
   \end{tikzpicture}
  };
 \node (G3) at (2,6)
  {
   \begin{tikzpicture}[scale=.9,xscale=.29,yscale=-.5]
    \node (030) at (0,-1) [svertex] {};
    \node (120) at (-1,0) [svertex] {};
    \node (021) at (1,0) [svertex] {};
    \node (210) at (-2,1) [svertex] {};
    \node (111) at (0,1) [svertex] {};
    \node (012) at (2,1) [svertex] {};
    \node (300) at (-3,2) [svertex] {};
    \node (201) at (-1,2) [svertex] {};
    \node (102) at (1,2) [svertex] {};
    \node (003) at (3,2) [svertex] {};
    \draw [admset] (120) -- (030);
    \draw [bgarrow] (030) -- (021);
    \draw [bgarrow] (021) -- (120);
    \draw [bgarrow] (210) -- (120);
    \draw [bgarrow] (120) -- (111);
    \draw [admset] (111) -- (021);
    \draw [bgarrow] (021) -- (012);
    \draw [bgarrow] (012) -- (111);
    \draw [admset] (111) -- (210);
    \draw [bgarrow] (300) -- (210);
    \draw [bgarrow] (210) -- (201);
    \draw [bgarrow] (201) -- (111);
    \draw [bgarrow] (111) -- (102);
    \draw [admset] (102) -- (012);
    \draw [bgarrow] (012) -- (003);
    \draw [bgarrow] (003) -- (102);
    \draw [admset] (102) -- (201);
    \draw [admset] (201) -- (300);
   \end{tikzpicture}
  };
 \node (H1) at (0,7)
  {
   \begin{tikzpicture}[scale=.9,xscale=.29,yscale=-.5]
    \node (030) at (0,-1) [svertex] {};
    \node (120) at (-1,0) [svertex] {};
    \node (021) at (1,0) [svertex] {};
    \node (210) at (-2,1) [svertex] {};
    \node (111) at (0,1) [svertex] {};
    \node (012) at (2,1) [svertex] {};
    \node (300) at (-3,2) [svertex] {};
    \node (201) at (-1,2) [svertex] {};
    \node (102) at (1,2) [svertex] {};
    \node (003) at (3,2) [svertex] {};
    \draw [bgarrow] (120) -- (030);
    \draw [admset] (030) -- (021);
    \draw [bgarrow] (021) -- (120);
    \draw [bgarrow] (210) -- (120);
    \draw [bgarrow] (120) -- (111);
    \draw [admset] (111) -- (021);
    \draw [bgarrow] (021) -- (012);
    \draw [bgarrow] (012) -- (111);
    \draw [admset] (111) -- (210);
    \draw [admset] (300) -- (210);
    \draw [bgarrow] (210) -- (201);
    \draw [bgarrow] (201) -- (111);
    \draw [admset] (111) -- (102);
    \draw [bgarrow] (102) -- (012);
    \draw [bgarrow] (012) -- (003);
    \draw [admset] (003) -- (102);
    \draw [bgarrow] (102) -- (201);
    \draw [bgarrow] (201) -- (300);
   \end{tikzpicture}
  };
 \node (H2) at (1,7)
  {
   \begin{tikzpicture}[scale=.9,xscale=.29,yscale=-.5]
    \node (030) at (0,-1) [svertex] {};
    \node (120) at (-1,0) [svertex] {};
    \node (021) at (1,0) [svertex] {};
    \node (210) at (-2,1) [svertex] {};
    \node (111) at (0,1) [svertex] {};
    \node (012) at (2,1) [svertex] {};
    \node (300) at (-3,2) [svertex] {};
    \node (201) at (-1,2) [svertex] {};
    \node (102) at (1,2) [svertex] {};
    \node (003) at (3,2) [svertex] {};
    \draw [bgarrow] (120) -- (030);
    \draw [bgarrow] (030) -- (021);
    \draw [admset] (021) -- (120);
    \draw [bgarrow] (210) -- (120);
    \draw [bgarrow] (120) -- (111);
    \draw [bgarrow] (111) -- (021);
    \draw [admset] (021) -- (012);
    \draw [bgarrow] (012) -- (111);
    \draw [admset] (111) -- (210);
    \draw [bgarrow] (300) -- (210);
    \draw [bgarrow] (210) -- (201);
    \draw [bgarrow] (201) -- (111);
    \draw [admset] (111) -- (102);
    \draw [bgarrow] (102) -- (012);
    \draw [bgarrow] (012) -- (003);
    \draw [admset] (003) -- (102);
    \draw [bgarrow] (102) -- (201);
    \draw [admset] (201) -- (300);
   \end{tikzpicture}
  };
 \node (H3) at (2,7)
  {
   \begin{tikzpicture}[scale=.9,xscale=.29,yscale=-.5]
    \node (030) at (0,-1) [svertex] {};
    \node (120) at (-1,0) [svertex] {};
    \node (021) at (1,0) [svertex] {};
    \node (210) at (-2,1) [svertex] {};
    \node (111) at (0,1) [svertex] {};
    \node (012) at (2,1) [svertex] {};
    \node (300) at (-3,2) [svertex] {};
    \node (201) at (-1,2) [svertex] {};
    \node (102) at (1,2) [svertex] {};
    \node (003) at (3,2) [svertex] {};
    \draw [bgarrow] (120) -- (030);
    \draw [admset] (030) -- (021);
    \draw [bgarrow] (021) -- (120);
    \draw [bgarrow] (210) -- (120);
    \draw [bgarrow] (120) -- (111);
    \draw [admset] (111) -- (021);
    \draw [bgarrow] (021) -- (012);
    \draw [bgarrow] (012) -- (111);
    \draw [admset] (111) -- (210);
    \draw [bgarrow] (300) -- (210);
    \draw [bgarrow] (210) -- (201);
    \draw [bgarrow] (201) -- (111);
    \draw [bgarrow] (111) -- (102);
    \draw [admset] (102) -- (012);
    \draw [bgarrow] (012) -- (003);
    \draw [bgarrow] (003) -- (102);
    \draw [admset] (102) -- (201);
    \draw [admset] (201) -- (300);
   \end{tikzpicture}
  };
 \node (I1) at (.5,8)
  {
   \begin{tikzpicture}[scale=.9,xscale=.29,yscale=-.5]
    \node (030) at (0,-1) [svertex] {};
    \node (120) at (-1,0) [svertex] {};
    \node (021) at (1,0) [svertex] {};
    \node (210) at (-2,1) [svertex] {};
    \node (111) at (0,1) [svertex] {};
    \node (012) at (2,1) [svertex] {};
    \node (300) at (-3,2) [svertex] {};
    \node (201) at (-1,2) [svertex] {};
    \node (102) at (1,2) [svertex] {};
    \node (003) at (3,2) [svertex] {};
    \draw [bgarrow] (120) -- (030);
    \draw [bgarrow] (030) -- (021);
    \draw [admset] (021) -- (120);
    \draw [bgarrow] (210) -- (120);
    \draw [bgarrow] (120) -- (111);
    \draw [bgarrow] (111) -- (021);
    \draw [admset] (021) -- (012);
    \draw [bgarrow] (012) -- (111);
    \draw [admset] (111) -- (210);
    \draw [admset] (300) -- (210);
    \draw [bgarrow] (210) -- (201);
    \draw [bgarrow] (201) -- (111);
    \draw [admset] (111) -- (102);
    \draw [bgarrow] (102) -- (012);
    \draw [bgarrow] (012) -- (003);
    \draw [admset] (003) -- (102);
    \draw [bgarrow] (102) -- (201);
    \draw [bgarrow] (201) -- (300);
   \end{tikzpicture}
  };
 \node (I2) at (1.5,8)
  {
   \begin{tikzpicture}[scale=.9,xscale=.29,yscale=-.5]
    \node (030) at (0,-1) [svertex] {};
    \node (120) at (-1,0) [svertex] {};
    \node (021) at (1,0) [svertex] {};
    \node (210) at (-2,1) [svertex] {};
    \node (111) at (0,1) [svertex] {};
    \node (012) at (2,1) [svertex] {};
    \node (300) at (-3,2) [svertex] {};
    \node (201) at (-1,2) [svertex] {};
    \node (102) at (1,2) [svertex] {};
    \node (003) at (3,2) [svertex] {};
    \draw [bgarrow] (120) -- (030);
    \draw [bgarrow] (030) -- (021);
    \draw [admset] (021) -- (120);
    \draw [bgarrow] (210) -- (120);
    \draw [bgarrow] (120) -- (111);
    \draw [bgarrow] (111) -- (021);
    \draw [admset] (021) -- (012);
    \draw [bgarrow] (012) -- (111);
    \draw [admset] (111) -- (210);
    \draw [bgarrow] (300) -- (210);
    \draw [bgarrow] (210) -- (201);
    \draw [bgarrow] (201) -- (111);
    \draw [bgarrow] (111) -- (102);
    \draw [admset] (102) -- (012);
    \draw [bgarrow] (012) -- (003);
    \draw [bgarrow] (003) -- (102);
    \draw [admset] (102) -- (201);
    \draw [admset] (201) -- (300);
   \end{tikzpicture}
  };
 \node (J1) at (.5,9)
  {
   \begin{tikzpicture}[scale=.9,xscale=.29,yscale=-.5]
    \node (030) at (0,-1) [svertex] {};
    \node (120) at (-1,0) [svertex] {};
    \node (021) at (1,0) [svertex] {};
    \node (210) at (-2,1) [svertex] {};
    \node (111) at (0,1) [svertex] {};
    \node (012) at (2,1) [svertex] {};
    \node (300) at (-3,2) [svertex] {};
    \node (201) at (-1,2) [svertex] {};
    \node (102) at (1,2) [svertex] {};
    \node (003) at (3,2) [svertex] {};
    \draw [bgarrow] (120) -- (030);
    \draw [bgarrow] (030) -- (021);
    \draw [admset] (021) -- (120);
    \draw [bgarrow] (210) -- (120);
    \draw [bgarrow] (120) -- (111);
    \draw [bgarrow] (111) -- (021);
    \draw [admset] (021) -- (012);
    \draw [bgarrow] (012) -- (111);
    \draw [admset] (111) -- (210);
    \draw [admset] (300) -- (210);
    \draw [bgarrow] (210) -- (201);
    \draw [bgarrow] (201) -- (111);
    \draw [bgarrow] (111) -- (102);
    \draw [admset] (102) -- (012);
    \draw [bgarrow] (012) -- (003);
    \draw [bgarrow] (003) -- (102);
    \draw [admset] (102) -- (201);
    \draw [bgarrow] (201) -- (300);
   \end{tikzpicture}
  };
 \node (J2) at (1.5,9)
  {
   \begin{tikzpicture}[scale=.9,xscale=.29,yscale=-.5]
    \node (030) at (0,-1) [svertex] {};
    \node (120) at (-1,0) [svertex] {};
    \node (021) at (1,0) [svertex] {};
    \node (210) at (-2,1) [svertex] {};
    \node (111) at (0,1) [svertex] {};
    \node (012) at (2,1) [svertex] {};
    \node (300) at (-3,2) [svertex] {};
    \node (201) at (-1,2) [svertex] {};
    \node (102) at (1,2) [svertex] {};
    \node (003) at (3,2) [svertex] {};
    \draw [bgarrow] (120) -- (030);
    \draw [bgarrow] (030) -- (021);
    \draw [admset] (021) -- (120);
    \draw [bgarrow] (210) -- (120);
    \draw [bgarrow] (120) -- (111);
    \draw [bgarrow] (111) -- (021);
    \draw [bgarrow] (021) -- (012);
    \draw [admset] (012) -- (111);
    \draw [admset] (111) -- (210);
    \draw [bgarrow] (300) -- (210);
    \draw [bgarrow] (210) -- (201);
    \draw [bgarrow] (201) -- (111);
    \draw [bgarrow] (111) -- (102);
    \draw [bgarrow] (102) -- (012);
    \draw [admset] (012) -- (003);
    \draw [bgarrow] (003) -- (102);
    \draw [admset] (102) -- (201);
    \draw [admset] (201) -- (300);
   \end{tikzpicture}
  };
 \node (K1) at (0,10)
  {
   \begin{tikzpicture}[scale=.9,xscale=.29,yscale=-.5]
    \node (030) at (0,-1) [svertex] {};
    \node (120) at (-1,0) [svertex] {};
    \node (021) at (1,0) [svertex] {};
    \node (210) at (-2,1) [svertex] {};
    \node (111) at (0,1) [svertex] {};
    \node (012) at (2,1) [svertex] {};
    \node (300) at (-3,2) [svertex] {};
    \node (201) at (-1,2) [svertex] {};
    \node (102) at (1,2) [svertex] {};
    \node (003) at (3,2) [svertex] {};
    \draw [bgarrow] (120) -- (030);
    \draw [bgarrow] (030) -- (021);
    \draw [admset] (021) -- (120);
    \draw [admset] (210) -- (120);
    \draw [bgarrow] (120) -- (111);
    \draw [bgarrow] (111) -- (021);
    \draw [admset] (021) -- (012);
    \draw [bgarrow] (012) -- (111);
    \draw [bgarrow] (111) -- (210);
    \draw [bgarrow] (300) -- (210);
    \draw [admset] (210) -- (201);
    \draw [bgarrow] (201) -- (111);
    \draw [bgarrow] (111) -- (102);
    \draw [admset] (102) -- (012);
    \draw [bgarrow] (012) -- (003);
    \draw [bgarrow] (003) -- (102);
    \draw [admset] (102) -- (201);
    \draw [bgarrow] (201) -- (300);
   \end{tikzpicture}
  };
 \node (K2) at (1,10)
  {
   \begin{tikzpicture}[scale=.9,xscale=.29,yscale=-.5]
    \node (030) at (0,-1) [svertex] {};
    \node (120) at (-1,0) [svertex] {};
    \node (021) at (1,0) [svertex] {};
    \node (210) at (-2,1) [svertex] {};
    \node (111) at (0,1) [svertex] {};
    \node (012) at (2,1) [svertex] {};
    \node (300) at (-3,2) [svertex] {};
    \node (201) at (-1,2) [svertex] {};
    \node (102) at (1,2) [svertex] {};
    \node (003) at (3,2) [svertex] {};
    \draw [bgarrow] (120) -- (030);
    \draw [bgarrow] (030) -- (021);
    \draw [admset] (021) -- (120);
    \draw [bgarrow] (210) -- (120);
    \draw [bgarrow] (120) -- (111);
    \draw [bgarrow] (111) -- (021);
    \draw [bgarrow] (021) -- (012);
    \draw [admset] (012) -- (111);
    \draw [admset] (111) -- (210);
    \draw [admset] (300) -- (210);
    \draw [bgarrow] (210) -- (201);
    \draw [bgarrow] (201) -- (111);
    \draw [bgarrow] (111) -- (102);
    \draw [bgarrow] (102) -- (012);
    \draw [admset] (012) -- (003);
    \draw [bgarrow] (003) -- (102);
    \draw [admset] (102) -- (201);
    \draw [bgarrow] (201) -- (300);
   \end{tikzpicture}
  };
 \node (K3) at (2,10)
  {
   \begin{tikzpicture}[scale=.9,xscale=.29,yscale=-.5]
    \node (030) at (0,-1) [svertex] {};
    \node (120) at (-1,0) [svertex] {};
    \node (021) at (1,0) [svertex] {};
    \node (210) at (-2,1) [svertex] {};
    \node (111) at (0,1) [svertex] {};
    \node (012) at (2,1) [svertex] {};
    \node (300) at (-3,2) [svertex] {};
    \node (201) at (-1,2) [svertex] {};
    \node (102) at (1,2) [svertex] {};
    \node (003) at (3,2) [svertex] {};
    \draw [bgarrow] (120) -- (030);
    \draw [bgarrow] (030) -- (021);
    \draw [admset] (021) -- (120);
    \draw [bgarrow] (210) -- (120);
    \draw [bgarrow] (120) -- (111);
    \draw [bgarrow] (111) -- (021);
    \draw [bgarrow] (021) -- (012);
    \draw [admset] (012) -- (111);
    \draw [admset] (111) -- (210);
    \draw [bgarrow] (300) -- (210);
    \draw [bgarrow] (210) -- (201);
    \draw [bgarrow] (201) -- (111);
    \draw [bgarrow] (111) -- (102);
    \draw [bgarrow] (102) -- (012);
    \draw [bgarrow] (012) -- (003);
    \draw [admset] (003) -- (102);
    \draw [admset] (102) -- (201);
    \draw [admset] (201) -- (300);
   \end{tikzpicture}
  };
 \node (Eq1) at (-.75,0) [circle,draw=black,thick] {=};
 \node (Eq2) at (2.75,10) [circle,draw=black,thick] {=};
 \draw [dashed] (-1,0) -- (Eq1);
 \draw [dashed] (Eq1) -- (A1);
 \draw [dashed] (A1) -- (A2);
 \draw [dashed] (A2) -- (A3);
 \draw [dashed] (A3) -- (3,0);
 \draw [dashed] (-1,10) -- (K1);
 \draw [dashed] (K1) -- (K2);
 \draw [dashed] (K2) -- (K3);
 \draw [dashed] (K3) -- (Eq2);
 \draw [dashed] (Eq2) -- (3,10);
 \draw (A1) -- (B1);
 \draw (A2) -- (B1);
 \draw (A2) -- (B2);
 \draw (A3) -- (B2);
 \draw (B1) -- (C1);
 \draw (B2) -- (C1);
 \draw (B2) -- (C2);
 \draw (C1) -- (D1);
 \draw (C1) -- (D2);
 \draw (C2) -- (D1);
 \draw (C2) -- (D2);
 \draw (C2) -- (D3);
 \draw (D1) -- (E1);
 \draw (D1) -- (E2);
 \draw (D2) -- (E1);
 \draw (D2) -- (E2);
 \draw (D3) -- (E2);
 \draw (D3) -- (E3);
 \draw (E1) -- (F1);
 \draw (E2) -- (F1);
 \draw (E2) -- (F2);
 \draw (E3) -- (F2);
 \draw (F1) -- (G1);
 \draw (F1) -- (G2);
 \draw (F2) -- (G2);
 \draw (F2) -- (G3);
 \draw (G1) -- (H1);
 \draw (G2) -- (H1);
 \draw (G2) -- (H2);
 \draw (G2) -- (H3);
 \draw (G3) -- (H2);
 \draw (G3) -- (H3);
 \draw (H1) -- (I1);
 \draw (H2) -- (I1);
 \draw (H2) -- (I2);
 \draw (H3) -- (I1);
 \draw (H3) -- (I2);
 \draw (I1) -- (J1);
 \draw (I2) -- (J1);
 \draw (I2) -- (J2);
 \draw (J1) -- (K1);
 \draw (J1) -- (K2);
 \draw (J2) -- (K2);
 \draw (J2) -- (K3);
\end{tikzpicture} \]}

\newcommand{\iteratedAprLinearAThreeTwo}{
\[ \begin{tikzpicture}[xscale=4,yscale=-2]
 \node (A) at (0,0)
  {
   \begin{tikzpicture}[xscale=.5,yscale=-.3]
    \node (0) at (0,-1) [svertex] {};
    \node (1) at (-1,0) [svertex] {};
    \node (2) at (1,0) [svertex] {};
    \node (3) at (-2,1) [svertex] {};
    \node (4) at (0,1) [svertex] {};
    \node (5) at (2,1) [svertex] {};
    \node (6) at (0,3) [svertex] {};
    \node (7) at (-1,4) [svertex] {};
    \node (8) at (1,4) [svertex] {};
    \node (9) at (0,6) [svertex] {};
    \draw [bgarrow] (1) -- (0);
    \draw [bgarrow] (0) -- (2);
    \draw [bgarrow] (3) -- (1);
    \draw [bgarrow] (1) -- (4);
    \draw [bgarrow] (4) -- (2);
    \draw [bgarrow] (2) -- (5);
    \draw [bgarrow] (6) -- (1);
    \draw [admset] (2) -- (6);
    \draw [bgarrow] (7) -- (3);
    \draw [admset] (4) -- (7);
    \draw [bgarrow] (8) -- (4);
    \draw [admset] (5) -- (8);
    \draw [bgarrow] (7) -- (6);
    \draw [bgarrow] (6) -- (8);
    \draw [bgarrow] (9) -- (7);
    \draw [admset] (8) -- (9);
   \end{tikzpicture}
  };
 \node (B) at (1,0)
  {
   \begin{tikzpicture}[xscale=.5,yscale=-.3]
    \node (0) at (0,-1) [svertex] {};
    \node (1) at (-1,0) [svertex] {};
    \node (2) at (1,0) [svertex] {};
    \node (3) at (-2,1) [svertex] {};
    \node (4) at (0,1) [svertex] {};
    \node (5) at (2,1) [svertex] {};
    \node (6) at (0,3) [svertex] {};
    \node (7) at (-1,4) [svertex] {};
    \node (8) at (1,4) [svertex] {};
    \node (9) at (0,6) [svertex] {};
    \draw [bgarrow] (1) -- (0);
    \draw [bgarrow] (0) -- (2);
    \draw [bgarrow] (3) -- (1);
    \draw [bgarrow] (1) -- (4);
    \draw [bgarrow] (4) -- (2);
    \draw [admset] (2) -- (5);
    \draw [bgarrow] (6) -- (1);
    \draw [admset] (2) -- (6);
    \draw [bgarrow] (7) -- (3);
    \draw [admset] (4) -- (7);
    \draw [bgarrow] (8) -- (4);
    \draw [bgarrow] (5) -- (8);
    \draw [bgarrow] (7) -- (6);
    \draw [bgarrow] (6) -- (8);
    \draw [admset] (9) -- (7);
    \draw [bgarrow] (8) -- (9);
   \end{tikzpicture}
  };
 \node (C) at (.5,1)
  {
   \begin{tikzpicture}[xscale=.5,yscale=-.3]
    \node (0) at (0,-1) [svertex] {};
    \node (1) at (-1,0) [svertex] {};
    \node (2) at (1,0) [svertex] {};
    \node (3) at (-2,1) [svertex] {};
    \node (4) at (0,1) [svertex] {};
    \node (5) at (2,1) [svertex] {};
    \node (6) at (0,3) [svertex] {};
    \node (7) at (-1,4) [svertex] {};
    \node (8) at (1,4) [svertex] {};
    \node (9) at (0,6) [svertex] {};
    \draw [bgarrow] (1) -- (0);
    \draw [bgarrow] (0) -- (2);
    \draw [bgarrow] (3) -- (1);
    \draw [bgarrow] (1) -- (4);
    \draw [bgarrow] (4) -- (2);
    \draw [bgarrow] (2) -- (5);
    \draw [bgarrow] (6) -- (1);
    \draw [admset] (2) -- (6);
    \draw [bgarrow] (7) -- (3);
    \draw [admset] (4) -- (7);
    \draw [bgarrow] (8) -- (4);
    \draw [admset] (5) -- (8);
    \draw [bgarrow] (7) -- (6);
    \draw [bgarrow] (6) -- (8);
    \draw [admset] (9) -- (7);
    \draw [bgarrow] (8) -- (9);
   \end{tikzpicture}
  };
 \node (D) at (1.5,1)
  {
   \begin{tikzpicture}[xscale=.5,yscale=-.3]
    \node (0) at (0,-1) [svertex] {};
    \node (1) at (-1,0) [svertex] {};
    \node (2) at (1,0) [svertex] {};
    \node (3) at (-2,1) [svertex] {};
    \node (4) at (0,1) [svertex] {};
    \node (5) at (2,1) [svertex] {};
    \node (6) at (0,3) [svertex] {};
    \node (7) at (-1,4) [svertex] {};
    \node (8) at (1,4) [svertex] {};
    \node (9) at (0,6) [svertex] {};
    \draw [bgarrow] (1) -- (0);
    \draw [bgarrow] (0) -- (2);
    \draw [bgarrow] (3) -- (1);
    \draw [bgarrow] (1) -- (4);
    \draw [bgarrow] (4) -- (2);
    \draw [admset] (2) -- (5);
    \draw [bgarrow] (6) -- (1);
    \draw [admset] (2) -- (6);
    \draw [admset] (7) -- (3);
    \draw [bgarrow] (4) -- (7);
    \draw [bgarrow] (8) -- (4);
    \draw [bgarrow] (5) -- (8);
    \draw [admset] (7) -- (6);
    \draw [bgarrow] (6) -- (8);
    \draw [bgarrow] (9) -- (7);
    \draw [bgarrow] (8) -- (9);
   \end{tikzpicture}
  };
 \node (E) at (1,2)
  {
   \begin{tikzpicture}[xscale=.5,yscale=-.3]
    \node (0) at (0,-1) [svertex] {};
    \node (1) at (-1,0) [svertex] {};
    \node (2) at (1,0) [svertex] {};
    \node (3) at (-2,1) [svertex] {};
    \node (4) at (0,1) [svertex] {};
    \node (5) at (2,1) [svertex] {};
    \node (6) at (0,3) [svertex] {};
    \node (7) at (-1,4) [svertex] {};
    \node (8) at (1,4) [svertex] {};
    \node (9) at (0,6) [svertex] {};
    \draw [bgarrow] (1) -- (0);
    \draw [bgarrow] (0) -- (2);
    \draw [bgarrow] (3) -- (1);
    \draw [bgarrow] (1) -- (4);
    \draw [bgarrow] (4) -- (2);
    \draw [bgarrow] (2) -- (5);
    \draw [bgarrow] (6) -- (1);
    \draw [admset] (2) -- (6);
    \draw [admset] (7) -- (3);
    \draw [bgarrow] (4) -- (7);
    \draw [bgarrow] (8) -- (4);
    \draw [admset] (5) -- (8);
    \draw [admset] (7) -- (6);
    \draw [bgarrow] (6) -- (8);
    \draw [bgarrow] (9) -- (7);
    \draw [bgarrow] (8) -- (9);
   \end{tikzpicture}
  };
 \node (F) at (2,2)
  {
   \begin{tikzpicture}[xscale=.5,yscale=-.3]
    \node (0) at (0,-1) [svertex] {};
    \node (1) at (-1,0) [svertex] {};
    \node (2) at (1,0) [svertex] {};
    \node (3) at (-2,1) [svertex] {};
    \node (4) at (0,1) [svertex] {};
    \node (5) at (2,1) [svertex] {};
    \node (6) at (0,3) [svertex] {};
    \node (7) at (-1,4) [svertex] {};
    \node (8) at (1,4) [svertex] {};
    \node (9) at (0,6) [svertex] {};
    \draw [bgarrow] (1) -- (0);
    \draw [bgarrow] (0) -- (2);
    \draw [bgarrow] (3) -- (1);
    \draw [bgarrow] (1) -- (4);
    \draw [bgarrow] (4) -- (2);
    \draw [admset] (2) -- (5);
    \draw [admset] (6) -- (1);
    \draw [bgarrow] (2) -- (6);
    \draw [admset] (7) -- (3);
    \draw [bgarrow] (4) -- (7);
    \draw [bgarrow] (8) -- (4);
    \draw [bgarrow] (5) -- (8);
    \draw [bgarrow] (7) -- (6);
    \draw [admset] (6) -- (8);
    \draw [bgarrow] (9) -- (7);
    \draw [bgarrow] (8) -- (9);
   \end{tikzpicture}
  };
 \node (G) at (.5,3)
  {
   \begin{tikzpicture}[xscale=.5,yscale=-.3]
    \node (0) at (0,-1) [svertex] {};
    \node (1) at (-1,0) [svertex] {};
    \node (2) at (1,0) [svertex] {};
    \node (3) at (-2,1) [svertex] {};
    \node (4) at (0,1) [svertex] {};
    \node (5) at (2,1) [svertex] {};
    \node (6) at (0,3) [svertex] {};
    \node (7) at (-1,4) [svertex] {};
    \node (8) at (1,4) [svertex] {};
    \node (9) at (0,6) [svertex] {};
    \draw [bgarrow] (1) -- (0);
    \draw [bgarrow] (0) -- (2);
    \draw [admset] (3) -- (1);
    \draw [bgarrow] (1) -- (4);
    \draw [bgarrow] (4) -- (2);
    \draw [bgarrow] (2) -- (5);
    \draw [bgarrow] (6) -- (1);
    \draw [admset] (2) -- (6);
    \draw [bgarrow] (7) -- (3);
    \draw [bgarrow] (4) -- (7);
    \draw [bgarrow] (8) -- (4);
    \draw [admset] (5) -- (8);
    \draw [admset] (7) -- (6);
    \draw [bgarrow] (6) -- (8);
    \draw [bgarrow] (9) -- (7);
    \draw [bgarrow] (8) -- (9);
   \end{tikzpicture}
  };
 \node (H) at (1.5,3)
  {
   \begin{tikzpicture}[xscale=.5,yscale=-.3]
    \node (0) at (0,-1) [svertex] {};
    \node (1) at (-1,0) [svertex] {};
    \node (2) at (1,0) [svertex] {};
    \node (3) at (-2,1) [svertex] {};
    \node (4) at (0,1) [svertex] {};
    \node (5) at (2,1) [svertex] {};
    \node (6) at (0,3) [svertex] {};
    \node (7) at (-1,4) [svertex] {};
    \node (8) at (1,4) [svertex] {};
    \node (9) at (0,6) [svertex] {};
    \draw [bgarrow] (1) -- (0);
    \draw [bgarrow] (0) -- (2);
    \draw [bgarrow] (3) -- (1);
    \draw [bgarrow] (1) -- (4);
    \draw [bgarrow] (4) -- (2);
    \draw [bgarrow] (2) -- (5);
    \draw [admset] (6) -- (1);
    \draw [bgarrow] (2) -- (6);
    \draw [admset] (7) -- (3);
    \draw [bgarrow] (4) -- (7);
    \draw [bgarrow] (8) -- (4);
    \draw [admset] (5) -- (8);
    \draw [bgarrow] (7) -- (6);
    \draw [admset] (6) -- (8);
    \draw [bgarrow] (9) -- (7);
    \draw [bgarrow] (8) -- (9);
   \end{tikzpicture}
  };
 \node (I) at (1,4)
  {
   \begin{tikzpicture}[xscale=.5,yscale=-.3]
    \node (0) at (0,-1) [svertex] {};
    \node (1) at (-1,0) [svertex] {};
    \node (2) at (1,0) [svertex] {};
    \node (3) at (-2,1) [svertex] {};
    \node (4) at (0,1) [svertex] {};
    \node (5) at (2,1) [svertex] {};
    \node (6) at (0,3) [svertex] {};
    \node (7) at (-1,4) [svertex] {};
    \node (8) at (1,4) [svertex] {};
    \node (9) at (0,6) [svertex] {};
    \draw [bgarrow] (1) -- (0);
    \draw [bgarrow] (0) -- (2);
    \draw [admset] (3) -- (1);
    \draw [bgarrow] (1) -- (4);
    \draw [bgarrow] (4) -- (2);
    \draw [bgarrow] (2) -- (5);
    \draw [admset] (6) -- (1);
    \draw [bgarrow] (2) -- (6);
    \draw [bgarrow] (7) -- (3);
    \draw [bgarrow] (4) -- (7);
    \draw [bgarrow] (8) -- (4);
    \draw [admset] (5) -- (8);
    \draw [bgarrow] (7) -- (6);
    \draw [admset] (6) -- (8);
    \draw [bgarrow] (9) -- (7);
    \draw [bgarrow] (8) -- (9);
   \end{tikzpicture}
  };
 \node (J) at (2,4)
  {
   \begin{tikzpicture}[xscale=.5,yscale=-.3]
    \node (0) at (0,-1) [svertex] {};
    \node (1) at (-1,0) [svertex] {};
    \node (2) at (1,0) [svertex] {};
    \node (3) at (-2,1) [svertex] {};
    \node (4) at (0,1) [svertex] {};
    \node (5) at (2,1) [svertex] {};
    \node (6) at (0,3) [svertex] {};
    \node (7) at (-1,4) [svertex] {};
    \node (8) at (1,4) [svertex] {};
    \node (9) at (0,6) [svertex] {};
    \draw [bgarrow] (1) -- (0);
    \draw [bgarrow] (0) -- (2);
    \draw [bgarrow] (3) -- (1);
    \draw [bgarrow] (1) -- (4);
    \draw [bgarrow] (4) -- (2);
    \draw [bgarrow] (2) -- (5);
    \draw [admset] (6) -- (1);
    \draw [bgarrow] (2) -- (6);
    \draw [admset] (7) -- (3);
    \draw [bgarrow] (4) -- (7);
    \draw [admset] (8) -- (4);
    \draw [bgarrow] (5) -- (8);
    \draw [bgarrow] (7) -- (6);
    \draw [bgarrow] (6) -- (8);
    \draw [bgarrow] (9) -- (7);
    \draw [admset] (8) -- (9);
   \end{tikzpicture}
  };
 \node (K) at (1.5,5)
  {
   \begin{tikzpicture}[xscale=.5,yscale=-.3]
    \node (0) at (0,-1) [svertex] {};
    \node (1) at (-1,0) [svertex] {};
    \node (2) at (1,0) [svertex] {};
    \node (3) at (-2,1) [svertex] {};
    \node (4) at (0,1) [svertex] {};
    \node (5) at (2,1) [svertex] {};
    \node (6) at (0,3) [svertex] {};
    \node (7) at (-1,4) [svertex] {};
    \node (8) at (1,4) [svertex] {};
    \node (9) at (0,6) [svertex] {};
    \draw [bgarrow] (1) -- (0);
    \draw [bgarrow] (0) -- (2);
    \draw [admset] (3) -- (1);
    \draw [bgarrow] (1) -- (4);
    \draw [bgarrow] (4) -- (2);
    \draw [bgarrow] (2) -- (5);
    \draw [admset] (6) -- (1);
    \draw [bgarrow] (2) -- (6);
    \draw [bgarrow] (7) -- (3);
    \draw [bgarrow] (4) -- (7);
    \draw [admset] (8) -- (4);
    \draw [bgarrow] (5) -- (8);
    \draw [bgarrow] (7) -- (6);
    \draw [bgarrow] (6) -- (8);
    \draw [bgarrow] (9) -- (7);
    \draw [admset] (8) -- (9);
   \end{tikzpicture}
  };
 \node (L) at (2.5,5)
  {
   \begin{tikzpicture}[xscale=.5,yscale=-.3]
    \node (0) at (0,-1) [svertex] {};
    \node (1) at (-1,0) [svertex] {};
    \node (2) at (1,0) [svertex] {};
    \node (3) at (-2,1) [svertex] {};
    \node (4) at (0,1) [svertex] {};
    \node (5) at (2,1) [svertex] {};
    \node (6) at (0,3) [svertex] {};
    \node (7) at (-1,4) [svertex] {};
    \node (8) at (1,4) [svertex] {};
    \node (9) at (0,6) [svertex] {};
    \draw [bgarrow] (1) -- (0);
    \draw [bgarrow] (0) -- (2);
    \draw [bgarrow] (3) -- (1);
    \draw [bgarrow] (1) -- (4);
    \draw [bgarrow] (4) -- (2);
    \draw [bgarrow] (2) -- (5);
    \draw [admset] (6) -- (1);
    \draw [bgarrow] (2) -- (6);
    \draw [admset] (7) -- (3);
    \draw [bgarrow] (4) -- (7);
    \draw [admset] (8) -- (4);
    \draw [bgarrow] (5) -- (8);
    \draw [bgarrow] (7) -- (6);
    \draw [bgarrow] (6) -- (8);
    \draw [admset] (9) -- (7);
    \draw [bgarrow] (8) -- (9);
   \end{tikzpicture}
  };
 \node (Eq1) at (2,0) [circle,draw=black,thick] {=};
 \node (Eq2) at (.5,5) [circle,draw=black,thick] {=};
 \draw [dashed] (-.5,0) -- (A);
 \draw [dashed] (A) -- (B);
 \draw [dashed] (B) -- (Eq1);
 \draw [dashed] (Eq1) -- (2.5,0);
 \draw [dashed] (0,5) -- (Eq2);
 \draw [dashed] (Eq2) -- (K);
 \draw [dashed] (K) -- (L);
 \draw [dashed] (L) -- (3,5);
 \draw (A) -- (C);
 \draw (B) -- (C);
 \draw (B) -- (D);
 \draw (C) -- (E);
 \draw (D) -- (E);
 \draw (D) -- (F);
 \draw (E) -- (G);
 \draw (E) -- (H);
 \draw (F) -- (H);
 \draw (G) -- (I);
 \draw (H) -- (I);
 \draw (H) -- (J);
 \draw (I) -- (K);
 \draw (J) -- (K);
 \draw (J) -- (L);
\end{tikzpicture} \]}

\title{$n$-representation-finite algebras and $n$-APR tilting}
\author{Osamu Iyama \and Steffen Oppermann}

\address{Osamu Iyama: Graduate School of Mathematics, Nagoya University, Chikusa-ku, Nagoya,
464-8602 Japan}
\email{iyama@math.nagoya-u.ac.jp}

\address{Steffen Oppermann: Institutt for matematiske fag, NTNU, 7491 Trondheim, Norway}
\email{steffen.oppermann@math.ntnu.no}

\thanks{The first author was supported by JSPS Grant-in-Aid for Scientific Research 21740010}
\thanks{The second author was supported by NFR Storforsk grant no.\ 167130.}

\begin{document}

\begin{abstract}
We introduce the notion of $n$-representation-finiteness, generalizing representation-finite hereditary algebras. We establish the procedure of $n$-APR tilting, and show that it preserves $n$-representation-finiteness. We give some combinatorial description of this procedure, and use this to completely describe a class of $n$-representation-finite algebras called ``type A''.
\end{abstract}

\maketitle
\tableofcontents

\section{Introduction}

One of the highlights in representation theory of algebras is given by representation-finite algebras, which provide a prototype of the use of functorial methods in representation theory.
In 1971, Auslander gave a one-to-one correspondence between representation-finite algebras and Auslander algebras, which was a milestone in modern representation theory leading to later Auslander-Reiten theory. Many categorical properties of module categories can be understood as analogues of homological properties of
Auslander algebras, and vice versa.

To study higher Auslander algebras, the notion of $n$-cluster tilting subcategories (=maximal $(n-1)$-orthogonal subcategories) was introduced in \cite{I2}, and a higher analogue of Auslander-Reiten theory was developed in a series of papers \cite{Iy_n-Auslander,Iy_Auslander_corr,IO2}, see also the survey paper \cite{Iy_ICRA}. Recent results (in particular \cite{Iy_n-Auslander}, but also this paper and \cite{HerI,HuangZhang1,HuangZhang2,MR2512629,IO2}) suggest, that $n$-cluster tilting modules behave very nicely if the algebra has global dimension $n$. In this paper, we call such algebras \emph{$n$-representation-finite} and study them from the viewpoint of APR (=Auslander-Platzeck-Reiten) tilting theory (see \cite{APR}).

For the case $n=1$, $1$-representation-finite algebras are representation-finite hereditary algebras. In the representation theory of path algebras, the notion of Bernstein-Gelfand-Ponomarev reflection functors play an important role. Nowadays they are formulated in terms of APR tilting modules from a functorial viewpoint (see \cite{APR}). A main property is that the class of representation-finite hereditary algebras is closed under taking endomorphism algebras of APR tilting modules. By iterating the APR tilting process, we get a family of path algebras with the same underlying graph with different orientations

We follow this idea to construct from one given $n$-representation-finite algebra a family of $n$-representation-finite algebras. We introduce the general notion of $n$-APR tilting modules, which are explicitly constructed tilting modules associated with simple projective modules. The difference from the case $n=1$ is that we need a certain vanishing condition of extension groups, but this is always satisfied if $\Lambda$ is $n$-representation-finite.

In Section~\ref{section.APR} we introduce $n$-APR tilting. We first introduce $n$-APR tilting modules. We give descriptions of the $n$-APR tilted algebra in terms of one-point (co)extensions (see Subsection~\ref{subsec.asonepointext}, in particular Theorem~\ref{theorem.APR_on_onepointext}), and for $n=2$ also in terms of quivers with relations (see Subsection~\ref{subsec.quivers}, in particular Theorem~\ref{theorem.quiver_2APR}). Finally we introduce $n$-APR tilting in derived categories.

In Section~\ref{section.APR_repfin} we apply $n$-APR tilts to $n$-representation-finite algebras. The first main result is that $n$-APR tilting preserves $n$-representation-finiteness (Theorems~\ref{main} and \ref{theorem.APRpres_repfin}). In Subsections~\ref{subsect.slices} and \ref{subsect.admsets} we introduce the notions of slices and admissible sets in order to gain a better understanding of what algebras are iterated $n$-APR tilts of a given $n$-representation-finite algebra. More precisely we show that the iterated $n$-APR tilts are precisely the quotients of an explicitly constructed algebra by admissible sets (Theorem~\ref{theorem.adm=iteratedAPR}).

As an application of our general $n$-APR tilting theory, in Section~\ref{section.typeA}, we give a family of $n$-representation-finite algebras by an explicit quivers with relations, which are iterated $n$-APR tilts of higher Auslander algebras given in \cite{Iy_n-Auslander}. We call them $n$-representation-finite algebras of \emph{type $A$}, since, for the case $n=1$, they are path algebras of type $A_s$ with arbitrary orientation. As shown in Section~\ref{section.APR_repfin} in general, they form a family indexed by admissible sets. In contrast to the general setup, for type $A$ we have a very simple combinatorial description of admissible sets (we call sets satisfying this description cuts until we can show that they coincide with admissible sets -- see Definition~\ref{def.ca-set} and Remark~\ref{rem.cset=set}). Then the $n$-APR tiling process can be written purely combinatorially in terms of `mutation' of admissible sets, and we can give a purely combinatorial proof of the fact that all admissible sets are transitive under successive mutation.

Summing up with results in \cite{IO2}, we obtain self-injective weakly $(n+1)$-representation-finite algebras as $(n+1)$-preprojective algebras of $n$-representation-finite algebras of type $A$. This is a generalization of a result of Geiss, Leclerc, and Schr{\"o}er \cite{GLS1}, saying that preprojective algebras of type $A$ are weakly $2$-representation-finite.

\subsection*{Acknowledgments} This work started when the first author visited K{\"o}ln in 2007 and continued when he visited Trondheim in spring 2008 and 2009. He would like to thank the people in K{\"o}ln and Trondheim for their hospitality.

\section{Background and notation}

Throughout this paper we assume $\Lambda$ to be a finite dimensional algebra over some field $k$. We denote by $\mod \Lambda$ the category of finite dimensional $\Lambda$-modules (all modules are left modules).

\subsection{$n$-representation-finiteness}

\begin{definition}[see \cite{Iy_n-Auslander}]
A module $M \in \mod \Lambda$ is called \emph{$n$-cluster tilting object} if
\begin{align*}
\add M & = \{ X \in \mod \Lambda \mid \Ext^i_{\Lambda}(M,X) = 0 \, \forall i \in \{1, \ldots, n-1\} \} \\
& = \{ X \in \mod \Lambda \mid \Ext^i_{\Lambda}(X,M) = 0 \, \forall i \in \{1, \ldots, n-1\} \}.
\end{align*}
\end{definition}

Clearly such an $M$ is a generator-cogenerator, and \emph{$n$-rigid} in the sense that $\Ext^i_{\Lambda}(M,M) = 0$ $\forall i \in \{1, \ldots, n-1\}$.

Note that a $1$-cluster tilting object is just an additive generator of the module category.

\begin{definition}
Let $\Lambda$ be a finite dimensional algebra. We say $\Lambda$ is \emph{weakly $n$-representation-finite} if there exists an $n$-cluster tilting object in $\mod\Lambda$. If moreover $\gld \Lambda \leq n$ we say that $\Lambda$ is \emph{$n$-representation-finite}.
\end{definition}

The main aim of this paper is to understand better $n$-representation-finite algebras, and to construct larger families of examples.

For $n \geq 1$ we define the following functors:
\begin{align*}
\Tr_n & := \Tr \Omega^{n-1} \colon \stabmod \Lambda \arrow[30]{<->} \stabmod \Lambda^{\op},\\
\tau_n &:= D\Tr_n \colon \stabmod \Lambda \to[30] \costmod \Lambda,\\
\tau_n^- & := \Tr_n D \colon \costmod \Lambda \to[30] \stabmod \Lambda.
\end{align*}
(See \cite{ARS} for definitions and properties of the functors $\Tr$, $D$, and $\tau_1$.)

\begin{proposition}[\cite{I2}] \label{prop.cluster}
Let $M$ be an $n$-cluster tilting object in $\mod\Lambda$.
\begin{itemize}
\item We have an equivalence $\tau_n \colon \stabadd M \to \costadd M$ with a quasi-inverse $\tau_n^- \colon \costadd M \to \stabadd M$.
\item We have functorial isomorphisms $\underline{\Hom}_\Lambda(\tau_n^-Y,X)\iso D\Ext^n_\Lambda(X,Y)\iso\overline{\Hom}_\Lambda(Y,\tau_nX)$
for any $X,Y\in\add M$.
\item If $\gld \Lambda \leq n$ then $\add M = \add \{ \tau_n^{-i} \Lambda \mid i \in \mathbb{N}\} = \add \{ \tau_n^i D\Lambda \mid i \in \mathbb{N}\}$.
\end{itemize}
\end{proposition}

We have the following criterion for $n$-representation-finiteness:

\begin{proposition}[{\cite[Theorem~5.1(3)]{I2}}] \label{cluster criterion}
Let $\Lambda$ be a finite dimensional algebra and $n\ge1$. Let $M$ be an $n$-rigid generator-cogenerator. The following conditions are equivalent.
\begin{enumerate}
\item $M$ is an $n$-cluster tilting object in $\mod\Lambda$.
\item $\gld \End_{\Lambda}(M) \leq n+1$.
\item For any indecomposable object $X\in\add M$, there exists an exact sequence
\[0 \to[30] M_n \to[30] \cdots \to[30] M_0 \tol[30]{f} X\]
with $M_i\in\add M$ and a right almost split map $f$ in $\add M$.
\end{enumerate}
\end{proposition}

\subsection{Derived categories and $n$-cluster tilting} \label{sect.backgr.derived}

Let $\Lambda$ be a finite dimensional algebra of finite global dimension. We denote by
\[ \mathcal{D}_{\Lambda} := \Db(\mod \Lambda) \]
the bounded derived category of $\mod \Lambda$. We denote by
\[ \nu := D\Lambda \otimes_{\Lambda}^{\mathbf{L}} - \iso D \mathbf{R} \Hom_{\Lambda}(-, \Lambda)\colon \mathcal{D}_{\Lambda} \to[30] \mathcal{D}_{\Lambda} \]
the Nakayama-functor in $\mathcal{D}_{\Lambda}$. Clearly $\nu$ restricts to the usual Nakayama functor
\[ \nu \colon \add \Lambda \to[30] \add D \Lambda.\]
We denote by $\nu_n$ the $n$-th desuspension of $\nu$, that is $\nu_n = \nu[-n]$.

Note that if $\gld \Lambda \leq n$ then $\tau_n^{\pm} = H^0(\nu_n^{\pm} -)$.

We set
\[ \mathcal{U} = \mathcal{U}_{\Lambda}^n := \add \{ \nu_n^i \Lambda \mid i \in \Z \} \subseteq \mathcal{D}_{\Lambda}. \]

\begin{theorem}[{\cite[Theorem~1.22]{Iy_n-Auslander}}] \label{theorem.U_is_ct}
Let $\Lambda$ be an algebra of $\gld \Lambda \leq n$, such that $\tau_n^{-i} \Lambda = 0$ for sufficiently large $i$. Then the category $\mathcal{U}$ is an $n$-cluster tilting subcategory of $\mathcal{D}_{\Lambda}$.

In particular, if $\Lambda$ is $n$-representation-finite, then $\mathcal{U}$ is $n$-cluster tilting.
\end{theorem}

We have the following criterion for $n$-representation-finiteness in terms of the derived category:

\begin{theorem}[{\cite[Theorem~3.1]{IO2}}] \label{theorem.repfin_U}
Let $\Lambda$ be an algebra with $\gld \Lambda \leq n$. Then the following are equivalent.
\begin{enumerate}
\item $\Lambda$ is $n$-representation-finite,
\item $D \Lambda \in \mathcal{U}$,
\item $\nu \mathcal{U} = \mathcal{U}$.
\end{enumerate}
\end{theorem}

\subsection{$n$-Amiot-cluster categories and $(n+1)$-preprojective algebras}

\begin{definition}[see \cite{CC, C_PhD}] \label{def.amiot_CC}
We denote by $\mathcal{D}_{\Lambda} / \nu_n$ the \emph{orbit category}, that is $\Ob \mathcal{D}_{\Lambda} / \nu_n = \Ob \mathcal{D}_{\Lambda}$, and
\[ \Hom_{\mathcal{D}_{\Lambda} / \nu_n}(X, Y) = \bigoplus_{i \in \Z} \Hom_{\mathcal{D}_{\Lambda}}(X, \nu_n^i Y). \]
We denote by $\mathcal{C}_{\Lambda}^n$ the \emph{$n$-Amiot-cluster category}, that is the triangulated hull (see \cite{CC, C_PhD} -- we do not give a definition because for the purposes in this paper it does not matter if we think of the orbit category or the $n$-Amiot-cluster category). We denote by $\pi \colon \mathcal{D}_{\Lambda} \to \mathcal{C}_{\Lambda}^n$ the functor induced by projection onto the orbit category.
\end{definition}

\begin{lemma}[Amiot \cite{CC, C_PhD}]
Let $\Lambda$ be an algebra with $\gld \Lambda \leq n$. The $n$-Amiot-cluster category $\mathcal{C}_{\Lambda}^n$ is $\Hom$-finite if and only if $\tau_n^{-i} \Lambda = 0$ for sufficiently large $i$.

In particular it is $\Hom$-finite for any $n$-representation-finite algebra.
\end{lemma}

\begin{theorem}[Amiot \cite{CC, C_PhD}]
Let $\Lambda$ be an algebra with $\gld \Lambda \leq n$ such that $\mathcal{C}_{\Lambda}^n$ is $\Hom$-finite. Then $\pi \Lambda$ is an $n$-cluster tilting object in $\mathcal{C}_{\Lambda}^n$.
\end{theorem}

\begin{observation} \label{obs.U_cover}
Note that $\add \pi \Lambda$ is the image of $\mathcal{U}$ under the functor of the derived category to the $n$-Amiot-cluster category as indicated in the following diagram.
\[ \begin{tikzpicture}[xscale=3,yscale=-1.5]
 \node (A) at (0,0) {$\mathcal{U}$};
 \node (B) at (1,0) {$\add \pi(\Lambda)$};
 \node (C) at (0,1) {$\mathcal{D}_{\Lambda}$};
 \node (D) at (1,1) {$\mathcal{C}_{\Lambda}^n$};
 \draw [->>] (A) -- (B);
 \draw [->] (C) -- node [above] {$\pi$} (D);
 \draw [right hook->] (A) -- (C);
 \draw [right hook->] (B) -- (D);
\end{tikzpicture} \]
\end{observation}

\begin{definition} \label{def.preproj}
Let $\Lambda$ be an algebra with $\gld \Lambda \leq n$. The \emph{$(n+1)$-preprojective algebra} $\widehat{\Lambda}$ of $\Lambda$ is the tensor algebra of the bimodule $\Ext_{\Lambda}^n(D\Lambda, \Lambda)$ over $\Lambda$
\[ \widehat{\Lambda} := {\rm T}_{\Lambda} \Ext_{\Lambda}^n(D\Lambda, \Lambda). \]
(See \cite{Kel_DefCY} or \cite{Kel_ICRA} for a motivation for this name.)
\end{definition}

\begin{proposition} \label{proposition.tensoralg}
The $(n+1)$-preprojective algebra $\widehat{\Lambda}$ is isomorphic to the endomorphism ring
\[ \End_{\mathcal{D}_{\Lambda} / \nu_n}(\Lambda)  \iso \End_{\mathcal{C}_{\Lambda}^n}(\pi \Lambda) . \]
\end{proposition}

\begin{proof}
The proof of \cite[Proposition~5.2.1]{C_PhD} or \cite[Proposition~4.7]{CC} carries over.
\end{proof}

\section{$n$-APR tilting} \label{section.APR}

In this section we introduce $n$-APR tilting, and prove some general properties.

In Subsection~\ref{subsec.firstprops} we introduce the notion of (weak) $n$-APR tilting modules and study their basic properties.

In Subsection~\ref{subsec.asonepointext} we will give a concrete description of the $n$-APR tilted algebra in terms of one-point (co)extensions. Namely, if $\Lambda$ is a one-point coextension of $\End_{\Lambda}(Q)^{\op}$ by a module $M$, then the $n$-APR tilt is the one-point extension of $\End_{\Lambda}(Q)^{\op}$ by $\Tr_{n-1} M$. This result will allow us to give an explicit description of the quivers and relations in case $n=2$ in Subsection~\ref{subsec.quivers}. 

Finally, in Subsection~\ref{subsect.APR_der}, we introduce a version version of APR tilting in the language of derived categories.

\subsection{$n$-APR tilting modules} \label{subsec.firstprops}

\begin{definition} \label{def.nAPR}
Let $\Lambda$ be a basic finite dimensional algebra and $n\ge1$. Let $P$ be a simple projective $\Lambda$-module satisfying $\Ext^i_\Lambda(D\Lambda,P)=0$ for any $0 \leq i < n$. We decompose $\Lambda=P\oplus Q$.
We call
\[T:=(\tau_n^-P)\oplus Q\]
the \emph{weak $n$-APR tilting module} associated with $P$.
If moreover $\id P=n$, then we call $T$ an \emph{$n$-APR tilting module}, and we call $\End_{\Lambda}(T)^{\op}$ an \emph{$n$-APR tilt} of $\Lambda$.

Dually we define \emph{(weak) $n$-APR cotilting modules}.
\end{definition}

The more general notion of $n$-BB tilting modules has been introduced in \cite{HuXi}. \medskip

The following result shows that weak $n$-APR tilting modules are in fact tilting $\Lambda$-modules.

\begin{theorem}\label{APR}
Let $\Lambda$ be a basic finite dimensional algebra, and let $T$ be a weak $n$-APR tilting $\Lambda$-module (as in Definition~\ref{def.nAPR}). Then $T$ is a tilting $\Lambda$-module with $\pd_{\Lambda} T=n$.
\end{theorem}

We also have the following useful properties.

\begin{proposition}\label{APR2}
Let $T = (\tau_n^- P) \oplus Q$ be a weak $n$-APR tilting $\Lambda$-module.
\begin{enumerate}
\item $\Ext^i_\Lambda(T,\Lambda)=0$ for any $0<i<n$.
\item If moreover $T$ is $n$-APR tilting, then $\Hom_\Lambda(\tau_n^-P,\Lambda)=0$.
\end{enumerate}
\end{proposition}

For the proof of Theorem~\ref{APR} and Proposition~\ref{APR2}, we use the following observation on tilting mutation due to Riedtmann-Schofield \cite{RieScho}.

\begin{lemma}[Riedtmann-Schofield \cite{RieScho}] \label{RS}
Let $T$ be a $\Lambda$-module and $Y \monol[30]{g} T' \epil[30]{f} X$ an exact sequence with $T'\in\add T$. Then the following conditions are equivalent.
\begin{itemize}
\item $T\oplus X$ is a tilting $\Lambda$-module and $f$ is a right $(\add T)$-approximation.
\item $T\oplus Y$ is a tilting $\Lambda$-module and $g$ is a left $(\add T)$-approximation.
\end{itemize}
\end{lemma}

\begin{proof}[Proof of Theorem~\ref{APR} and Proposition~\ref{APR2}]
Take a minimal injective resolution
\begin{equation}\label{injective}
0 \to[30] P \to[30] I_0 \to[30] \cdots \to[30] I_{n-1} \tol[30]{g} I_n.
\end{equation}
Applying $D$, we have an exact sequence
\begin{equation}\label{projective}
DI_n \tol[30]{Dg} DI_{n-1} \to[30] \cdots \to[30] DI_0 \to[30] DP \to[30] 0.
\end{equation}
Applying the functor $(-)^*=\Hom_{\Lambda^{\rm op}}(-,\Lambda)$, to this projective resolution of $DP$ we obtain a complex
\[ 0 \to[30] (DI_0)^* \to[30] \cdots \to[30] (DI_n)^* \to[30] 0.\]
By definition the homology in its rightmost term is $\tau_n^- P$, and since $\Ext^i_{\Lambda}(D\Lambda,P) = 0$ for $0 \leq i < n$ all other homologies vanish. Since $(DI_0)^*$ is an indecomposable projective $\Lambda$-module with $\top(DI_0)^*=\Soc I_0=P$, we have $(DI_0)^*=P$. Thus we have an exact sequence
\begin{equation}\label{exchange}
0 \to[30] P \to[30] (DI_1)^* \to[30] \cdots \tol[30]{(Dg)^*} (DI_n)^* \tol[30]{f} \tau_n^-P \to[30] 0.
\end{equation}
Thus we have $\pd{}_\Lambda T=n$. Since $P$ is a simple projective $\Lambda$-module, we have $(DI_i)^*\in\add Q$ for $0<i\le n$.

Applying the functor $(-)^*$ to the sequence \eqref{exchange}, we have an exact sequence \eqref{projective}. Thus we have Proposition~\ref{APR2}(1).
If $\id P=n$, then $g$ in \eqref{injective} is surjective and $Dg$ in \eqref{projective} is injective. Since $(Dg)^{**} = Dg$ we have
\[ \Hom_{\Lambda}(\tau_n^- P, \Lambda) = (\tau_n^- P)^* = (\Cok (Dg)^*)^* = \Ker (Dg)^{**} = \Ker Dg = 0.\]
Thus we have Proposition~\ref{APR2}(2).

Note that we have a functorial isomorphism
\[ \Hom_{\Lambda}((DI_i)^*,-) \iso (DI_i) \otimes_{\Lambda}-.\]
Apply the functors $- \otimes_\Lambda Q$ and $\Hom_\Lambda(-,Q)$ to Sequences~\eqref{projective} and \eqref{exchange} respectively, the above isomorphism gives rise to to a commutative diagram
\[ \begin{tikzpicture}[xscale=2.5,yscale=-1]
 \node (A1) at (0,0) {$(DI_n)\otimes_\Lambda Q$};
 \node (A2) at (1,0) {$\cdots$};
 \node (A3) at (2,0) {$(DI_1)\otimes_\Lambda Q$};
 \node (A4) at (3.5,0) {$(DI_0)\otimes_\Lambda Q$};
 \node (A5) at (4.5,0) {$0$};
 \node (B1) at (0,1) {$\Hom_{\Lambda}((DI_n)^*, Q)$};
 \node (B2) at (1,1) {$\cdots$};
 \node (B3) at (2,1) {$\Hom_{\Lambda}((DI_1)^*, Q)$};
 \node (B4) at (3.5,1) {$\Hom_{\Lambda}((DI_0)^*, Q)$};
 \node (B5) at (4.5,1) {$0$};
 \node at (0,.5) [rotate=90] {$\iso$};
 \node at (2,.5) [rotate=90] {$\iso$};
 \node at (3.5,.5) [rotate=90] {$\iso$};
 \draw [->] (A1) -- (A2);
 \draw [->] (A2) -- (A3);
 \draw [->] (A3) -- (A4);
 \draw [->] (A4) -- (A5);
 \draw [->] (B1) -- (B2);
 \draw [->] (B2) -- (B3);
 \draw [->] (B3) -- (B4);
 \draw [->] (B4) -- (B5);
\end{tikzpicture} \]
of exact sequences.
Thus \eqref{exchange} is a left $(\add Q)$-approximation sequence of $P$, and we have that $T$ is a tilting $\Lambda$-module by using Lemma~\ref{RS} repeatedly.
\end{proof}

We recall the following result from tilting theory \cite{Ha}:
For a tilting $\Lambda$-module $T$ with $\Gamma:=\End_\Lambda(T)^{\op}$, we have functors
\begin{align*}
\mathbf{F} &:= \mathbf{R} \Hom_\Lambda(T,-) \colon \mathcal{D}_{\Lambda} \to \mathcal{D}_{\Gamma},\\
\mathbf{F}_i&:=\Ext^i_\Lambda(T,-) \colon \mod\Lambda\to\mod\Gamma\ \ \ (i\ge0).
\end{align*}
Put
\begin{align*}
\mathcal{F}_i & := \{X\in\mod\Lambda\ |\ \Ext^j_\Lambda(T,X)=0\ \mbox{ for any }\ j\neq i\},\\
\mathcal{X}_i & := \{Y\in\mod\Gamma\ |\ \Tor_j^\Gamma(Y,T)=0\ \mbox{ for any }\ j\neq i\}.
\end{align*}

\begin{lemma}[Happel \cite{Ha}]\label{tilting theory}
\begin{itemize}
\item $\mathbf{F}=\mathbf{R}\Hom_\Lambda(T,-) \colon \mathcal{D}_{\Lambda} \to\mathcal{D}_{\Gamma}$ is an equivalence.
\item For any $i\ge0$, we have an equivalence $\mathbf{F}_i:=\Ext^i_\Lambda(T,-) \colon \mathcal{F}_i\to\mathcal{X}_i$ which is isomorphic to the restriction of $[i]\circ\mathbf{F}$.
\item For any $X\in\mathcal{F}_0$, there exists an exact sequence
\[0 \to[30] T_m \to[30] \cdots \to[30] T_0 \to[30] X \to[30] 0\]
with $T_i\in\add T$ and $m\le\gld\Lambda$.
\end{itemize}
\end{lemma}

We now prove the following result, saying that the class of algebras of global dimension at most $n$ is closed under $n$-APR tilting.

\begin{proposition}\label{global dimension}
If $\gld\Lambda=n$, and $T$ is an $n$-APR tilting $\Lambda$-module, then $\gld\Gamma=n$ holds for $\Gamma:=\End_\Lambda(T)^{\op}$.
\end{proposition}

\begin{proof}
We only have to show that $\pd_\Gamma(\top\mathbf{F}_0X)\le n$ for any indecomposable $X\in\add T$.

(i) First we consider the case $X\in\add Q$.
Since $\gld\Lambda=n$, we can take a minimal projective resolution
\[0 \to[30] P_n \to[30] \cdots \to[30] P_1 \tol[30]{f} X \to[30] \top X \to[30] 0.\]
Since $\Hom_\Lambda(\tau_n^-P,\Lambda)=0$ by Proposition~\ref{APR2}(2), we have that any morphism $T\to X$ which is not a split epimorphism factors through $f$.

Applying $\Hom_\Lambda(T,-)$, we have an exact sequence
\[0 \to[30] \mathbf{F}_0P_n \to[30] \cdots \to[30] \mathbf{F}_0P_1 \tol[30]{\mathbf{F}_0f} \mathbf{F}_0X\]
since we have $\Ext^i_\Lambda(T,\Lambda)=0$ for any $0<i<n$ by Proposition~\ref{APR2}(1).
Moreover the above observation implies $\Cok\mathbf{F}_0f=\top\mathbf{F}_0X$. Thus we have $\pd{}_\Gamma(\top\mathbf{F}_0X)\le n$.

(ii) Next we consider the case $X=\tau_n^-P$. We will show that $\pd_{\Gamma} (\top \mathbf{F}_0 \tau_n^- P)$ is precisely $n$.
Applying $\mathbf{F}_0$ to the sequence \eqref{exchange} in the proof of Theorem~\ref{APR}, we have an exact sequence
\[0=\mathbf{F}_0P \to[30] \mathbf{F}_0(DI_1)^* \to[30] \cdots \to[30] \mathbf{F}_0(DI_n)^* \tol[30]{\mathbf{F}_0f} \mathbf{F}_0\tau_n^-P\]
since we have $\Ext^i_\Lambda(T,\Lambda)=0$ for any $0<i<n$ by Proposition~\ref{APR2}(1).

Since $Q$, $(DI_n)^*$, and $\tau_n^-P$ are in $\mathcal{F}_0$, we have a commutative diagram
\[\begin{tikzpicture}[yscale=-1,xscale=2.5]
 \node (A) at (0,0) {$\Hom_\Gamma(\mathbf{F}_0Q,\mathbf{F}_0(DI_n)^*)$};
 \node (B) at (2,0) {$\Hom_\Gamma(\mathbf{F}_0Q,\mathbf{F}_0\tau_n^-P)$};
 \node (C) at (0,1) {$\Hom_\Lambda(Q,(DI_n)^*)$};
 \node (D) at (2,1) {$\Hom_\Lambda(Q,\tau_n^-P)$};
 \node (E) at (3,1) {$0$};
 \node at (0,.5) [rotate=90] {$\iso$};
 \node at (2,.5) [rotate=90] {$\iso$};
 \draw [->] (A) -- node [above] {$\mathbf{F}_0 f$} (B);
 \draw [->] (C) -- node [above] {$f$} (D);
 \draw [->] (D) -- (E);
\end{tikzpicture}\]
where the lower sequence is exact since $Q$ is a projective $\Lambda$-module.
Since
\begin{align*}
\End_{\Gamma}(\mathbf{F}_0\tau_n^-P) & \iso \End_\Lambda(\tau_n^-P) \\
& = \underline{\End}_\Lambda(\tau_n^-P) && \text{(Proposition~\ref{APR2}(2))}\\
& \iso \overline{\End}_\Lambda(\Omega^{-(n-1)}P) && \text{(AR-translation)} \\
& \iso \overline{\End}_\Lambda(P) && \text{(since $\Ext_{\Lambda}^i(D \Lambda, P) = 0 \, \forall i \in \{1, \ldots, n-1\}$,} \\
&&& \qquad \text{see for instance \cite{AusBri})},
\end{align*}
any non-zero endomorphism of $\mathbf{F}_0\tau_n^-P$ is an automorphism.
Thus $\mathbf{F}_0f$ is a right almost split map in $\add\Gamma$, and we have $\Cok\mathbf{F}_0f=\top\mathbf{F}_0\tau_n^-P$ and $\pd_\Gamma(\top\mathbf{F}_0\tau_n^-P) = n$.
\end{proof}

Later we shall use the following observation.

\begin{lemma}\label{F_nP}
Under the circumstances in Theorem~\ref{APR}, we have the following.
\begin{enumerate}
\item $P\in\mathcal{F}_n$.
\item $\mathbf{F}_nP$ is a simple $\Gamma$-module. If $\id P=n$, then $\mathbf{F}_nP$ is an injective $\Gamma$-module.
\end{enumerate}
\end{lemma}

\begin{proof}
(1) Follows immediately from Proposition~\ref{APR2} and the fact that $P$ is simple.

(2) By AR-duality we have
\[\mathbf{F}_nP=\Ext^n_\Lambda(T,P)\iso\Ext^1_\Lambda(T,\Omega^{-(n-1)}P)\iso D\underline{\Hom}_\Lambda(\tau_n^-P,T).\]

First we show that $\mathbf{F}_nP$ is a simple $\Gamma$-module.
Since $\mathbf{F}_nP\iso D\underline{\Hom}_\Lambda(\tau_n^-P,T)=D\underline{\End}_\Lambda(\tau_n^-P)$, any composition factor of the $\Gamma$-module $\mathbf{F}_nP$ is isomorphic.
Thus we only have to show that $\End_\Gamma(\mathbf{F}_nP)$ is a division ring.
By Lemma \ref{tilting theory}, we have $\End_\Gamma(\mathbf{F}_nP)\iso\End_\Lambda(P)$. Thus the assertion follows.

Next we show the second assertion.
Since we have $\Hom_\Lambda(\tau_n^-P,\Lambda)=0$ by Proposition~\ref{APR2}, we have 
$\mathbf{F}_nP\iso D\underline{\Hom}_\Lambda(\tau_n^-P,T)=D\Hom_\Lambda(\tau_n^-P,T)$.
Thus $\mathbf{F}_nP$ is an injective $\Gamma$-module.
\end{proof}

\subsection{$n$-APR tilting as one-point extension} \label{subsec.asonepointext}

Let $\Lambda$ be a finite dimensional algebra, $M \in \mod \Lambda^{\op}$ and $N \in \mod\Lambda$. Slightly more general than ``classical'' one-point (co)extensions, we consider the algebras $\begin{smallpmatrix} K & M \\ & \Lambda \end{smallpmatrix}$ and $\begin{smallpmatrix} K & \\ N & \Lambda \end{smallpmatrix}$ if $K$ is a finite skew-field extension of our base field $k$, such that $K \subseteq \End_{\Lambda^{\op}}(M)$ and $K \subseteq \End_{\Lambda}(N)^{\op}$, respectively.

Now let $\Lambda$ be a basic algebra which has a simple projective module $P$. We set $K_P = \End_{\Lambda}(P)^{\op}$. Let $Q$ be the direct sum over the other indecomposable projective $\Lambda$-modules, that is $\Lambda = P \oplus Q$. We set $\Lambda_P := \End_{\Lambda}(Q)^{\op}$ and $M_P := \Hom_{\Lambda}(P, Q) \in \mod (K_P \otimes_k \Lambda_P^{\op})$. Then we have an isomorphism $\Lambda \iso \begin{smallpmatrix} K_P & M_P \\ & \Lambda_P \end{smallpmatrix}$ and $P$ is identified with the module $\begin{smallpmatrix} K_P \\ 0 \end{smallpmatrix}$.

\begin{theorem} \label{theorem.APR_on_onepointext}
Assume $\Lambda$ is a basic finite dimensional algebra with simple projective module $P$, and $n > 1$. Then the following are equivalent:
\begin{enumerate} \renewcommand{\labelenumi}{(\roman{enumi})}
\item $P$ gives rise to an $n$-APR tilting module,
\item $M_P$ has the following properties
\begin{itemize}
\item $\pd_{\Lambda_P^{\op}} M_P = n-1$,
\item $\Ext_{\Lambda_P^{\op}}^i(M_P, \Lambda_P) = 0$ for $0 \leq i \leq n-2$,
\item $\Ext_{\Lambda_P^{\op}}^i(M_P, M_P) = 0$ for $1 \leq i \leq n-2$, and
\item $\End_{\Lambda_P^{\op}}(M_P) = K_P$.
\end{itemize}
\end{enumerate}
Moreover, if the above conditions are satisfied and $\Gamma = \End_{\Lambda}((\tau_n^- P) \oplus Q)^{\op}$, then
\[ \Gamma \iso \begin{pmatrix} K_P \\ \Tr_{n-1} M_P & \Lambda_P \end{pmatrix} \]
\end{theorem}

\begin{remark}
The object $\Tr_{n-1} M_P$ is uniquely determined only up to projective summands. In this section we always understand $\Tr_{n-1} M_P$ to be constructed using a minimal projective resolution, or, equivalently, $\Tr_{n-1} M_P$ to not have any non-zero projective summands.
\end{remark}

\begin{proof}[Proof of Theorem~\ref{theorem.APR_on_onepointext}]
Let
\[0 \to[30] D M_P \to[30] I_0 \to[30] I_1 \to[30] \cdots\]
be an injective resolution of the $\Lambda_P$-module $D M_P$. Then the injective resolution of the $\Lambda$-module $P = \begin{smallpmatrix} K_P \\ 0 \end{smallpmatrix}$ is
\[ 0 \to[30] \begin{pmatrix} K_P \\ 0 \end{pmatrix} \to[30] \begin{pmatrix} K_P \\ D M_P \end{pmatrix} \to[30] \begin{pmatrix} 0 \\ I_0 \end{pmatrix} \to[30] \begin{pmatrix} 0 \\ I_1 \end{pmatrix} \to[30] \cdots. \]
Hence $\pd_{\Lambda_P^{\op}} M_P = \id_{\Lambda_P} D M_P = \id_{\Lambda} P - 1$. In particular we have $\id_{\Lambda} P = n \iff \pd_{\Lambda_P^{\op}} M_P = n-1$.

Moreover, for any $i \geq 1$ and any $I \in \inj \Lambda_P$ we have
\begin{align*}
\Ext^i_{\Lambda}(\begin{smallpmatrix} 0 \\ I \end{smallpmatrix}, P) & = \Ext^{i-1}_{\Lambda_P}(I, D M_P) \\
& = \Ext^{i-1}_{\Lambda_P^{\op}}(M_P, D I).
\end{align*}
(Note that the first equality also holds for $n=1$, since there are no non-zero maps from $\begin{smallpmatrix} 0 \\ I \end{smallpmatrix}$ to the injective $\Lambda$-module $\begin{smallpmatrix} K_P \\ D M_P \end{smallpmatrix}$.)

Finally we look at extensions between $P$ and the corresponding injective module. For $i > 1$ we have
\begin{align*}
\Ext^i_{\Lambda}(\begin{smallpmatrix} K_P \\ D M_P \end{smallpmatrix}, P) & = \Ext^{i-1}_{\Lambda}(\begin{smallpmatrix} K_P \\ D M_P \end{smallpmatrix}, \begin{smallpmatrix} 0 \\ D M_P \end{smallpmatrix}) \\
& = \Ext^{i-1}_{\Lambda_P}(D M_P, D M_P) \\
& = \Ext^{i-1}_{\Lambda_P^{\op}}(M_P, M_P).
\end{align*}
For $i = 1$ we obtain
\begin{align*}
\Ext^1_{\Lambda}(\begin{smallpmatrix} K_P \\ D M_P \end{smallpmatrix}, P) & = \Hom_{\Lambda}( \begin{smallpmatrix} K_P \\ D M_P \end{smallpmatrix}, \begin{smallpmatrix} 0 \\ D M_P \end{smallpmatrix}) /  \left( \End_{\Lambda}(\begin{smallpmatrix} K_P \\ D M_P \end{smallpmatrix}) \cdot [\begin{smallpmatrix} K_P \\ D M_P \end{smallpmatrix} \epi \begin{smallpmatrix} 0 \\ D M_P \end{smallpmatrix}]  \right) \\
& \iso \End_{\Lambda_P}(M_P) / K_P.
\end{align*}
This proves the equivalence of (i) and (ii).

For the second claim note that by Proposition~\ref{APR2}(2) we have $\Hom_{\Lambda}(\tau_n^- P, Q) = 0$. Therefore it only remains to verify $\Hom_{\Lambda}(Q, \tau_n^- P) = \Tr_{n-1} M_P$ and $\End_{\Lambda}(\tau_n^- P) = K_P$. This follows by looking at the injective resolution of $P$ above and applying $D$ to it to obtain (a projective resolution of) $\tau_n^- P$. 
\end{proof}

\subsection{Quivers for $2$-APR tilts} \label{subsec.quivers}

In this subsection we give an explicit desctription of $2$-APR tilts in terms of quivers with relations.

\begin{remark}
For comparison, recall the classical case ($n=1$): Assume $\Lambda = kQ / (R)$, and the set of relations $R$ is minimal ($\forall r \in R \colon r \not\in (R \setminus \{r\})$). Simple projective modules correspond to sources of $Q$. Let $P$ be a simple projective, $i \in Q_0$ the corresponding vertex. Then $\id P = 1 \iff$ no relation in $R$ involves a path starting in $i$. In this situation we have 
\[ \Lambda_P = k[Q \setminus \{i\}] / (R) \text{,} \qquad M_P = \bigoplus_{\substack{a \in Q_1 \\ \mathfrak{s}(a)=i}} P_j^* \text{,} \quad \text{and} \quad \Gamma = kQ'/(R), \]
 where $Q'$ is the quiver obtained from $Q$ by reversing all arrows starting in $i$.
\end{remark}

For $n=2$ we have to take into account the second cosyzygy of $P$, which corresponds to relations involving the corresponding vertex of the quiver. 

Let $\Lambda=kQ/(R)$ be a finite dimensional $k$-algebra presented by a quiver $Q=(Q_0,Q_1)$ with relations $R$ (which is assumed to be a minimal set of relations). Let $P$ be a simple projective $\Lambda$-module associated to a source $i$ of $Q$. We define a quiver $Q'=(Q'_0,Q'_1)$ with relations $R'$ as follows.
\begin{align*}
Q'_0 & = Q_0,\\
Q'_1 & = \{ a \in Q_1 \mid \mathfrak{s}(a) \neq i\} \coprod \{r^* \colon \mathfrak{e}(r) \to i \mid r \in R,\ \mathfrak{s}(r) = i \}
\end{align*}
where $r^*$ is a new arrow associated to each $r \in R$ with $\mathfrak{s}(r)=i$. We write $r \in R$ with $\mathfrak{s}(r)=i$ as
\[r=\sum_{\substack{a\in Q_1 \\ \mathfrak{s}(a)=i}} ar_a,\]
and define $a^* \in kQ'$ for each $a \in Q_1$ with $\mathfrak{s}(a) = i$ by
\[a^* := \sum_{\substack{r\in R \\ \mathfrak{s}(r)=i}} r_a r^* \in kQ'.\]
Now we define a set $R'$ of relations on $Q'$ by
\[R' = \{r\in R \mid \mathfrak{s}(r)\neq i\} \coprod \{a^* \colon \mathfrak{e}(a)\to i \mid a\in Q_1,\ \mathfrak{s}(a)=i\}.\]
 
\begin{theorem} \label{theorem.quiver_2APR}
Let $\Lambda = kQ / (R)$ and $P$ a simple projective $\Lambda$-module. Assume that $P$ gives rise to a 2-APR tilting $\Lambda$-module $T$. Then $\End_{\Lambda}(T)$ is isomorphic to $kQ'/(R')$ (with $Q'$ and $R'$ as explained above).
\end{theorem}

\begin{remark}
Roughly speaking, Theorem~\ref{theorem.quiver_2APR} means that arrows in $Q$ starting in $i$ become relations, and relations become arrows.
\end{remark}

Let us start with the following general observation.

\begin{observation} \label{obs.rel_onepoint}
Let $\Delta=kQ/(R)$ be a finite dimensional $k$-algebra 
presented by a quiver $Q$ with relations $R$.
Let $M$ be a $\Delta$-module with a projective presentation
\[\bigoplus_{1\le n\le N} P_{j_n} \tol[30]{(r_{n\ell})} \bigoplus_{1\le\ell\le L} P_{i_\ell}\to[30] M \to[30] 0\]
for $r_{n\ell}\in kQ$.
Then the one-point coextension algebra $\spm{k & \\ M & \Delta}$ is isomorphic to $k\widetilde{Q}/(\widetilde{R})$ for the quiver
$\widetilde{Q}=(\widetilde{Q}_0,\widetilde{Q}_1)$ with relations $\widetilde{R}$
defined by
\begin{align*}
\widetilde{Q}_0 & = Q_0\coprod\{i\},\\
\widetilde{Q}_1 & = Q_1\coprod\{a_\ell \colon i_\ell\to i \mid 1\le\ell\le L\},\\
\widetilde{R} & = R\coprod\{\sum_{1\le\ell\le L}r_{n\ell}a_\ell \mid 1\le n\le N\}.
\end{align*}
\end{observation}

Now we are ready to prove Theorem~\ref{theorem.quiver_2APR}.

\begin{proof}[Proof of Theorem~\ref{theorem.quiver_2APR}]
We can write $\Lambda = \spm{k & M_P \\ & \Lambda_P}$ as in Subsection~\ref{subsec.asonepointext}.
Let $Q_P$ be the quiver obtained from $Q$ by removing the vertex $i$,
and $R_P:=\{r\in R \mid \mathfrak{s}(r)\neq i\}$.
Then we have
\begin{equation}\label{1}
\Lambda_P \iso kQ_P / (R_P).
\end{equation}
By Theorem~\ref{theorem.APR_on_onepointext} we have
\begin{equation}\label{2}
\End_\Lambda(T) = \begin{pmatrix} \Lambda_P & \\ \Tr M_P & k \end{pmatrix}.
\end{equation}
Since we have a minimal projective resolution
\[ \bigoplus_{\substack{r\in R \\ \mathfrak{s}(r)=i}} P^*_{\mathfrak{e}(r)} \tol[30]{(r_a)} \bigoplus_{\substack{a\in Q_1 \\ \mathfrak{s}(a)=i}} P^*_{\mathfrak{e}(a)} \to[30] M_P \to[30] 0 \]
of the $\Lambda_P^{\op}$-module $M_P$,
we have a projective resolution
\begin{equation}\label{3}
\bigoplus_{\substack{a\in Q_1 \\ \mathfrak{s}(a)=i}} P_{\mathfrak{e}(a)} \tol[30]{(r_a)}
\bigoplus_{\substack{r\in R \\ \mathfrak{s}(r)=i}} P_{\mathfrak{e}(r)} \to[30]\Tr M_P\to[30] 0
\end{equation}
of the $\Lambda_P$-module $\Tr M_P$.
Applying Observation~\ref{obs.rel_onepoint} to the one-point coextension (\ref{2}),
we have the assertion from \eqref{1} and \eqref{3}.
\end{proof}

For example we could take $Q$ to be the Auslander Reiten quiver of $A_3$ and $R$ to be the mesh relations. Then $kQ / (R)$ is the Auslander algebra. See Tables~\ref{table.a3linear} (linear oriented $A_3$) and \ref{table.a3nonlin} (non-linear oriented $A_3$) for the iterated $2$-APR tilts of these Auslander algebras. In the pictures a downward line is a $2$-APR tilt. Vertices labeled $T$ are sources that have an associated $2$-APR tilt, vertices labeled $C$ are sinks having an associated $2$-APR cotilt. Sources and sinks that do not admit a $2$-APR tilt or cotilt are marked $X$.

\begin{table}
\iteratedAprLinearAThree
\caption{Iterated 2-APR tilts of the Auslander algebra of linear oriented $A_3$}  \label{table.a3linear}
\end{table}

\begin{table}
\iteratedAprNonlinearAThree
\caption{Iterated 2-APR tilts of the Auslander algebra of non-linear oriented $A_3$}  \label{table.a3nonlin}
\end{table}

Note that there are no $X$es occurring in Table~\ref{table.a3linear}. In fact, by \cite[Theorem~1.18]{Iy_n-Auslander} (see Theorem~\ref{theorem.auslanderalg}) the Auslander algebras of linear oriented $A_n$ are $2$-representation-finite, and hence every source and sink has an associated $2$-APR tilt and cotilt, respectively. We will more closely investigate $n$-APR tilts on $n$-representation-finite algebras in Section~\ref{section.APR_repfin}, and the particular algebras coming from linear oriented $A_n$ in Section~\ref{section.typeA}.

\subsection{$n$-APR tilting complexes} \label{subsect.APR_der}

As in Section~\ref{sect.backgr.derived}, throughout this section we assume $\Lambda$ to be a basic finite dimensional algebra of finite global dimension. We will constantly use the functors $\nu$ and $\nu_n$ introduced in the first paragraph of Section~\ref{sect.backgr.derived}.

\begin{definition} \label{def.nAPRcomplex}
Let $n\ge1$, and let $\Lambda = P \oplus Q$ be any direct summand decomposition of the $\Lambda$-module $\Lambda$ such that
\begin{enumerate}
\item $\Hom_{\Lambda}(Q,P) = 0$, and
\item $\Ext^i_\Lambda(\nu Q,P)=0$ for any $0 < i \neq n$,
\end{enumerate}
where $Q$ is such that $\Lambda=P\oplus Q$. Clearly (1) implies $\Hom_{\Lambda}(\nu Q, P) = 0$, so (2) also holds for $i = 0$.

We call
\[T:=(\nu_n^- P)\oplus Q\]
the \emph{$n$-APR tilting complex} associated with $P$.

By abuse of notation (see Remark~\ref{remark.in_mod=in_der} below for a justification), we also call $\End_{\mathcal{D}_{\Lambda}}(T)^{\op}$ an \emph{$n$-APR tilt} of $\Lambda$.
\end{definition}

\begin{remark} \label{remark.in_mod=in_der}
\begin{enumerate}
\item Any $n$-APR tilting module $(\tau_n^- P) \oplus Q$ in the sense of Definition~\ref{def.nAPR} is an $n$-APR tilting comples, since in that case $\nu_n^- P = \tau_n^- P$ holds (under the assumption that $\Lambda$ has finite global dimension).
\item Weak $n$-APR tilting modules are in general not $n$-APR tilting complexes.
\end{enumerate}
\end{remark}

\begin{remark}
In the setup of Definition~\ref{def.nAPRcomplex} there is no big difference between tilting and cotilting: the $n$-APR tilting complex $(\nu_n^- P) \oplus Q$ associated to $P$, and the $n$-APR cotilting complex $\nu P \oplus \nu_n \nu Q$ associated to (the injective module) $\nu Q$ are mapped to each other by the autoequivalence $\nu_n \nu$ of the derived category.
\end{remark}

In the rest of this subsection we will show that $n$-APR tilting complexes are indeed tilting complexes, and that they preserve $\gld \leq n$.

\begin{theorem} \label{thm.APR_comp_is_tilting}
Let $\Lambda$ be an algebra of finite global dimension, and $T$ an $n$-APR tilting complex (as in Definition~\ref{def.nAPRcomplex}). Then $T$ is a tilting complex in $\mathcal{D}_{\Lambda}$.
\end{theorem}

\begin{remark}
More generally, in Theorem~\ref{thm.APR_comp_is_tilting} it is possible to replace the assumption that $\Lambda$ has finite global dimension by the weaker assumption that $P$ has finite injective dimension. (In this case $\nu_n^- P = \mathbf{R} \Hom_{\Lambda}(D \Lambda, P)$ is still in $\mathcal{K}^{\rm b}(\proj \Lambda)$, the homotopy category of complexes of finitely generated projective $\Lambda$-modules.)
\end{remark}

\begin{proof}[Proof of Theorem~\ref{thm.APR_comp_is_tilting}]
We have to check that $T$ has no self-extensions, and that $T$ generates the derived category $\mathcal{D}_{\Lambda}$. We first check that $T$ has no self-extensions. Clearly for all $i \neq 0$ we have $\Hom_{\mathcal{D}_{\Lambda}}(\nu_n^- P, \nu_n^- P[i]) = 0$ and $\Hom_{\mathcal{D}_{\Lambda}}(Q,Q[i]) = 0$. Moreover
\begin{align*}
\Hom_{\mathcal{D}_{\Lambda}}(\nu_n^- P, Q[i]) & = \Hom_{\mathcal{D}_{\Lambda}}(\nu^- P, Q[i-n]) = D \Hom_{\mathcal{D}_{\Lambda}}(Q[i-n], P) \\
& = 0 \quad \forall i \in \mathbb{Z}.
\end{align*}
Finally $\Hom_{\mathcal{D}_{\Lambda}}(Q, \nu_n^- P[i]) = \Ext^{n+i}_{\Lambda}(\nu Q, P)$, which vanishes for $i \neq 0$ by assumption (2) of the definition.

Now we prove that $T$ generates $\mathcal{D}_{\Lambda}$. Let $X \in \mathcal{D}_{\Lambda}$ such that $\Hom_{\mathcal{D}_{\Lambda}}(\nu^- P[i], X) = 0$ and $\Hom_{\mathcal{D}_{\Lambda}}(Q[i], X) = 0$ for all $i$. By the latter property we see that the homology of $X$ does not contain any composition factors in $\add(\top Q)$. We can assume that $X$ is a complex
\[ \cdots \tol[30]{\scriptstyle d^{i-1}} X^i \tol[30]{\scriptstyle d^i} X^{i+1} \tol[30]{\scriptstyle d^{i+1}} \cdots \]
in $\Kb(\proj \Lambda)$, such that $\Im d^i \subseteq \Rad X^{i+1}$ for any $i$. 

Assume there is an $i$ such that $X^i \not\in \add P$. Let $i_M$ be the maximal $i$ with this property. Let $Q' \in \add Q$ be a non-zero summand of $X^{i_M}$. Since $X^{i_M+1} \in \add P$ by our choice of $i_M$, we have $\Hom_{\Lambda}(Q', X^{i_M+1}) \in \add \Hom_{\Lambda}(Q,P) = 0$ (see Definition~\ref{def.nAPRcomplex}(1)). Hence we have $Q' \subseteq \Ker d^{i_M}$. Since $\Im d^{i_M-1} \subseteq \Rad X^{i_M}$, we have $Q' \not\subseteq \Im d^{i_M-1}$, and hence $\Hom_{\mathcal{D}_{\Lambda}}(Q', X[i_M]) \neq 0$, a contradiction to our choice of $X$. Consequently, we have $X \in \Kb(\add P)$.

Now we assume $X \neq 0$. Let $i_N$ be the minimal $i$ such that $X^i \neq 0$. Since $X^{i_N} \in \add P$ we have $\Hom_{\mathcal{D}_{\Lambda}}(X[i_N], P) \neq 0$. This is a contradiction to our choice of $X$, since $\Hom_{\mathcal{D}_{\Lambda}}(X[i_N], P) = D \Hom_{\mathcal{D}_{\Lambda}}(\nu_n^- P, X[n+i_N])$.
\end{proof}

The following result generalizes Proposition~\ref{global dimension} to the setup af $n$-APR tilting complexes.

\begin{proposition} \label{proposition.gld_preserved}
If $\gld \Lambda \leq n$ and $T$ is an $n$-APR tilting complex in $\mathcal{D}_{\Lambda}$ then, for $\Gamma := \End_{\mathcal{D}_{\Lambda}}(T)^{\op}$ we have $\gld \Gamma \leq n$.
\end{proposition}

\begin{proof}
By \cite{Ric_Morita} the algebra $\Gamma$ has finite global dimension, and hence
\begin{align*}
\gld \Gamma & = \max \{i \mid \Ext^i_{\Gamma}(\nu \Gamma, \Gamma) \neq 0 \} \\
& = \max \{i \mid \Hom_{\mathcal{D}_{\Gamma}}(\nu \Gamma, \Gamma[i]) \neq 0 \} \\
& = \max \{i \mid \Hom_{\mathcal{D}_{\Lambda}}(\nu T, T[i]) \neq 0 \}.
\end{align*}
Clearly $\gld \Lambda \leq n$ implies that for $i > n$ we have $\Hom_{\mathcal{D}_{\Lambda}}(\nu \nu_n^- P, \nu_n^- P[i]) = \Ext^i_{\Lambda} (\nu P, P) = 0$, and $\Hom_{\mathcal{D}_{\Lambda}}(\nu Q, Q[i]) = \Ext^i_{\Lambda}(\nu Q, Q) = 0$. We have
\[ \Hom_{\mathcal{D}_{\Lambda}}(\nu \nu_n^- P, Q[i]) = \Hom_{\mathcal{D}_{\Lambda}}(P, Q[i-n]), \]
which is non-zero only for $i = n$. Finally
\[ \Hom_{\mathcal{D}_{\Lambda}}(\nu Q, \nu_n^- P[i]) = \Hom_{\mathcal{D}_{\Lambda}}(\nu^2 Q, P[n+i]). \]
Since $\nu Q \in \mod \Lambda$ and $\gld \Lambda \leq n$ it follows that $\nu^2 Q$ has non-zero homology only in degrees $-n, \ldots, 0$. Hence the above $\Hom$-space vanishes for $i > n$, since $\gld \Lambda \leq n$.

Summing up we obtain $\Hom_{\mathcal{D}_{\Lambda}}(\nu T, T[i]) = 0$ for $i > n$, which implies the claim of the theorem by the remark at the beginning of the proof.
\end{proof}

Recall the definition of the subcategory
\[ \mathcal{U}^n_{\Lambda} = \add \{ \nu_n^i \Lambda \mid i \in \Z \} \subseteq \mathcal{D}_{\Lambda} \]
given in Section~\ref{sect.backgr.derived}.

\begin{proposition} \label{prop.preserves_U}
Let $\Lambda$ be $n$-representation-finite, and $T$ an $n$-APR tilting complex in $\mathcal{D}_{\Lambda}$. Let $\Gamma := \End_{\mathcal{D}_{\Lambda}}(T)^{\op}$. Then the derived equivalence $\mathbf{R} \Hom_{\Lambda}(T, -) \colon \mathcal{D}_{\Lambda} \to \mathcal{D}_{\Gamma}$ (see \cite{Keller_handbook_tilting}) induces an equivalence $\mathcal{U}^n_{\Lambda} \to \mathcal{U}^n_{\Gamma}$.
\end{proposition}

\begin{proof}
This is clear since the derived equivalence $\mathbf{R} \Hom_{\Lambda}(T, -)$ commutes with $\nu_n$ and $T \in \mathcal{U}_{\Lambda}^n$.
\end{proof}

An application of Proposition~\ref{prop.preserves_U} we will use in Subsection~\ref{subsect.admsets} is the following.

\begin{proposition}
The $(n+1)$-preprojective algebra (see Definition~\ref{def.preproj}) is invariant under $n$-APR tilts.
\end{proposition}

\begin{proof}
By Propositions~\ref{prop.preserves_U} and \ref{proposition.tensoralg} we have
\[ \widehat{\Lambda} = \End_{\mathcal{U}^n_{\Lambda} / \nu_n}(\Lambda) = \End_{\mathcal{U}^n_{\Lambda} / \nu_n}(T) = \End_{\mathcal{U}^n_{\Gamma} / \nu_n}(\Gamma) = \widehat{\Gamma}. \qedhere \]
\end{proof}

\section{$n$-APR tilting for $n$-representation-finite algebras} \label{section.APR_repfin}

In this section we study the effect of $n$-APR tilts on $n$-representation-finite algebras.

The first main result is that $n$-APR tilting preserves $n$-representation-finiteness (Theorems~\ref{main} and \ref{theorem.APRpres_repfin}). We give two independant proofs for this fact. In Subsection~\ref{subsec.pres_mod} we study $n$-APR tilting modules for $n$-representation-finite algebras. We give an explicit description of a cluster tilting object in the new module category in terms of the cluster tilting object of the original algebra (Theorem~\ref{main}). In Subsection~\ref{subsec.APR_der_preserves} we give an independant proof (which is less explicit and relies heavily on a result from \cite{IO2}) that the more general procedure of tilting in $n$-APR tilting complexes also preserves $n$-representation-finiteness.

In Subsections~\ref{subsect.slices} and \ref{subsect.admsets} we introduce the notions of slices and admissible sets, which classify, for a given $n$-representation-finite algebra, all iterated $n$-APR tilts (see Theorem~\ref{theorem.adm=iteratedAPR}).

Throughout this Section, let $\Lambda$ be an $n$-representation-finite algebra. For simplicity of notation we assume $\Lambda$ to be basic.

\subsection{$n$-APR tilting modules preserve $n$-representation-finiteness} \label{subsec.pres_mod}

The following proposition shows that the setup of $n$-representation-finite algebras is particularly well-suited for looking at $n$-APR tilts.

\begin{observation} \label{obs.alwaysAPR}
\begin{enumerate}
\item Any simple projective and non-injective $\Lambda$-module $P$ admits an $n$-APR tilting $\Lambda$-module.
\item Any simple injective and non-projective $\Lambda$-module $I$ admits an $n$-APR cotilting $\Lambda$-module.
\end{enumerate}
\end{observation}

\begin{proof}
We have $\id P \leq n$ by $\gld \Lambda \leq n$. Since $M$ is an $n$-rigid generator-cogenerator, we have $\Ext^i_\Lambda(D\Lambda,P)=0$ for any $0<i<n$. This proves (1), the proof of (2) is dual.
\end{proof}

Throughout this subsection, we denote by $M$ the unique basic $n$-cluster tilting object in $\mod \Lambda$ (see the last point of Proposition~\ref{prop.cluster}). 

Now let $P$ be a simple projective and non-injective $\Lambda$-module. We decompose $\Lambda=P\oplus Q$. Since $P \in \add M$ we can also decompose $M=P\oplus M'$. By Observation~\ref{obs.alwaysAPR} we have an $n$-APR tilting $\Lambda$-module $T :=(\tau_n^-P)\oplus Q$.

\begin{theorem} \label{main}
Under the above circumstances, we have the following.
\begin{enumerate}
\item $T \in \add M$.
\item $\Gamma:=\End_\Lambda(T)^{\op}$ is an $n$-representation-finite algebra with $n$-cluster tilting object  $N:=\Hom_\Lambda(T,M')\oplus\Ext^n_\Lambda(T,P)$.
\end{enumerate}
\end{theorem}

Before we prove the theorem let us note the following immediate consequence.

\begin{corollary}
Any iterated $n$-APR tilt of an $n$-representation-finite algebra is $n$-representation-finite.
\end{corollary}

In the rest we shall show Theorem~\ref{main}.
The assertion (1) follows immediately from the first part of Proposition~\ref{prop.cluster}.

Proposition~\ref{global dimension} proves that $\gld \Gamma = n$ in Theorem~\ref{main}. We shall show that $N$ in Theorem~\ref{main}(2) is an $n$-cluster tilting object. We will use the subcategories $\mathcal{F}_i \subseteq \mod \Lambda$, and the functors $\mathbf{F}_i$ which were introduced in Section~\ref{subsec.firstprops} (see in particular Lemma~\ref{tilting theory}).

\begin{lemma}
$M'\in\mathcal{F}_0$.
\end{lemma}

\begin{proof}
By Theorem \ref{main}(1) we know that $T\in\add M$. Hence, since $M$ is an $n$-rigid $\Lambda$-module, we have $\Ext^i_\Lambda(T,M)=0$ for any $0<i<n$. 
Since $\gld\Lambda\le n$, we only have to check $\Ext^n_\Lambda(T,M')=0$. Of course, we have $\Ext^n_\Lambda(Q,M')=0$ since $Q$ is projective. By Proposition~\ref{prop.cluster}(2), we have $\Ext^n_\Lambda(\tau_n^-P,M')\iso\overline{\Hom}_\Lambda(M',P)$, and the latter $\Hom$-space vanishes since $P$ is simple projective.
\end{proof}

\begin{lemma} \label{lemma.Nrigid} 
$N=\mathbf{F}_0M'\oplus\mathbf{F}_nP$ is an $n$-rigid $\Gamma$-module.
\end{lemma}

\begin{proof}
We have $\Ext^i_\Gamma(-,\mathbf{F}_nP)=0$ for any $i>0$ since $\mathbf{F}_nP$ is injective (Lemma~\ref{F_nP}(2)). Since $M' \in \mathcal{F}_0$ and $P \in \mathcal{F}_n$ by (1) and Lemma~\ref{F_nP}(1) respectively, we can check the assertion as follows by using Lemma~\ref{tilting theory}:
\begin{align*}
\Ext^i_\Gamma(\mathbf{F}_0M',\mathbf{F}_0M')&=\Hom_{\mathcal{D}_{\Gamma}}(\mathbf{F} M',\mathbf{F} M'[i])\\
&\iso\Hom_{\mathcal{D}_{\Lambda}}(M',M'[i]) \\
& =\Ext^i_\Lambda(M',M'),\\
\Ext^i_\Gamma(\mathbf{F}_nP,\mathbf{F}_0M')&=\Hom_{\mathcal{D}_{\Gamma}}(\mathbf{F} P[n],\mathbf{F} M'[i])\\
&\iso \Hom_{\mathcal{D}_{\Lambda}}(P,M'[i-n]) \\
& = \Ext^{i-n}_\Lambda(P,M').
\end{align*}
For $0 < i < n$ both of the above vanish, since $M$ is $n$-rigid.
\end{proof}

We now complete the proof of Theorem~\ref{main}.

\begin{proof}[Proof of Theorem~\ref{main}(2)]
By Lemma~\ref{lemma.Nrigid} we know that $N$ is $n$-rigid, and hence we may apply Proposition~\ref{cluster criterion}. We will show that $N$ is $n$-cluster tilting by checking the third of the equivalent conditions in Proposition~\ref{cluster criterion}(3).

(i) First we consider $\mathbf{F}_nP$.
Take a minimal injective resolution
\[0 \to[30] P \to[30] I_0 \to[30] \cdots \to[30] I_n \to[30] 0.\]
By Proposition~\ref{APR2} we have $\Ext^i_\Lambda(T,P)=0$ for $0 \leq i<n$. Hence, applying $\Hom_\Lambda(T,-)$, we have an exact sequence
\[0 \to[30] \mathbf{F}_0I_0 \to[30] \cdots \to[30] \mathbf{F}_0I_n \tol[30]{f} \mathbf{F}_nP \to[30] 0\]
with $\mathbf{F}_0I_i\in\add N$. We shall show that $f$ is a right almost split map in $\add N$.

By Lemma \ref{tilting theory}, we have $\Ext^j_\Gamma(\mathbf{F}_0M',\mathbf{F}_0I_i)\iso\Ext^j_\Lambda(M',I_i)=0$ for any $i$ and any $j>0$. Using this, we see that the map 
\[ \Hom_\Gamma(\mathbf{F}_0M',\mathbf{F}_0I_n) \tol[30]{f} \Hom_\Gamma(\mathbf{F}_0M',\mathbf{F}_nP) \]
is surjective. Since $\mathbf{F}_nP$ is a simple injective $\Lambda$-module by Lemma~\ref{F_nP}, any non-zero endomorphism of $\mathbf{F}_nP$ is an automorphism.
Thus $f$ is a right almost split map in $\add N$.

(ii) Next we consider $\mathbf{F}_0X$ for any indecomposable $X\in\add M'$.
Since $M$ is an $n$-cluster tilting object in $\mod\Lambda$, we have an exact sequence
\[0 \to[30] M_n \to[30] \cdots \to[30] M_0 \tol[30]{f} X\]
with $M_i\in\add M$ and a right almost split map $f$ in $\add M$ by Proposition~\ref{cluster criterion}. Applying $\mathbf{F}_0$, we have an exact sequence
\[0 \to[30] \mathbf{F}_0M_n \to[30] \cdots \to[30] \mathbf{F}_0M_0 \tol[30]{\mathbf{F}_0f} \mathbf{F}_0X\]
since we have $\Ext^i_\Lambda(T,M)=0$ for any $0<i<n$.
Since $\mathbf{F}_0M_i\in\add N$, we only have to show that $\mathbf{F}_0f$ is a right almost split map in $\add N$.

Since $\mathbf{F}_nP$ is a simple injective $\Lambda$-module by Lemma \ref{F_nP}, there is no non-zero map from $\mathbf{F}_n P$ to $\mathbf{F}_0 X$. Thus we only have to show that any morphism $g \colon \mathbf{F}_0M'\to\mathbf{F}_0X$ which is not a split epimorphism factors through $f$.
By Lemma \ref{tilting theory}, we can put $g=\mathbf{F}_0h$ for some $h \colon M'\to X$ which is not a split epimorphism.
Since $h$ factors through $f$, we have that $g=\mathbf{F}_0h$ factors through $\mathbf{F}_0f$. Thus we have shown that $\mathbf{F}_0f$ is a right almost split map in $\add N$.
\end{proof}

\subsection{$n$-APR tilting complexes preserve $n$-representation-finiteness} \label{subsec.APR_der_preserves}

Similar to Observation~\ref{obs.alwaysAPR} we have the following result for $n$-representation finite algebras.

\begin{observation} \label{obs.alwaysAPR_der}
Let $\Lambda = P \oplus Q$ as $\Lambda$-modules, such that $\Hom_{\Lambda}(Q, P) = 0$. Then $P$ has an associated $n$-APR tilting complex.
\end{observation}

We have the following result.

\begin{theorem} \label{theorem.APRpres_repfin}
Let $\Lambda$ be $n$-representation-finite, and $T$ be an $n$-APR tilting complex in $\mathcal{D}_{\Lambda}$. Then $\End_{\mathcal{D}_{\Lambda}}(T)^{\op}$ is also $n$-representation-finite.
\end{theorem}

\begin{proof}
We set $\Gamma = \End_{\mathcal{D}_{\Lambda}}(T)^{\op}$. By Proposition~\ref{proposition.gld_preserved} we know that, since $\gld \Lambda \leq n$ we also have $\gld \Gamma \leq n$.

By Proposition~\ref{prop.preserves_U} we know that the derived equivalence $\mathcal{D}_{\Lambda} \to \mathcal{D}_{\Gamma}$ induces an equivalence $\mathcal{U}_{\Lambda} \to \mathcal{U}_{\Gamma}$. Hence, by Theorem~\ref{theorem.repfin_U} we have
\begin{align*}
\Lambda \text{ is } n \text{-representation finite} & \iff \nu \mathcal{U}_{\Lambda} = \mathcal{U}_{\Lambda} \\
& \iff \nu \mathcal{U}_{\Gamma} = \mathcal{U}_{\Gamma} \\
& \iff \Gamma \text{ is } n \text{-representation finite} \qedhere
\end{align*}
\end{proof}

\subsection{Slices} \label{subsect.slices}

In this subsection we introduce the notion of slices in the $n$-cluster tilting subcategory $\mathcal{U}$ (see Definition~\ref{def.slicegen}). The aim is to provide a bijection between these slices and the iterated $n$-APR tilting complexes of $\Lambda$ (Theorem~\ref{theorem.slices=itereatedAPR}). This will be done by introducing a notion of mutation of slices (Definition~\ref{def.slicemutation}), and by proving that this mutation coincides with $n$-APR tilts.

Throughout, let $\Lambda$ be an $n$-representation-finite algebra. We consider the $n$-cluster tilting subcategory $\mathcal{U} = \mathcal{U}^n_{\Lambda} \subseteq \mathcal{D}_{\Lambda}$ given in Section~\ref{sect.backgr.derived}.

\begin{definition} \label{def.slicegen}
An object $S \in \mathcal{U}$ is called $\emph{slice}$ if
\begin{enumerate}
\item for any indecomposable projective module $P$ there is exactly one $i$ such that $\nu_n^i P \in \add S$, and
\item $\add S$ is convex, that means any path (that is, sequence of non-zero maps) in $\ind \mathcal{U}$, which starts and ends in $\add S$, lies entirely in $\add S$.
\end{enumerate}
\end{definition}

The following two observations give us the slices we are interested in here.

\begin{observation} \label{obs.tiltisslice}
In the setup above, $\Lambda \in \mathcal{U}$ is a slice, since we have $\Hom_{\mathcal{D}_{\Lambda}}(\nu_n^i \Lambda, \nu_n^j \Lambda) = H^0(\nu_n^{j-i} \Lambda) = 0$ if $i < j$.

Similarly, any iterated $n$-APR tilting complex of $\Lambda$ is a slice in $\mathcal{U}$, by Theorem~\ref{theorem.APRpres_repfin} and Proposition~\ref{prop.preserves_U}.
\end{observation}

\begin{proposition} \label{proposition.homslices}
Let $S$ be a slice. Then $\Hom_{\mathcal{D}_{\Lambda}}(S, \nu_n^i S) = 0$ for any $i > 0$.
\end{proposition}

For the proof we will need the following observation:

\begin{lemma}
Assume $\Lambda$ is indecomposable (as a ring) and not semi-simple. For any indecomposable $X \in \mathcal{U}$ there is a path $\nu_n X \leadsto X$ in $\mathcal{U}$.
\end{lemma}

\begin{proof}
Assume first that $X$ is a non-projective $\Lambda$-module. By \cite[Theorem~2.2]{Iy_n-Auslander} there is an $n$-almost split sequence
\[ \nu_n X = \tau_n X \mono[30] X_{n-1} \to[30] \cdots \to[30] X_1 \epi[30] X \]
with $X_i \in \mathcal{U} \cap \mod \Lambda$. This sequence gives rise to the desired path $\nu_n X \leadsto X$ in $\mathcal{U}$.

Now let $X \in \mathcal{U}$ be arbitrary indecomposable. By \cite[Lemma 4.9]{IO2} there exists $i \in \mathbb{Z}$ such that $\nu^i X$ is a non-projective $\Lambda$-module. Then there exists a path $\nu_n\nu^iX \leadsto \nu^i X$ in $\mathcal{U}$. Since $\nu$ is an autoequivalence of $\mathcal{U}$ by Theorem~\ref{theorem.repfin_U}, we have a path $\nu_nX \leadsto X$ in $\mathcal{U}$.
\end{proof}

\begin{proof}[Proof of \ref{proposition.homslices}]
We may assume $\Lambda$ to be connected, and not semi-simple. Then, by the above lemma, for any indecomposable $S' \in \add S$ there is a path $\nu_n S' \leadsto S'$ in $\mathcal{U}$. Hence there are also a paths $\nu_n^i S' \leadsto S'$ for $i > 0$. If $\Hom_{\mathcal{D}_\Lambda}(S,\nu_n^i S') \neq 0$ for some $i>0$, then we have $\nu_n^i S' \in \add S$ by Definition~\ref{def.slicegen}(2), contradicting \ref{def.slicegen}(1).
\end{proof}

\begin{definition} \label{def.slicemutation}
Let $S$ be a slice, $S = S' \oplus S''$ a direct summand decomposition of $S$, such that $\Hom_{\mathcal{D}_{\Lambda}}(S'', S') = 0$. We set
\begin{align*}
\mu^+_{S'}(S) & = (\nu_n^- S') \oplus S'' \text{ and} \\
\mu^-_{S''}(S) & = S' \oplus (\nu_n S'').
\end{align*}
We call them \emph{mutations} of $S$.
\end{definition}

\begin{lemma}  \label{lemma.slicemutation}
In the setup of Definition~\ref{def.slicemutation}, $\mu^+_{S'}(S)$ and $\mu^-_{S''}(S)$ are slices again.
\end{lemma}

\begin{proof}
We restrict to the case of $\mu^+_{S'}(S)$. It is clear that it satisfies Condition~(1) of Definition~\ref{def.slicegen}. To see that $\mu^+_{S'}(S)$ is convex, let $p$ be a path in $\ind \mathcal{U}$ starting and ending in $\mu^+_{S'}(S)$. We have the following four cases with respect to where $p$ starts and ends:
\begin{enumerate}
\item If $p$ starts and ends in $S''$ then it lies entirely in $S$. Since $\Hom_{\mathcal{D}_{\Lambda}}(S'', S') = 0$ it lies entirely in $S''$.
\item Similarly, if $p$ starts and ends in $\nu_n^- S'$ then it lies entirely in $\nu_n^- S'$.
\item By Proposition~\ref{proposition.homslices} $p$ cannot start in $\nu_n^- S'$ and end in $S''$.
\item Assume that $p$ starts in $S''$ and ends in $\nu_n^- S'$. Hence, by Proposition~\ref{proposition.homslices} the path $p$ to lies entirely in $S \oplus \nu_n^- S$. Then, since $\Hom_{\mathcal{D}_{\Lambda}}(S'', S') = 0$, it can pass neither through $S'$ nor through $\nu_n^- S''$. Therefore it lies entirely in our slice.
\end{enumerate}
Thus also Condition~(2) of Definition~\ref{def.slicegen} is satisfied.
\end{proof}

\begin{lemma} \label{lemma.slicestrans}
\begin{enumerate}
\item Any two slices in $\mathcal{U}$ are iterated mutations of each other. 
\item If moreover the quiver of $\Lambda$ contains no oriented cycles, then any two slices are iterated mutations with respect to sinks or sources of each other.
\end{enumerate}
\end{lemma}

\begin{proof}
Let $\Lambda = \bigoplus P_i$ be a decomposition into indecomposable projectives. We choose $d_i$ and $e_i$ such that the two slices are $\bigoplus \nu_n^{d_i} P_i$ and $\bigoplus \nu_n^{e_i} P_i$, respectively. Since $\mu^+_S(S) = \nu_n^-S$, we can assume $e_i > d_i$ for all $i$. We set $I = \{ i \mid e_i - d_i \text{ is maximal} \}$,
\[ S' = \bigoplus_{i \in I} \nu_n^{e_i} P_i \text{ and } S'' = \bigoplus_{j \not\in I} \nu_n^{e_j} P_j.\]
Now for $i \in I$ and $j \not\in I$ we have
\[\Hom_{\mathcal{D}_{\Lambda}}(\nu_n^{e_j} P_j, \nu_n^{e_i} P_i) = \Hom_{\mathcal{D}_{\Lambda}}(\nu_n^{d_j} P_j, \nu_n^{(e_i - d_i) - (e_j - d_j)} \nu_n^{d_i} P_i).\]
Since by our choice of $I$ we have $(e_i - d_i) - (e_j - d_j) > 0$, the above space vanishes by Proposition~\ref{proposition.homslices}. Hence we may mutate, and obtain
\[ \mu^+_{S'} (\bigoplus \nu_n^{e_i} P_i) = (\bigoplus_{i \in I} \nu_n^{e_i-1} P_i) \oplus ( \bigoplus_{j \not\in I} \nu_n^{e_j} P_j ). \]
Repeating this procedure we see that any two slices are iterated mutations of each other.

For the proof of the second claim first note that if the quiver of $\Lambda$ containes no oriented cycles, then neither does the quiver of $\mathcal{U}$. So we can number the indecomposable direct summands of $S'$ as $S' = S'_1 \oplus \cdots \oplus S'_\ell$ such that $\Hom_{\mathcal{D}_{\Lambda}}(S'_i,S'_j) = 0$ for any $i>j$. Then we have $\mu_{S'}^+(S) = \mu_{S'_\ell}^+ \circ \cdots \circ \mu_{S'_1}^+(S)$ by Proposition~\ref{proposition.homslices}.
\end{proof}

\begin{theorem} \label{theorem.slices=itereatedAPR}
Assume that $\Lambda$ is $n$-representation-finite.
\begin{enumerate}
\item The iterated $n$-APR tilting complexes of $\Lambda$ are exactly the slices in $\mathcal{U}$.
\item If moreover the quiver of $\Lambda$ contains no oriented cycles, then any iterated $n$-APR tilting complex can be obtained by a sequence of $n$-APR (co)tilts in the sense of Definition~\ref{def.nAPR}.
\end{enumerate}
\end{theorem}

\begin{proof}
(1) By Observation~\ref{obs.tiltisslice} any iterated $n$-APR tilt comes from a slice. The converse follows from Lemma~\ref{lemma.slicestrans}(1) and Observation~\ref{obs.alwaysAPR_der}.

(2) Follows similarly using Lemma~\ref{lemma.slicestrans}(2) and Remark~\ref{remark.in_mod=in_der}.
\end{proof}

\subsection{Admissible sets} \label{subsect.admsets}

In this subsection we will see that all the endomorphism rings of slices, and hence all the iterated $n$-APR tilts, of an $n$-representation-finite algebra are obtained as quotients of the $(n+1)$-preprojective algebra (see Definition~\ref{def.preproj}).

\begin{lemma} \label{lemma.slicesradical}
Let $S$ be a slice in $\mathcal{U}$. Then
\[ \Hom_{\mathcal{U}}(S, \nu_n^{-i} S) \subseteq \Rad_{\mathcal{U}}^i(S, \nu_n^{-i} S). \]
\end{lemma}

\begin{proof}
By Theorem~\ref{theorem.slices=itereatedAPR} we may assume $S$ to be the slice $\Lambda$. Then the claim follows from Proposition~\ref{proposition.tensoralg}.
\end{proof}

\begin{construction} \label{const.standardadmset}
For $P, Q \in \add \Lambda$ indecomposable we choose
\begin{align*}
 C_0(P, Q) & \subseteq \Rad_{\mathcal{U}}(P, \nu_n^- Q) \text{ such that } C_0(P, Q) \text{ is a minimal generating set of } \\
& \Rad_{\mathcal{U}}(P, \nu_n^- Q) / \Rad^2_{\mathcal{U}} (P, \nu_n^- Q) \text{ as } {\textstyle \frac{\End_{\mathcal{U}}(P)^{\op}}{\Rad \End_{\mathcal{U}}(P)^{\op}}\text{-}\frac{\End_{\mathcal{U}}(Q)^{\op}}{\Rad \End_{\mathcal{U}}(Q)^{\op}}} \text{-bimodule, and} \\
 H(P, Q) & \subseteq \Rad_{\mathcal{U}}(P, Q) \text{ such that } H(P, Q) \text{ is a minimal generating set of } \\
& \Rad_{\mathcal{U}}(P, Q) / \Rad^2_{\mathcal{U}} (P, Q) \text{ as } {\textstyle \frac{\End_{\mathcal{U}}(P)^{\op}}{\Rad \End_{\mathcal{U}}(P)^{\op}}\text{-}\frac{\End_{\mathcal{U}}(Q)^{\op}}{\Rad \End_{\mathcal{U}}(Q)^{\op}}} \text{-bimodule.}
\intertext{We set}
 A(P,Q) & = C_0(P, Q) {\textstyle \coprod} H(P, Q) \subseteq \Hom_{\mathcal{C}_{\Lambda}^n}(P,Q)
\end{align*}
We write $C_0 = \coprod_{P,Q} C_0(P,Q)$ and $A = \coprod_{P,Q} A(P,Q)$. Note that by Definition~\ref{def.preproj} the set $A(P,Q)$ generates $\Rad_{\mathcal{C}_{\Lambda}^n}(P,Q) / \Rad_{\mathcal{C}_{\Lambda}^n}^2(P,Q)$.
\end{construction}

If $k$ is algebraically closed, then $H$ consists of the arrows in the quiver of $\Lambda$, and $C_0$ consist of the additional arrows in the quiver of $\widehat{\Lambda}$. Thus $A$ consist of all arrows in the quiver of $\widehat{\Lambda}$.

\begin{lemma}
\[ \Lambda \iso \widehat{\Lambda} / (C_0). \]
\end{lemma}

\begin{proof}
This follows from Proposition~\ref{proposition.tensoralg} and the definition of $C_0$ above.
\end{proof}

\begin{definition} \label{def.admset}
\begin{enumerate}
\item We call $C_0$ as above the \emph{standard admissible set}.
\item For $C \subset A$ and a decomposition $\Lambda = \Lambda' \oplus \Lambda''$ (as modules) with
\begin{enumerate}
\item $\add \Lambda' \cap \add \Lambda'' = 0$,
\item for $P \in \add \Lambda'$ and $Q \in \add \Lambda''$ indecomposable we have $C(P,Q) = \emptyset$
\item for $P \in \add \Lambda''$ and $Q \in \add \Lambda'$ indecomposable we have $C(P,Q) = A(P,Q)$
\end{enumerate}
we define a new subset $\mu_{\Lambda'}^+(C) = \mu_{\Lambda''}^-(C) \subseteq A$ by
\[ \mu_{\Lambda'}^+(C)(P,Q) = \left\{ \begin{array}{cl} C(P,Q) & \text{ if } P \oplus Q \in \add \Lambda' \\ C(P,Q) & \text{ if } P \oplus Q \in \add \Lambda'' \\ A(P,Q) & \text{ if } P \in \add \Lambda' \text{ and } Q \in \add \Lambda'' \\ \emptyset & \text{ if } P \in \add \Lambda'' \text{ and } Q \in \add \Lambda' \end{array} \right. \]
That is, we remove from $C$ all arrows $\add \Lambda'' \to \add \Lambda'$, and add all arrows $\add \Lambda' \to \add \Lambda''$ in $A$.

We call this set a \emph{mutation} of $C$.
\item An \emph{admissible set} is a subset of $A$ which is an iterated mutation of the standard admissible set.
\end{enumerate}
\end{definition}

We will now investigate the relation of slices in $\mathcal{U}$ and admissible sets.

\begin{construction} \label{construct.setfromslice}
Let $S = \bigoplus \nu_n^{s_i} P_i$ be a slice in $\mathcal{U}$. We set
\[ C_S(P_i, P_j) = \{ \varphi \in A(P_i, P_j) \mid \varphi \text{ is a map } P_i \to \nu_n^{s_j - s_i - r} P_j \text{ for some } r > 0 \}. \]
\end{construction}

\begin{proposition} \label{prop.endo=slice}
For any slice $S$ in $\mathcal{U}$ we have
\[ \End_{\mathcal{D}_{\Lambda}}(S)^{\op} \iso \widehat{\Lambda} / (C_S). \]
\end{proposition}

\begin{proof}
We have 
\begin{align*}
\End_{\mathcal{D}_{\Lambda}}(S)^{\op} & = \Hom_{\mathcal{D}_{\Lambda}}(S, \bigoplus \nu_n^i S) / (\text{maps } S \to \nu_n^- S) && (\text{by \ref{proposition.homslices} \text{ and } \ref{lemma.slicesradical}})\\
& = \widehat{\Lambda} / (C_S) && (\text{by definition of } C_S) \qedhere
\end{align*}
\end{proof}

\begin{proposition} \label{prop.admset=slice}
\begin{enumerate}
\item The map $C_? \colon S \mapsto C_S$ sends slices in $\mathcal{U}$ to admissible sets. Moreover any admissible set is of the form $C_S$ for some slice $S$.
\item $C_?$ commutes with mutations in the following way:
\begin{align*}
C_{\mu^+_{S'}(S)} & = \mu^+_{\Lambda'} (C_S) \text{, and} \\
C_{\mu^-_{S''}(S)} & = \mu^-_{\Lambda''} (C_S)\\
\end{align*}
whenever $S = S' \oplus S''$ and $\Lambda = \Lambda' \oplus \Lambda''$ such that $\pi(S') \iso \pi(\Lambda')$ and $\pi(S'') \iso \pi(\Lambda'')$ (recall that $\pi$ denotes the map from the derived category to the $n$-Amiot cluster category -- see Definition~\ref{def.amiot_CC}). In particular the mutations of slices are defined if and only if the mutations of admissible sets are.
\end{enumerate}
\end{proposition}

\begin{proof}
By definition $\Lambda$ is a slice, and $C_{\Lambda} = C_0$ is the standard admissible set. We now proceed by checking that all these properties are preserved under mutation.

Assume we are in the setup of (2), that is $S = S' \oplus S''$ is a slice, $\Lambda = \Lambda' \oplus \Lambda''$, such that $\pi(S') \iso \pi(\Lambda')$ and $\pi(S'') \iso \pi(\Lambda'')$. We may further inductively assume that $C_S$ is an admissible set.
\begin{align*}
& \text{all maps } \widehat{\Lambda}'' \to \widehat{\Lambda}' \text{ in } A \text{ lie in } C_S \\
\iff & \Hom_{\mathcal{D}_{\Lambda}}(S'', S') = 0 && \text{(by Proposition~\ref{prop.endo=slice})} \\
\iff & S' \text{ admits a mutation} && \text{(by Definition~\ref{def.slicemutation})} \\
\implies & \mu^+_{S'}(S) = \nu_m^- S' \oplus S'' \text{ is also a slice } && \text{(by Lemma~\ref{lemma.slicemutation})} \\
\implies & \Hom_{\mathcal{U}}(S', \nu_n^- S'') \subseteq \Rad_{\mathcal{U}}^2(S', \nu_n^- S'') && \text{(by Lemma~\ref{lemma.slicesradical})} \\
\implies & \text{no maps } \widehat{\Lambda}' \to \widehat{\Lambda}'' \text{ in } A \text{ lie in } C_S && \text{(by Proposition~\ref{prop.endo=slice})}
\end{align*}
Therefore, the ``in particular''-part of (2) holds. Similar to the arguments above one sees that $C_{\mu^+_{S'}(S)} = \mu^+_{\Lambda'} (C_S)$.

Now the surjectivity in (1) follows from the fact that, by definition, any admissible set is an iterated mutation of the standard admissible set.
\end{proof}

\begin{theorem} \label{theorem.adm=iteratedAPR}
Let $\Lambda$ be $n$-representation-finite. Then the iterated $n$-APR tilts of $\Lambda$ are precisely the algebras of the form $\widehat{\Lambda} / (C)$, where $C$ is an admissible set.

In particular all these algebras are also $n$-representation-finite.
\end{theorem}

\begin{proof}
The first part follows from Propositions~\ref{prop.endo=slice}, \ref{prop.admset=slice} and Theorem~\ref{theorem.slices=itereatedAPR}. The second part then follows by Theorem~\ref{theorem.APRpres_repfin}.
\end{proof}

\section{$n$-representation-finite algebras of type $A$} \label{section.typeA}

The aim of this section is to construct $n$-representation-finite algebras of `type $A$'. The starting point (and the reason we call these algebras type $A$) is the construction of higher Auslander algebras of type $A_s$ in \cite{Iy_n-Auslander} (we recall this in Theorem~\ref{theorem.auslanderalg} here). The main result of this section is Theorem~\ref{theorem.main_typeA}, which gives an explicit combinatorial description of all iterated $n$-APR tilts of these higher Auslander algebras by removing certain arrows from a given quiver (see also Definitions~\ref{def.qns} and \ref{def.ca-set} for the notation used in that theorem).

\begin{definition} \label{def.qns}
\begin{enumerate}
\item For $n \geq 1$ and $s \ge 1$, let $Q^{(n,s)}$ be the quiver with vertices
\[Q_0^{(n,s)} = \{(\ell_1,\ell_2,\cdots,\ell_{n+1}) \in\Z_{\ge0}^{n+1} \mid \sum_{i=1}^{n+1}\ell_i=s-1\}\]
and arrows
\[Q_1^{(n,s)} = \{x \tol{\scriptstyle i} x+f_i \mid i \in \{1, \ldots, n+1\}, x,x+f_i \in Q_0^{(n,s)}\},\]
where $f_i$ denotes the vector
\[ f_i = (0, \ldots,0, \overset{\mathclap{i}}{-1}, \overset{\mathclap{i+1}}{1}, 0, \ldots,0)\in\Z^{n+1} \]
(cyclically, that is $f_{n+1} = (\overset{\mathclap{1}}{1}, 0, \ldots, 0, \overset{\mathclap{n+1}}{-1})$).
\item For $n \geq 1$ and $s \geq 1$, we define the $k$-algebra $\widehat{\Lambda}^{(n,s)}$ to be the quiver algebra of $Q^{(n,s)}$ with the following relations:

For any $x\in Q_0^{(n,s)}$ and $i,j\in \{1, \ldots, n+1\}$ satisfying $x+f_i$, $x+f_i+f_j\in Q_0^{(n,s)}$,
\[(x \tol{\scriptstyle i} x+f_i \tol{\scriptstyle j} x+f_i+f_j)
=\left\{\begin{array}{cl}
(x \tol{\scriptstyle j} x+f_j \tol{\scriptstyle i} x+f_i+f_j) & \text{if } x+f_j\in Q_0^{(n,s)},\\
0&\text{otherwise.}
\end{array}\right.\]
(We will show later that this notation is justified: In Subsection~\ref{subsect.outline_ca} we construct algebras $\Lambda^{(n,s)}$, such that $\widehat{\Lambda}^{(n,s)}$ is the $(n+1)$-preprojective algebra of $\Lambda^{(n,s)}$ -- see also Proposition~\ref{prop.preproj}.)
\end{enumerate}
\end{definition}

\begin{example}
\begin{enumerate}
\item The quiver $Q^{(1,s)}$ is the following.
\[ \begin{tikzpicture}[xscale=2.5,node distance=-13pt]
 \node (A) at (-.1,0) {$(s-1,0)$}; \node (A++) [above right=of A] {}; \node (A+-) [below right=of A] {};
 \node (B) at (1,0) {$(s-2,1)$}; \node (B--) [below left=of B] {}; \node (B-+) [above left=of B] {}; \node (B+-) [below right=of B] {}; \node (B++) [above right=of B] {};
 \node (C) at (2,0) [text=white] {$(0,0)$};  \node (C--) [below left=of C] {}; \node (C-+) [above left=of C] {}; \node (C+-) [below right=of C] {}; \node (C++) [above right=of C] {};
 \node (C) at (2,0) {$\cdots$};
 \node (D) at (3,0) {$(1,s-2)$};  \node (D--) [below left=of D] {}; \node (D-+) [above left=of D] {}; \node (D+-) [below right=of D] {}; \node (D++) [above right=of D] {};
 \node (E) at (4.1,0) {$(0,s-1)$};  \node (E--) [below left=of E] {}; \node (E-+) [above left=of E] {};
 \draw [->] (A++) -- node [gap] {$\scriptstyle 1$} (B-+);
 \draw [->] (B++) -- node [gap] {$\scriptstyle 1$} (C-+);
 \draw [->] (C++) -- node [gap] {$\scriptstyle 1$} (D-+);
 \draw [->] (D++) -- node [gap] {$\scriptstyle 1$} (E-+);
 \draw [->] (B--) -- node [gap] {$\scriptstyle 2$} (A+-);
 \draw [->] (C--) -- node [gap] {$\scriptstyle 2$} (B+-);
 \draw [->] (D--) -- node [gap] {$\scriptstyle 2$} (C+-);
 \draw [->] (E--) -- node [gap] {$\scriptstyle 2$} (D+-);
\end{tikzpicture} \]
The algebra $\widehat{\Lambda}^{(1,s)}$ is the preprojective algebra of type $A_s$.
\item The quiver $Q^{(2,4)}$ is
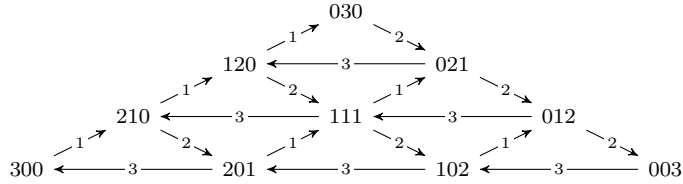
\[ \begin{tikzpicture}[xscale=1.4,yscale=.7]
 \node (300) at (0,0) {$\scriptstyle 300$};
 \node (210) at (1,1) {$\scriptstyle 210$};
 \node (120) at (2,2) {$\scriptstyle 120$};
 \node (030) at (3,3) {$\scriptstyle 030$};
 \node (201) at (2,0) {$\scriptstyle 201$};
 \node (111) at (3,1) {$\scriptstyle 111$};
 \node (021) at (4,2) {$\scriptstyle 021$};
 \node (102) at (4,0) {$\scriptstyle 102$};
 \node (012) at (5,1) {$\scriptstyle 012$};
 \node (003) at (6,0) {$\scriptstyle 003$};
 \draw [->] (300) -- node [gap] {$\scriptscriptstyle 1$} (210);
 \draw [->] (210) -- node [gap] {$\scriptscriptstyle 1$} (120);
 \draw [->] (120) -- node [gap] {$\scriptscriptstyle 1$} (030);
 \draw [->] (201) -- node [gap] {$\scriptscriptstyle 1$} (111);
 \draw [->] (111) -- node [gap] {$\scriptscriptstyle 1$} (021);
 \draw [->] (102) -- node [gap] {$\scriptscriptstyle 1$} (012);
 \draw [->] (210) -- node [gap] {$\scriptscriptstyle 2$} (201);
 \draw [->] (120) -- node [gap] {$\scriptscriptstyle 2$} (111);
 \draw [->] (030) -- node [gap] {$\scriptscriptstyle 2$} (021);
 \draw [->] (111) -- node [gap] {$\scriptscriptstyle 2$} (102);
 \draw [->] (021) -- node [gap] {$\scriptscriptstyle 2$} (012);
 \draw [->] (012) -- node [gap] {$\scriptscriptstyle 2$} (003);
 \draw [->] (201) -- node [gap] {$\scriptscriptstyle 3$} (300);
 \draw [->] (111) -- node [gap] {$\scriptscriptstyle 3$} (210);
 \draw [->] (021) -- node [gap] {$\scriptscriptstyle 3$} (120);
 \draw [->] (102) -- node [gap] {$\scriptscriptstyle 3$} (201);
 \draw [->] (012) -- node [gap] {$\scriptscriptstyle 3$} (111);
 \draw [->] (003) -- node [gap] {$\scriptscriptstyle 3$} (102);
\end{tikzpicture} \]
The algebras $\widehat{\Lambda}^{(2,s)}$ appeared in the work of Geiss, Leclerc, and Schr{\"o}er \cite{GLS1,GLS2}.
\item The quiver $Q^{(3,3)}$ is
\[ \begin{tikzpicture}[xscale=1.6,yscale=.4]
 \node (2000) at (0,0) {$\scriptstyle 2000$};
 \node (1100) at (1,1) {$\scriptstyle 1100$};
 \node (0200) at (2,2) {$\scriptstyle 0200$};
 \node (1010) at (2,0) {$\scriptstyle 1010$};
 \node (0110) at (3,1) {$\scriptstyle 0110$};
 \node (0020) at (4,0) {$\scriptstyle 0020$};
 \node (1001) at (1.2,-4) {$\scriptstyle 1001$};
 \node (0101) at (2.2,-3) {$\scriptstyle 0101$};
 \node (0011) at (3.2,-4) {$\scriptstyle 0011$};
 \node (0002) at (2.4,-8) {$\scriptstyle 0002$};
 \draw [->] (2000) -- node [gap] {$\scriptscriptstyle 1$} (1100);
 \draw [->] (1100) -- node [gap] {$\scriptscriptstyle 1$} (0200);
 \draw [->] (1010) -- node [gap] {$\scriptscriptstyle 1$} (0110);
 \draw [->] (1100) -- node [gap] {$\scriptscriptstyle 2$} (1010);
 \draw [->] (0200) -- node [gap] {$\scriptscriptstyle 2$} (0110);
 \draw [->] (0110) -- node [gap] {$\scriptscriptstyle 2$} (0020);
 \draw [->] (1010) -- node [gap] {$\scriptscriptstyle 3$} (1001);
 \draw [->] (0110) -- node [gap] {$\scriptscriptstyle 3$} (0101);
 \draw [->] (0020) -- node [gap] {$\scriptscriptstyle 3$} (0011);
 \draw [->] (1001) -- node [gap] {$\scriptscriptstyle 4$} (2000);
 \draw [->] (0101) -- node [gap] {$\scriptscriptstyle 4$} (1100);
 \draw [->] (0011) -- node [gap] {$\scriptscriptstyle 4$} (1010);
 \draw [->] (1001) -- node [gap] {$\scriptscriptstyle 1$} (0101);
 \draw [->] (0101) -- node [gap] {$\scriptscriptstyle 2$} (0011);
 \draw [->] (0011) -- node [gap] {$\scriptscriptstyle 3$} (0002);
 \draw [->] (0002) -- node [gap] {$\scriptscriptstyle 4$} (1001);
\end{tikzpicture} \]
\end{enumerate}
\end{example}

\begin{definition} \label{def.ca-set}
We call a subset $C \subseteq Q_1^{(n,s)}$ of the arrows of $Q^{(n,s)}$ \emph{cut}, if it contains exactly one arrow from each $(n+1)$-cycle (see \cite{BMR_cta,BRS,BFPPT} for similar constructions).
\end{definition}

\begin{remark}
\begin{enumerate}
\item We will show later (see Remark~\ref{rem.cset=set}) that cuts coincide with admissible sets (as introduced in Definition~\ref{def.admset}).
\item
Clearly, in Definition~\ref{def.ca-set}, any $(n+1)$-cycle is of the form
\[ x \tol[30]{\scriptstyle \sigma(1)} x + f_{\sigma(1)} \tol[30]{\scriptstyle \sigma(2)} x + f_{\sigma(1)} + f_{\sigma(2)} \tol[30]{\scriptstyle \sigma(3)} \cdots \tol[30]{\scriptstyle \sigma(n)} x + f_{\sigma(1)} + \cdots + f_{\sigma(n)} \tol[30]{\scriptstyle \sigma(n+1)} x, \]
for some $\sigma \in \mathfrak{S}_{n+1}$.
\end{enumerate}
\end{remark}

\begin{example}
\begin{enumerate}
\item Clearly the cuts of $Q^{(1,s)}$ correspond bijectively to orientations of the Dynkin diagram $A_s$.
\item See Tables~\ref{table.a3linear}, \ref{table.a4linear} and \ref{table.highera3linear} for the cuts of $Q^{(2,3)}$, $Q^{(2,4)}$, and $Q^{(3,3)}$, respectively.
\end{enumerate}
\end{example}

We are now ready to state the main result of this section.

\begin{theorem} \label{theorem.main_typeA}
\begin{enumerate}
\item Let $Q^{(n,s)}$ as in Definition~\ref{def.qns}, and let $C$ be a cut. Then the algebra
\[ \Lambda^{(n,s)}_C := \widehat{\Lambda}^{(n,s)} / (C) \]
is $n$-representation-finite.
\item All these algebras (for fixed $(n,s)$) are iterated $n$-APR tilts of one another.
\end{enumerate}
\end{theorem}

We call the algebras of the form $\Lambda^{(n,s)}_C$ as in the theorem above \emph{$n$-representation-finite of type $A$}. Note that $1$-representation-finite algebras of type $A$ are exactly path algebras of Dynkin quivers of type $A$. See Tables~\ref{table.a3linear}, \ref{table.a4linear}, and \ref{table.highera3linear} for the examples $(n,s) = (2,3)$, $(2,4)$, and $(3,3)$, respectively.

\begin{table}
\iteratedAprLinearAFour
\caption{Iterated 2-APR tilts of the Auslander algebra of linear oriented $A_4$ (thick lines indicate cuts)} \label{table.a4linear}
\end{table}

\begin{table}
\iteratedAprLinearAThreeTwo
\caption{Iterated 3-APR tilts of the higher Auslander algebra of linear oriented $A_3$ (thick lines indicate cuts)} \label{table.highera3linear}
\end{table}

\subsection{Outline of the proof of Theorem~\ref{theorem.main_typeA}} \label{subsect.outline_ca} $ $

\medskip \noindent
{\sc Step 1.} Let $C_0$ be the set of all arrows of type $n+1$. This is clearly a cut. We set
\[ \Lambda^{(n,s)} := \Lambda_{C_0}^{(n,s)}. \]
For example, $\Lambda^{(1,s)}$ is a path algebra of the linearly oriented Dynkin quiver $A_s$, and $\Lambda^{(2,s)}$ is the Auslander algebra of $\Lambda^{(1,s)}$. More generally, the following result is shown in \cite{Iy_n-Auslander}.

\begin{theorem}[see \cite{Iy_n-Auslander}] \label{theorem.auslanderalg}
The algebra $\Lambda^{(n,s)}$ is $n$-representation-finite. In particular $\mod \Lambda^{(n,s)}$ has a unique basic $n$-cluster tilting object $M^{(n,s)}$. We have
\[ \Lambda^{(n+1,s)} \iso \End_{\Lambda^{(n,s)}} (M^{(n,s)})^{\op}, \]
that is $\Lambda^{(n+1,s)}$ is the $n$-Auslander algebra of $\Lambda^{(n,s)}$.
\end{theorem}

\noindent
{\sc Step 2.} We now introduce mutation on cuts.

For simplicity of notation, we fix $n$ and $s$ for the rest of this section, and omit all superscripts $-^{(n,s)}$ whenever there is no danger of confusion. (That is, by $Q$ we mean $Q^{(n,s)}$, by $\Lambda$ we mean $\Lambda^{(n,s)}$ and similar.)

\begin{definition} \label{definition.csetmut}
Let $C$ be a cut of $Q$.
\begin{enumerate}
\item We denote by $Q_C$ the quiver obtained by removing all arrows in $C$ from $Q$.
\item Let $x$ be a source of the quiver $Q_C$. Define a subset $\mu_x^+(C)$ of $Q_1$ by removing all arrows in $Q$ ending at $x$ from $C$ and adding all arrows in $Q$ starting at $x$ to $C$.
\item Dually, for each sink $x$ of $Q_C$, we get another subset $\mu_x^-(C)$ of $Q_1$.
\end{enumerate}
We call the process of replacing a cut $C$ by $\mu_x^+(C)$ or $\mu_x^-(C)$, when the conditions of (2) or (3) above are satisfied, \emph{mutation} of cuts.
\end{definition}

We will show in Proposition~\ref{proposition.csetmutation} in Subsection~\ref{subsect.ca-set-mut_OK} that mutations of cuts are again cuts.

\begin{observation}
The quiver $Q_C$ is the quiver of the algebra $kQ / (C)$.
\end{observation}

The following remark explains the relationship between cuts and admissible sets.

\begin{remark} \label{rem.adm_subset_ca}
\begin{enumerate}
\item Whenever we mention admissible sets, it is implicitly understood that we choose $A = Q_1$ the set of arrows in $Q$ in Definition~\ref{def.admset}. (It is shown in Subsection~\ref{subsect.admsets} that the choice of $A$ does not matter there, but with this choice we can easier compare admissible sets and cuts.)
\item When $C$ is a cut and an admissible set, and $x$ is a source of $Q_C$, then the mutations $\mu_x^+(C)$ of $C$ as cut and as admissible set coincide.
\item The standard admissible set $C_0$, as defined in Construction~\ref{const.standardadmset} and Definition~\ref{def.admset}, is identical to the set $C_0$ defined in {\sc Step 1}. In particular it is a cut.
\item By (3) and (2) we know that any admissible set is a cut. The converse follows when we have shown that all cuts are iterated mutations of one another (see Theorem~\ref{new_theorem_B} and Remark~\ref{rem.cset=set}).
\end{enumerate}
\end{remark}

We need the following purely combinatorial result, which will be proven in Subsections~\ref{subsect.comb_slices} to \ref{subsect.cycle_lemma}.

\begin{theorem} \label{new_theorem_B}
All cuts of $Q$ are successive mutations of one another.
\end{theorem}

\noindent
{\sc Step 3.} Finally, we need the following result which will also be shown in Subsections~\ref{subsect.comb_slices} to \ref{subsect.cycle_lemma}.

\begin{proposition} \label{prop.new_theorem_A}
\begin{enumerate}
\item $\Lambda_{\mu^+_x(C)}$ is an $n$-APR tilt of $\Lambda_C$.
\item $\Lambda_{\mu^+_x(C)}$ is $n$-representation-finite if and only if so is $\Lambda_C$.
\end{enumerate}
\end{proposition}

Now Theorem~\ref{theorem.main_typeA} follows:
\begin{proof}[Proof of Theorem~\ref{theorem.main_typeA}]
By Theorem~\ref{theorem.auslanderalg}, there is a cut $C_0$ such that $\Lambda_{C_0}$ is strictly $n$-re\-pre\-sen\-ta\-tion-finite. By Proposition~\ref{prop.new_theorem_A} this property is preserved under mutation of cuts, and by Theorem~\ref{new_theorem_B} all cuts are iterated mutations of $C_0$.
\end{proof}

\begin{remark} \label{rem.cset=set}
Theorem~\ref{new_theorem_B}, together with Remark~\ref{rem.adm_subset_ca}(2) and (3), shows that in the setup of Definition~\ref{def.qns} the set of cuts (as defined in \ref{def.ca-set}) and the set of admissible sets (as defined in \ref{def.admset}), coincide.
\end{remark}

\subsection{Mutation of cuts} \label{subsect.ca-set-mut_OK}

In this subsection we show that the mutations $\mu_x^+(C)$ (or $\mu_x^-(C)$) as in Definition~\ref{definition.csetmut} for a cut $C$ are cuts again.

\begin{proposition} \label{proposition.csetmutation}
In the setup of Definition~\ref{definition.csetmut} we have the following:
\begin{enumerate}
\item Any arrow in $Q$ ending at $x$ belongs to $C$, and any arrow in $Q$ starting at $x$ does not belong to $C$.
\item $\mu^+_x(C)$ is again a cut.
\item $x$ is a sink of the quiver $Q_{\mu^+_x(C)}$.
\end{enumerate}
\end{proposition}

For the proof we need the following observation, which tells us that any sequence of arrows of pairwise different type may be completed to an $(n+1)$-cycle.

\begin{lemma} \label{addition}
Let $x \in \Z^{n+1}$ and let $\sigma \colon \{1,\ldots,\ell\} \to \{1, \ldots, n+1\}$ be an injective map.
Assume that $x+\sum_{j=1}^{i}f_{\sigma(j)}$ belongs to $Q_0$ for any $0\le i\le\ell$.
Then $\sigma$ extends to an element $\sigma\in\mathfrak{S}_{n+1}$ such that
$x+\sum_{j=1}^{i}f_{\sigma(j)}$ belongs to $Q_0$ for any $0\le i\le n+1$.
\end{lemma}

\begin{proof}
The statement makes sense only for $s \geq 2$. We set $I := \{0, \ldots, s-1\}$. For any $i \in \{1, \ldots, n+1\}$ we have $x_i+1 \in I$ or $x_i-1 \in I$.

We can assume $\ell<n+1$.
We will define $\sigma(\ell+1)\in\{1, \ldots, n+1\}$ such that $x+\sum_{j=1}^{\ell+1}f_{\sigma(j)}$ belongs to $Q_0$.
Without loss of generality, we assume that $i_0$ and $i_1(\neq i_0,i_0+1)$ belong to $\Im\sigma$, but none of $i_0+1,i_0+2,\cdots,i_1-1$ belong to $\Im\sigma$.
Since $x$ and $x+\sum_{j=1}^{\ell}f_{\sigma(j)}$ belong to $Q_0$, we have
\[ x_{i_0+1} \in I,\,  x_{i_0+1} + 1 \in I,\, x_{i_1} \in I, \text{ and } x_{i_1} - 1 \in I.\]
If $i_1=i_0+2$, then $\sigma(\ell+1):=i_0+1$ satisfies the desired condition.
In the rest, we assume $i_1\neq i_0+2$.
We divide into three cases.

(i) If $x_{i_0+2}+1 \in I$, then $\sigma(\ell+1):=i_0+1$ satisfies the desired condition.

(ii) If $x_{i_1-1}-1 \in I$, then $\sigma(\ell+1):=i_1-1$ satisfies the desired condition.

(iii) By (i) and (ii), we can assume $x_{i_0+2}-1 \in I$, $x_{i_1-1}+1 \in I$ and $i_1 \neq i_0+3$.
Then there exists $i_0+2\le i_2<i_1-1$ satisfying
\[x_{i_2}-1 \in I \quad \text{ and } \quad x_{i_2+1}+1 \in I.\]
Then $\sigma(\ell+1):=i_2$ satisfies the desired condition.
\end{proof}

\begin{proof}[Proof of Proposition~\ref{proposition.csetmutation}]
\begin{enumerate}
\item The former condition is clear since $x$ is a source of $Q_C$. Assume that an arrow $a$ starting at $x$ belongs to $C$. By Lemma~\ref{addition} we know that $a$ is part of an $(n+1)$-cycle $c$. Then $c$ contains at least two arrows which belong to $C$, a contradiction.
\item Let $c$ be an $(n+1)$-cycle. We only have to check that exactly one of the $(n+1)$ arrows in $c$ is contained in $\mu^+_x(C)$. This is clear if $x$ is not contained in $c$. Assume that $x$ is contained in $c$, and let $a$ and $b$ be the arrows in $c$ ending and starting in $x$, respectively. Since $C$ is a cut $a$ is the unique arrow in $c$ contained in $C$. Thus $b$ is the unique arrow in $c$ contained in $\mu^+_x(C)$.
\item Clear from (1). \qedhere
\end{enumerate}
\end{proof}

\subsection{$n$-cluster tilting in derived categories} \label{subsect.comb_slices}

This and the following two subsections are dedicated to the proofs of Theorem~\ref{new_theorem_B} and Proposition~\ref{prop.new_theorem_A}.

We consider a covering $\widetilde{Q}$ of $Q$, then introduce the notion of slices (see Definition~\ref{def.combslice}) in $\widetilde{Q}$, and their mutation. Then we construct a correspondence between cuts and $\nu_n$-orbits of slices (Theorem~\ref{new_theorem_C}) and show that slices are transitive under mutations (Theorem~\ref{new_theorem_D}). These results are the key steps of the proofs of Theorem~\ref{new_theorem_B} and Proposition~\ref{prop.new_theorem_A}.

We give the conceptual part of the proof in this subsection, and postpone the proof of the combinatorial parts (Theorems~\ref{new_theorem_C} and \ref{new_theorem_D}) to Subsection~\ref{subsect.first_combi}.

We recall the subcategory
\[ \mathcal{U} = \add \{ \nu_n^i \Lambda \mid i \in \mathbb{Z} \} \]
of $\mathcal{D}_{\Lambda}$ (see Subsection~\ref{sect.backgr.derived}).

\begin{definition} \label{notation.quiverU}
We denote by $\widetilde{Q} = \widetilde{Q}^{(n,s)}$ the quiver with
\begin{align*}
\widetilde{Q}_0 & = \{(\ell_1,\ell_2,\cdots,\ell_{n+1}\!:\!i)\in\Z_{\ge0}^{n+1}\times\Z \mid \sum_{j=1}^{n+1} \ell_j = s-1\},
\intertext{ (we separate the last entry of the vector to emphasize its special role)}
\widetilde{Q}_1 & = \{\widetilde{a}_{x,i} \colon x \tol{\scriptstyle i} x+g_i \mid 1 \le i\le n+1,\ x,x+g_i\in\widetilde{Q}_0 \},
\end{align*}
where $g_i$ denotes the vector
\[ g_i = \left\{ \begin{array}{cl} (0, \ldots,0, \overset{\mathclap{i}}{-1}, \overset{\mathclap{i+1}}{1}, 0, \ldots,0\!:\!0) & 1 \leq i \leq n \\ (\overset{\mathclap{1}}{1}, 0, \ldots, 0, \overset{\mathclap{n+1}}{-1}\!:\!1) & i = n+1 \end{array} \right. . \]
We consider the category obtained from the quiver $\widetilde{Q}$ by factoring out the relations
\begin{align*}
& [ x \tol[30]{\scriptstyle i} x+g_i \tol[30]{\scriptstyle j} x+g_i+g_j ] = [ x \tol[30]{\scriptstyle j} x+g_j \tol[30]{\scriptstyle i} x+g_i+g_j ] \\
& \qquad \text{if } x, x+g_i, x+g_j, x+g_i+g_j \in \widetilde{Q}_0 \\
& [ x \tol[30]{\scriptstyle i} x+g_i \tol[30]{\scriptstyle j} x+g_i+g_j ] = 0 \\
& \qquad \text{if } x, x+g_i, x+g_i+g_j \in \widetilde{Q}_0 \text{ and } x +g_j \not\in \widetilde{Q}_0
\end{align*}
\end{definition}

\begin{example}
The quiver $\widetilde{Q}^{(1,4)}$ is the following.
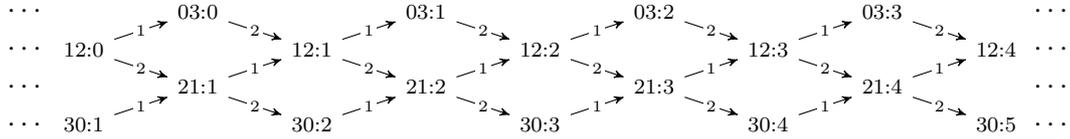
\[ \begin{tikzpicture}[xscale=1.5,yscale=.5]
 \foreach \x in {1,2,3,4,5}
  \node (30\x) at (2*\x,0) {$\scriptstyle 30:\x$};
 \foreach \x in {1,2,3,4}
  \node (21\x) at (2*\x+1,1) {$\scriptstyle 21:\x$};
 \foreach \x in {0,1,2,3,4}
  \node (12\x) at (2*\x+2,2) {$\scriptstyle 12:\x$};
 \foreach \x in {0,1,2,3}
  \node (03\x) at (2*\x+3,3) {$\scriptstyle 03:\x$};
 \foreach \x in {1,2,3,4}
  {
   \draw [->] (30\x) -- node [gap] {$\scriptscriptstyle 1$} (21\x);
   \draw [->] (21\x) -- node [gap] {$\scriptscriptstyle 1$} (12\x);
  }
 \foreach \x in {0,1,2,3}
  \draw [->] (12\x) -- node [gap] {$\scriptscriptstyle 1$} (03\x);
 \foreach \x/\y in {1/2,2/3,3/4,4/5}
  \draw [->] (21\x) -- node [gap] {$\scriptscriptstyle 2$} (30\y);
 \foreach \x/\y in {0/1,1/2,2/3,3/4}
  {
   \draw [->] (12\x) -- node [gap] {$\scriptscriptstyle 2$} (21\y);
   \draw [->] (03\x) -- node [gap] {$\scriptscriptstyle 2$} (12\y);
  }
 \foreach \x in {0,1,2,3}
  {
   \node at (1.5,\x) {$\cdots$};
   \node at (10.5,\x) {$\cdots$};
  }
\end{tikzpicture} \]
The quiver $\widetilde{Q}^{(2,4)}$ is the following.
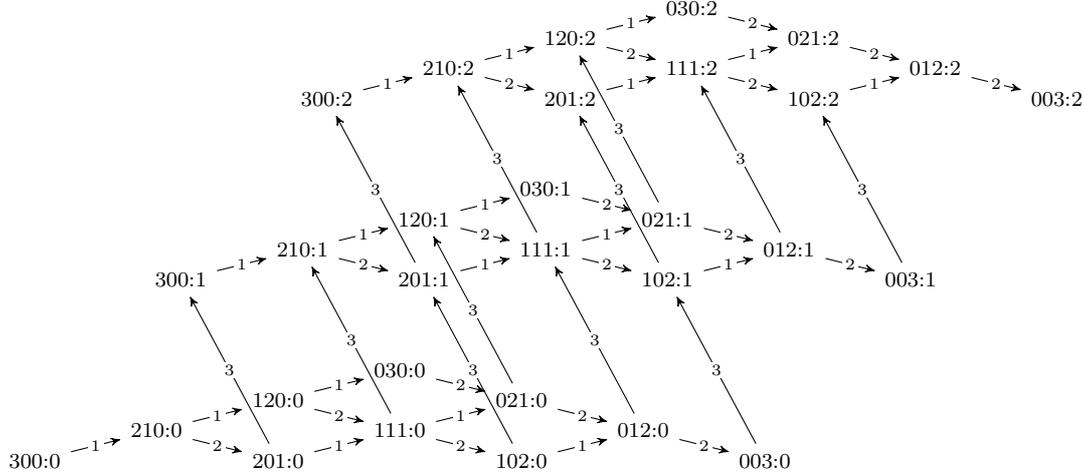
\[ \begin{tikzpicture}[xscale=1.6,yscale=.4]
 \node (3000) at (0,0) {$\scriptstyle 300:0$};
 \node (2100) at (1,1) {$\scriptstyle 210:0$};
 \node (1200) at (2,2) {$\scriptstyle 120:0$};
 \node (0300) at (3,3) {$\scriptstyle 030:0$};
 \node (2010) at (2,0) {$\scriptstyle 201:0$};
 \node (1110) at (3,1) {$\scriptstyle 111:0$};
 \node (0210) at (4,2) {$\scriptstyle 021:0$};
 \node (1020) at (4,0) {$\scriptstyle 102:0$};
 \node (0120) at (5,1) {$\scriptstyle 012:0$};
 \node (0030) at (6,0) {$\scriptstyle 003:0$};
 \node (3001) at (1.2,6) {$\scriptstyle 300:1$};
 \node (2101) at (2.2,7) {$\scriptstyle 210:1$};
 \node (1201) at (3.2,8) {$\scriptstyle 120:1$};
 \node (0301) at (4.2,9) {$\scriptstyle 030:1$};
 \node (2011) at (3.2,6) {$\scriptstyle 201:1$};
 \node (1111) at (4.2,7) {$\scriptstyle 111:1$};
 \node (0211) at (5.2,8) {$\scriptstyle 021:1$};
 \node (1021) at (5.2,6) {$\scriptstyle 102:1$};
 \node (0121) at (6.2,7) {$\scriptstyle 012:1$};
 \node (0031) at (7.2,6) {$\scriptstyle 003:1$};
 \node (3002) at (2.4,12) {$\scriptstyle 300:2$};
 \node (2102) at (3.4,13) {$\scriptstyle 210:2$};
 \node (1202) at (4.4,14) {$\scriptstyle 120:2$};
 \node (0302) at (5.4,15) {$\scriptstyle 030:2$};
 \node (2012) at (4.4,12) {$\scriptstyle 201:2$};
 \node (1112) at (5.4,13) {$\scriptstyle 111:2$};
 \node (0212) at (6.4,14) {$\scriptstyle 021:2$};
 \node (1022) at (6.4,12) {$\scriptstyle 102:2$};
 \node (0122) at (7.4,13) {$\scriptstyle 012:2$};
 \node (0032) at (8.4,12) {$\scriptstyle 003:2$};
 \draw [->] (3000) -- node [gap] {$\scriptscriptstyle 1$} (2100);
 \draw [->] (2100) -- node [gap] {$\scriptscriptstyle 1$} (1200);
 \draw [->] (1200) -- node [gap] {$\scriptscriptstyle 1$} (0300);
 \draw [->] (2010) -- node [gap] {$\scriptscriptstyle 1$} (1110);
 \draw [->] (1110) -- node [gap] {$\scriptscriptstyle 1$} (0210);
 \draw [->] (1020) -- node [gap] {$\scriptscriptstyle 1$} (0120);
 \draw [->] (2100) -- node [gap] {$\scriptscriptstyle 2$} (2010);
 \draw [->] (1200) -- node [gap] {$\scriptscriptstyle 2$} (1110);
 \draw [->] (0300) -- node [gap] {$\scriptscriptstyle 2$} (0210);
 \draw [->] (1110) -- node [gap] {$\scriptscriptstyle 2$} (1020);
 \draw [->] (0210) -- node [gap] {$\scriptscriptstyle 2$} (0120);
 \draw [->] (0120) -- node [gap] {$\scriptscriptstyle 2$} (0030);
 \draw [->] (2010) -- node [gap] {$\scriptscriptstyle 3$} (3001);
 \draw [->] (1110) -- node [gap] {$\scriptscriptstyle 3$} (2101);
 \draw [->] (0210) -- node [gap] {$\scriptscriptstyle 3$} (1201);
 \draw [->] (1020) -- node [gap] {$\scriptscriptstyle 3$} (2011);
 \draw [->] (0120) -- node [gap] {$\scriptscriptstyle 3$} (1111);
 \draw [->] (0030) -- node [gap] {$\scriptscriptstyle 3$} (1021);
 \draw [->] (3001) -- node [gap] {$\scriptscriptstyle 1$} (2101);
 \draw [->] (2101) -- node [gap] {$\scriptscriptstyle 1$} (1201);
 \draw [->] (1201) -- node [gap] {$\scriptscriptstyle 1$} (0301);
 \draw [->] (2011) -- node [gap] {$\scriptscriptstyle 1$} (1111);
 \draw [->] (1111) -- node [gap] {$\scriptscriptstyle 1$} (0211);
 \draw [->] (1021) -- node [gap] {$\scriptscriptstyle 1$} (0121);
 \draw [->] (2101) -- node [gap] {$\scriptscriptstyle 2$} (2011);
 \draw [->] (1201) -- node [gap] {$\scriptscriptstyle 2$} (1111);
 \draw [->] (0301) -- node [gap] {$\scriptscriptstyle 2$} (0211);
 \draw [->] (1111) -- node [gap] {$\scriptscriptstyle 2$} (1021);
 \draw [->] (0211) -- node [gap] {$\scriptscriptstyle 2$} (0121);
 \draw [->] (0121) -- node [gap] {$\scriptscriptstyle 2$} (0031);
 \draw [->] (2011) -- node [gap] {$\scriptscriptstyle 3$} (3002);
 \draw [->] (1111) -- node [gap] {$\scriptscriptstyle 3$} (2102);
 \draw [->] (0211) -- node [gap] {$\scriptscriptstyle 3$} (1202);
 \draw [->] (1021) -- node [gap] {$\scriptscriptstyle 3$} (2012);
 \draw [->] (0121) -- node [gap] {$\scriptscriptstyle 3$} (1112);
 \draw [->] (0031) -- node [gap] {$\scriptscriptstyle 3$} (1022);
 \draw [->] (3002) -- node [gap] {$\scriptscriptstyle 1$} (2102);
 \draw [->] (2102) -- node [gap] {$\scriptscriptstyle 1$} (1202);
 \draw [->] (1202) -- node [gap] {$\scriptscriptstyle 1$} (0302);
 \draw [->] (2012) -- node [gap] {$\scriptscriptstyle 1$} (1112);
 \draw [->] (1112) -- node [gap] {$\scriptscriptstyle 1$} (0212);
 \draw [->] (1022) -- node [gap] {$\scriptscriptstyle 1$} (0122);
 \draw [->] (2102) -- node [gap] {$\scriptscriptstyle 2$} (2012);
 \draw [->] (1202) -- node [gap] {$\scriptscriptstyle 2$} (1112);
 \draw [->] (0302) -- node [gap] {$\scriptscriptstyle 2$} (0212);
 \draw [->] (1112) -- node [gap] {$\scriptscriptstyle 2$} (1022);
 \draw [->] (0212) -- node [gap] {$\scriptscriptstyle 2$} (0122);
 \draw [->] (0122) -- node [gap] {$\scriptscriptstyle 2$} (0032);
\end{tikzpicture} \]
\end{example}

\begin{remark}
By abuse of notation we also denote the automorphism of $\widetilde{Q}$ induced by sending $(\ell_1,\ell_2,\cdots,\ell_{n+1}\!:\!i)$ to $(\ell_1,\ell_2,\cdots,\ell_{n+1}\!:\!i-1)$ by $\nu_n$, and the map $\widetilde{Q} \to Q$ induced by sending $(\ell_1,\ell_2,\cdots,\ell_{n+1}\!:\!i)$ to $(\ell_1,\ell_2,\cdots,\ell_{n+1})$ by $\pi$.
\end{remark}

The following result is shown in \cite[Theorem~6.10]{Iy_n-Auslander} (see Theorem~\ref{theorem.U_is_ct}).

\begin{theorem} \label{theorem.quiverU}
\begin{enumerate}
\item The $n$-cluster tilting subcategory $\mathcal{U}$ of $\mathcal{D}_{\Lambda}$ is presented by the quiver $\widetilde{Q}$ with relations as in \ref{notation.quiverU}.
\item In this presentation the indecomposable projective $\Lambda$-modules correspond to the vertices $(\ell_1, \ldots, \ell_{n+1}\!:\!0)$, and the indecomposable injective $\Lambda$-modules correspond to the vertices $(\ell_1, \ldots, \ell_{n+1}\!:\!\ell_1)$.
\item The $n$-cluster tilting $\Lambda$-module is given by the direct sum of all objects corresponding to the vertices between projective and injective $\Lambda$-modules.
\end{enumerate}
\end{theorem}

We now carry over the concept of slices to the quiver-setup.

\begin{definition} \label{def.combslice}
A \emph{slice} of $\widetilde{Q}$ is a full subquiver $S$ of $\widetilde{Q}$ satisfying the following conditions.
\begin{enumerate}
\item Any $\nu_n$-orbit in $\widetilde{Q}$ contains precisely one vertex which belongs to $S$.
\item $S$ is convex, i.e. for any path $p$ in $\widetilde{Q}$ connecting two vertices in $S$, all vertices appearing in $p$ belong to $S$.
\end{enumerate}
\end{definition}

\begin{remark}
Definition~\ref{def.combslice} is just a ``quiver version'' of Definition~\ref{def.slicegen}. In particular it is clear that slices in $\widetilde{Q}$ and slices in $\mathcal{U}$ are in natural bijection.
\end{remark}

Next we carry over Construction~\ref{construct.setfromslice} to this combinatorial situation, that is, we produce from any slice in $\widetilde{Q}$ a cut $C_S$.

\begin{proposition} \label{admissible set and slice}
\begin{enumerate}
\item For any slice $S$ in $\widetilde{Q}$, we have a cut 
\[ C_S := Q_1 \setminus \pi(S_1) \]
in $Q$.
\item $\pi$ gives an isomorphism $S \to Q_{C_S}$ of quivers.
\end{enumerate}
\end{proposition}

\begin{proof}
(1) Let
\[x_1 \tol[30]{\scriptstyle a_1} x_2 \tol[30]{\scriptstyle a_2} \cdots \tol[30]{\scriptstyle a_n} x_{n+1} \tol[30]{\scriptstyle a_{n+1}} x_1 \]
be an $(n+1)$-cycle in $Q$.
We only have to show that there exists precisely one $i \in \{1, \ldots, n+1\}$ such that the arrow $a_i$ does not lie in $\pi(S_1)$.

Let $\widetilde{Q}'$ be the full subquiver of $\widetilde{Q}$ defined by $\widetilde{Q}'_0:=\pi^{-1}(\{x_1,\cdots,x_{n+1}\})$. Then $\widetilde{Q}'$ is isomorphic to the $A_\infty^\infty$ quiver
\[\cdots \to[30] y_{-1} \to[30] y_0 \to[30] y_1 \to[30] y_2 \to[30] \cdots\]
where $\pi(y_{i + (n+1) j})= \{x_i\}$ holds for any $i\in \{1, \ldots, n+1\}, j \in \mathbb{Z}$.
Since $S$ is a slice, there exists $k\in\Z$ such that the $n+1$ vertices $y_k,y_{k+1},\cdots,y_{k+n}$ belong to $S_0$ and any other $y_i$ does not belong to $S_0$. Take $k' \in \{1, \ldots, n+1\}$ such that $k - k' \in (n+1) \mathbb{Z}$. 
Then the $n$ arrows
\[ x_{k'} \tol[30]{\scriptstyle a_{k'}} x_{k'+1} \tol[30]{\scriptstyle a_{k'+1}} \cdots \tol[30]{\scriptstyle a_n} x_{n+1} \tol[30]{\scriptstyle a_{n+1}} x_1 \tol[30]{\scriptstyle a_1} \cdots \tol[30]{\scriptstyle a_{k'-2}} x_{k'-2} \]
belong to $\pi(S_1)$, and $x_{k'-1}\tol[30]{\scriptstyle a_{k'-1}} x_{k'}$ does not belong to $\pi(S_1)$.

(2) By Definition~\ref{def.combslice}(1), $\pi \colon S_0 \to (Q_{C_S})_0 = Q_0$ is bijective and $\pi \colon S_1 \to (Q_{C_S})_1$ is injective. Since $(Q_{C_S})_1 = \pi(S_1)$ by our construction, we have that $\pi$ is an isomorphism.
\end{proof}

\begin{example}
Two slices and the corresponding cuts for $n=1$ and $s=4$:
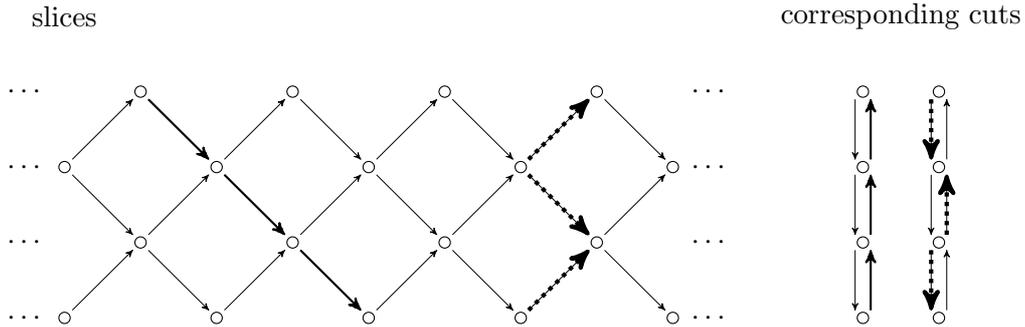
\[ \begin{tikzpicture}[xscale=1,yscale=1]\
 \node at (2,4) {slices};
 \foreach \x in {1,2,3,4,5}
  \node (30\x) at (2*\x,0) [vertex] {};
 \foreach \x in {1,2,3,4}
  \node (21\x) at (2*\x+1,1) [vertex] {};
 \foreach \x in {0,1,2,3,4}
  \node (12\x) at (2*\x+2,2) [vertex] {};
 \foreach \x in {0,1,2,3}
  \node (03\x) at (2*\x+3,3) [vertex] {};
 \foreach \x in {1,2,3,4}
  {
   \draw [->,ultra thin] (30\x) -- (21\x);
   \draw [->,ultra thin] (21\x) -- (12\x);
  }
 \foreach \x in {0,1,2,3}
  \draw [->,ultra thin] (12\x) -- (03\x);
 \foreach \x/\y in {1/2,2/3,3/4,4/5}
  \draw [->,ultra thin] (21\x) -- (30\y);
 \foreach \x/\y in {0/1,1/2,2/3,3/4}
  {
   \draw [->,ultra thin] (12\x) -- (21\y);
   \draw [->,ultra thin] (03\x) -- (12\y);
  }
 \foreach \x in {0,1,2,3}
  {
   \node at (1.5,\x) {$\cdots$};
   \node at (10.5,\x) {$\cdots$};
  }
 \draw [->,thick] (030) -- (121);
 \draw [->,thick] (121) -- (212);
 \draw [->,thick] (212) -- (303);
 \draw [->,ultra thick,dotted] (304) -- (214);
 \draw [->,ultra thick,dotted] (123) -- (214);
 \draw [->,ultra thick,dotted] (123) -- (033);
 \node at (13,4) {corresponding cuts};
 \node (P1) at (12.5, 3) [vertex] {};
 \node (P2) at (12.5, 2) [vertex] {};
 \node (P3) at (12.5, 1) [vertex] {};
 \node (P4) at (12.5, 0) [vertex] {};
 \draw [->,ultra thin] (P1.south west) -- (P2.north west);
 \draw [->,ultra thin] (P2.south west) -- (P3.north west);
 \draw [->,ultra thin] (P3.south west) -- (P4.north west);
 \draw [->,ultra thin] (P2.north east) -- (P1.south east);
 \draw [->,ultra thin] (P3.north east) -- (P2.south east);
 \draw [->,ultra thin] (P4.north east) -- (P3.south east);
 \draw [->,thick] (P2.north east) -- (P1.south east);
 \draw [->,thick] (P3.north east) -- (P2.south east);
 \draw [->,thick] (P4.north east) -- (P3.south east);
 \node (Q1) at (13.5, 3) [vertex] {};
 \node (Q2) at (13.5, 2) [vertex] {};
 \node (Q3) at (13.5, 1) [vertex] {};
 \node (Q4) at (13.5, 0) [vertex] {};
 \draw [->,ultra thin] (Q1.south west) -- (Q2.north west);
 \draw [->,ultra thin] (Q2.south west) -- (Q3.north west);
 \draw [->,ultra thin] (Q3.south west) -- (Q4.north west);
 \draw [->,ultra thin] (Q2.north east) -- (Q1.south east);
 \draw [->,ultra thin] (Q3.north east) -- (Q2.south east);
 \draw [->,ultra thin] (Q4.north east) -- (Q3.south east);
 \draw [->,ultra thick,dotted] (Q1.south west) -- (Q2.north west);
 \draw [->,ultra thick,dotted] (Q3.north east) -- (Q2.south east);
 \draw [->,ultra thick,dotted] (Q3.south west) -- (Q4.north west);
\end{tikzpicture} \]
Some slices and corresponding cuts for $n=2$ and $s=3$ can be found in Table~\ref{table.slices_ca-sets}.
\begin{table}
\[ \begin{tikzpicture}[xscale=1.6,yscale=.4]
 \node at (2,28.5) {slices:};
 \foreach \l in {0,...,6}
  {
   \node (A\l) at (0,4*\l) [vertex] {};
   \node (B\l) at (1,4*\l+1) [vertex] {};
   \node (C\l) at (2,4*\l+2) [vertex] {};
   \node (D\l) at (2,4*\l) [vertex] {};
   \node (E\l) at (3,4*\l+1) [vertex] {};
   \node (F\l) at (4,4*\l) [vertex] {};
   \draw [->,ultra thin] (A\l) -- (B\l);
   \draw [->,ultra thin] (B\l) -- (C\l);
   \draw [->,ultra thin] (B\l) -- (D\l);
   \draw [->,ultra thin] (C\l) -- (E\l);
   \draw [->,ultra thin] (D\l) -- (E\l);
   \draw [->,ultra thin] (E\l) -- (F\l);
  }
 \foreach \l/\m in {0/1,1/2,2/3,3/4,4/5,5/6}
  {
   \draw [->,ultra thin] (D\l) -- (A\m);
   \draw [->,ultra thin] (E\l) -- (B\m);
   \draw [->,ultra thin] (F\l) -- (D\m);
  }
 \foreach \i in {0,...,4}
  {
   \node at (\i,27) {$\vdots$};
   \node at (\i,-1) {$\vdots$};
  }
 \draw [->,very thick] (A5) -- (B5);
 \draw [->,very thick] (B5) -- (C5);
 \draw [->,very thick] (B5) -- (D5);
 \draw [->,very thick] (C5) -- (E5);
 \draw [->,very thick] (D5) -- (E5);
 \draw [->,very thick] (E5) -- (F5);
 \draw [->,ultra thick,dotted] (B3) -- (C3);
 \draw [->,ultra thick,dotted] (B3) -- (D3);
 \draw [->,ultra thick,dotted] (C3) -- (E3);
 \draw [->,ultra thick,dotted] (D3) -- (E3);
 \draw [->,ultra thick,dotted] (E3) -- (F3);
 \draw [->,ultra thick,dotted] (D3) -- (A4);
 \draw [->,very thick,densely dashed] (F0) -- (D1); 
 \draw [->,very thick,densely dashed] (B1) -- (C1);
 \draw [->,very thick,densely dashed] (B1) -- (D1);
 \draw [->,very thick,densely dashed] (C1) -- (E1);
 \draw [->,very thick,densely dashed] (D1) -- (E1);
 \draw [->,very thick,densely dashed] (D1) -- (A2);
 \node at (7,28.5) {corresponding cuts:};
 \foreach \l in {1,3,5}
  {
   \node (A\l) at (5,4*\l-.5) [vertex] {};
   \node (B\l) at (6,4*\l+1) [vertex] {};
   \node (C\l) at (7,4*\l+2.5) [vertex] {};
   \node (D\l) at (7,4*\l-.5) [vertex] {};
   \node (E\l) at (8,4*\l+1) [vertex] {};
   \node (F\l) at (9,4*\l-.5) [vertex] {};
   \draw [->,ultra thin] (A\l) -- (B\l);
   \draw [->,ultra thin] (B\l) -- (C\l);
   \draw [->,ultra thin] (B\l) -- (D\l);
   \draw [->,ultra thin] (C\l) -- (E\l);
   \draw [->,ultra thin] (D\l) -- (E\l);
   \draw [->,ultra thin] (E\l) -- (F\l);
   \draw [->,ultra thin] (D\l) -- (A\l);
   \draw [->,ultra thin] (E\l) -- (B\l);
   \draw [->,ultra thin] (F\l) -- (D\l);
  }
 \draw [->,very thick] (D5) -- (A5);
 \draw [->,very thick] (E5) -- (B5);
 \draw [->,very thick] (F5) -- (D5);
 \draw [->,ultra thick,dotted] (A3) -- (B3);
 \draw [->,ultra thick,dotted] (E3) -- (B3);
 \draw [->,ultra thick,dotted] (F3) -- (D3);
 \draw [->,very thick,densely dashed] (A1) -- (B1);
 \draw [->,very thick,densely dashed] (E1) -- (B1);
 \draw [->,very thick,densely dashed] (E1) -- (F1);
\end{tikzpicture} \]
\caption{Some slices and corresponding cuts for $n=2$ and $s=3$} \label{table.slices_ca-sets}
\end{table}
\end{example}

Now we state the first main assertion of this subsection, which will be proven in the next one.

\begin{theorem} \label{new_theorem_C}
The correspondence $S \mapsto C_S$ in Proposition~\ref{admissible set and slice} gives a bijection between $\nu_n$-orbits of slices in $\widetilde{Q}$ and cuts in $Q$.
\end{theorem}

Let us introduce the following notion.

\begin{definition} \label{slice}
Let $S$ be a slice in $\widetilde{Q}$.
\begin{enumerate}
\item Let $x$ be a source of $S$. Define a full
subquiver $\mu_x^+(S)$ of $\widetilde{Q}$ by removing
$x$ from $S$ and adding $\nu_n^- x$.
\item Dually, for each sink $x$ of $S$, we define $\mu_x^-(S)$.
\end{enumerate}
We call the process of replacing a slice $S$ by $\mu_x^+(S)$ or $\mu_x^-(S)$ \emph{mutation} of slices.
\end{definition}

\begin{proposition} \label{prop.mutation.ca-slice}
In the setup of Definition~\ref{slice}(1) we have the following.
\begin{enumerate}
\item Any successor of $x$ in $\widetilde{Q}$ belongs to $S$,
and any predecessor of $x$ in $\widetilde{Q}$ does not belong to $S$.
\item Any successor of $\nu_n^- x$ in $\widetilde{Q}$ does not belong to $\mu_x^+(S)$,
and any predecessor of $\nu_n^- x$ in $\widetilde{Q}$ belongs to $\mu_x^+(S)$.
\item $\mu_x^+(S)$ is again a slice, and $\nu_n^- x$ is a sink of $\mu_x^+(S)$.
\item We have $C_{\mu_x^+(S)}=\mu_{\pi(x)}^+(C_S)$.
\end{enumerate}
\end{proposition}

\begin{proof}
(1) Let $C_S$ be the cut given in Proposition~\ref{admissible set and slice}. Then $x$ is a source of $Q_{C_S}$. By Remark~\ref{rem.adm_subset_ca}(1), we have the assertion.

(2) The former assertion follows from the former assertion in (1) and 
the definition of slice.

Take a predecessor $y$ of $\nu_n^- x$ and an integer $i$ such that
$\nu_n^i y \in S_0$.
If $i>0$, then we have $i=1$ since there exists a path from $\nu_n^i y$
to $\nu_n^- x$ passing through $\nu_n y$. This is a contradiction
to the latter assertion of (1) since $\nu_n y$ is a predecessor of $x$.
Thus we have $i\le 0$. Since there exists a path from $x$ to $\nu_n^i y$
passing through $y$, we have $y\in S_0$.

(3) By (2), $\nu_n^- x$ is a sink of $\mu_x^+(S)$.
We only have to show that $\mu_x^+(S)$ is convex.
We only have to consider paths $p$ in $\widetilde{Q}$ starting at a vertex
in $\mu_x^+(S)$ and ending at $\nu_n^- x$.
Since any predecessor of $\nu_n^- x$ in $\widetilde{Q}$ belongs to $S$ by (2)
and $S$ is convex, any vertex appearing in $p$ belongs to $\mu_x^+(S)$.

(4) This is clear from (1) and (2).
\end{proof}

The following is the second main statement in this section, which will be proven in the next one.

\begin{theorem} \label{new_theorem_D}
The slices in $\widetilde{Q}$ are transitive under successive mutation.
\end{theorem}

\begin{remark}
Note that one can prove Theorem~\ref{new_theorem_D} by using the categorical argument in Lemma~\ref{lemma.slicestrans}. But we will give a purely combinatorial proof in the next subsection since it has its own interest.
\end{remark}

Clearly Theorem~\ref{new_theorem_B} is an immediate consequence of Proposition~\ref{prop.mutation.ca-slice}(4) and Theorems~\ref{new_theorem_C} and \ref{new_theorem_D} above. \vskip.5em

We now work towards a proof of Proposition~\ref{prop.new_theorem_A}. We identify a slice $S$ in $\widetilde{Q}$ with the direct sum of all objects in $\mathcal{D}_{\Lambda}$ corresponding to vertices in $S$.

\begin{lemma} \label{lemma.slices_adm-set_comb}
\begin{enumerate}
\item $\End_{\mathcal{D}_{\Lambda}}(S) \iso \Lambda_{C_S}$.
\item Let $x$ be a source of $S$. If $S$ is a tilting complex in $\mathcal{D}_{\Lambda}$, then $\mu_x^+(S)$ is an $n$-APR tilting $\Lambda_{C_S}$-module.
\end{enumerate}
\end{lemma}

\begin{proof}
(1) $\pi$ gives an isomorphism $S \to C_S$. It is easily checked that the relations for $\mathcal{U}$ correspond to those for $\widehat{\Lambda}$.

(2) This is clear from the definition.
\end{proof}

\begin{proposition} \label{prop.slice_is_iteratedAPR}
For any slice $S$ in $\widetilde{Q}$, the corresponding object $S \in \mathcal{D}_{\Lambda}$ is an iterated $n$-APR tilting complex.
\end{proposition}

\begin{proof}
This is clear for the slice consisting of the vertices of the form $(\ell_1, \ldots \ell_{n+1} \! : \! 0)$ by Theorem~\ref{theorem.quiverU}(2). We have the assertion by Theorem~\ref{new_theorem_D} and Lemma~\ref{lemma.slices_adm-set_comb}(2).
\end{proof}

\begin{proof}[Proof of Proposition~\ref{prop.new_theorem_A}]
By Theorem~\ref{new_theorem_C} there exists a slice $S$ in $\widetilde{Q}$ such that $C = C_S$. Take a source $y$ of $S$ such that $x = \pi(y)$. By Lemma~\ref{lemma.slices_adm-set_comb}(1) we can identify $\Lambda_C$ with $S$. By Lemma~\ref{lemma.slices_adm-set_comb}(2) and Proposition~\ref{prop.slice_is_iteratedAPR}, $\mu_x^+(S)$ is an $n$-APR tilting $\Lambda_C$-module with
\[ \End_{\mathcal{D}_{\Lambda}}(\mu_x^+(S)) \iso \Lambda_{C_{\mu_x^+(S)}} = \Lambda_{\mu_x^+(C)}. \]
Thus the assertion follows.
\end{proof}

\subsection{Proof of Theorems~\ref{new_theorem_C} and \ref{new_theorem_D}} \label{subsect.first_combi}

In this subsection we give the proofs of Theorems~\ref{new_theorem_C} and \ref{new_theorem_D} which were postponed in Subsection~\ref{subsect.comb_slices}. We postpone further (to Subsection~\ref{subsect.cycle_lemma}) the proof of Proposition~\ref{prop.generator}, a technical classification result needed in the proofs here.

We need the following preparation.

\begin{definition}
\begin{enumerate}
\item We denote by $\walk(Q)$ the set of \emph{walks} in $Q$ (that is, finite sequences of arrows and inverse arrows such that consecutive entries involve matching vertices). For a walk $p$ we denote by $\mathfrak{s}(p)$ and $\mathfrak{e}(p)$ the starting and ending vertex of $p$, respectively. A walk $p$ is called \emph{cyclic} if $\mathfrak{s}(p) = \mathfrak{e}(p)$.
\item We define an equivalence relation $\sim$ on $\walk(Q)$ as the transitive closure of the following relations:
\begin{enumerate}
\item $aa^{-1}\sim e_{\mathfrak{s}(a)}$ and $a^{-1}a\sim e_{\mathfrak{e}(a)}$ for any $a\in Q_1$.
\item If $p\sim q$, then $rpr'\sim rqr'$ for any $r$ and $r'$.
\end{enumerate}
\end{enumerate}
Similarly we define $\walk(\widetilde{Q})$ and the equivalence relation $\sim$ on $\walk(\widetilde{Q})$.
\end{definition}

For a walk $p = a_1 \cdots a_n$ we denote by $p^{-1} := a_n^{-1}\cdots a_1^{-1}$ the inverse walk.

Any map $\omega \colon Q_1 \to A$ with an abelian group $A$ is naturally extended to a map $\omega \colon \walk(Q)\to A$ by putting $\omega(a^{-1}) := -\omega(a)$ for any $a \in Q_1$ and 
\[\omega(p) := \sum_{i=1}^\ell \omega(b_i)\]
for any walk $p=b_1\cdots b_\ell$. We define $\omega\colon \walk(\widetilde{Q}) \to A$ by $\omega(p) := \omega(\pi (p))$.
Clearly these maps $\omega \colon \walk(Q)\to A$ and $\omega \colon \walk(\widetilde{Q})\to A$ are invariant under the equivalence relation $\sim$.

In particular, we define maps
\[ \phi_i \colon \walk(Q) \to \Z \quad \text{ and } \quad \Phi = (\phi_1, \ldots, \phi_{n+1}) \colon \walk(Q) \to \Z^{n+1} \]
by setting $\phi_i(a):= \delta_{ij}$ for any arrow $a$ of type $j$ in $Q$.

\begin{definition}
We denote by $G$ be the set of cyclic walks satisfying
\[p \sim (q_1c_1^{\pm1}q_1^{-1})(q_2c_2^{\pm1}q_2^{-1})\cdots(q_\ell c_\ell^{\pm1}q_\ell^{-1})\]
for some walks $q_i$ and $(n+1)$-cycles $c_i$.
\end{definition}

We will prove Theorems~\ref{new_theorem_C} and \ref{new_theorem_D} by using the following result, which will be shown in the next subsection.

\begin{proposition} \label{prop.generator}
Any cyclic walk on $Q$ belongs to $G$.
\end{proposition}

Using this, we will now prove the following proposition, telling us that on $Q_C$ the value $\Phi(p)$ depends only on $\mathfrak{s}(p)$ and $\mathfrak{e}(p)$.

\begin{proposition}\label{crucial}
Let $C$ be a cut of $Q$.
\begin{enumerate}
\item For any cyclic walk $p$ on $Q_C$, we have $\Phi(p)=0$.
\item For any walks $p$ and $q$ on $Q_C$ satisfying $\mathfrak{s}(p)=\mathfrak{s}(q)$ and $\mathfrak{e}(p)=\mathfrak{e}(q)$, we have $\Phi(p)=\Phi(q)$.
\end{enumerate}
\end{proposition}

To prove Proposition \ref{crucial}, we define a map
\[\phi_C \colon \walk(Q)\to\Z\]
by setting
\[\phi_C(a):=\left\{\begin{array}{cc}
1&\mbox{ if $a\notin C$}\\
-n&\mbox{ if $a\in C$}
\end{array}\right.\]
for any arrow $a\in Q_1$.

\begin{lemma}\label{invariance}
For any cyclic walk $p$ on $Q$, we have $\phi_C(p)=0$.
\end{lemma}

\begin{proof}
Any $(n+1)$-cycle $C$ satisfies $\phi_C(c) = 0$. By Proposition~\ref{prop.generator} we have the assertion.
\end{proof}

We define a map
\[\ell_C \colon \walk(Q)\to\Z\]
by putting
\[\ell_C(a):=\left\{\begin{array}{cc}
0 & \text{ if } a \notin C \\
1 & \text{ if } a \in C
\end{array}\right.\]
for any arrow $a\in Q_1$.

The following result is clear.

\begin{lemma}\label{phi}
For any $p\in\walk(Q)$ we have $\sum_{i=1}^{n+1}\phi_i(p)=\phi_C(p)+(n+1)\ell_C(p)$.
\end{lemma}

Now we are ready to prove Proposition~\ref{crucial}.

\begin{proof}[Proof of Proposition~\ref{crucial}]
(1) Since $p$ is a cyclic walk, we have $\sum_{i=1}^{n+1} \phi_i(p) f_i = 0$ (with $f_i$ as in Definition~\ref{def.qns}). This implies $\phi_1(p) = \cdots = \phi_{n+1}(p)$.

Since $p$ is a cyclic walk on $Q_C$, we have
\[ \sum_{i=1}^{n+1} \phi_i(p) = \phi_C(p) + (n+1) \ell_C(p) = 0 + (n+1) \cdot 0 = 0 \]
by Lemmas~\ref{invariance} and \ref{phi}.
Thus we have $\phi_1(p) = \cdots = \phi_{n+1}(p) = 0$.

(2) We have $\Phi(p)-\Phi(q)=\Phi(pq^{-1})=0$ by (1).
\end{proof}

The fact that $\widetilde{Q} \to Q$ is a Galois covering is reflected by the following lemma on lifting of walks.

\begin{lemma} \label{lemma.lifting}
Fix $x_0\in Q_0$ and $\widetilde{x}_0\in\widetilde{Q}_0$ such that $\pi(\widetilde{x}_0)=x_0$.
For any walk $p$ in $Q$ with $\mathfrak{s}(p)=x_0$, there exists a unique walk
$\widetilde{p}$ in $\widetilde{Q}$ such that $\mathfrak{s}(\widetilde{p})=\widetilde{x}_0$
and $\pi(\widetilde{p})=p$.
\end{lemma}

\begin{proof}
For any $x\in Q_0$ and $y\in\widetilde{Q}_0$ such that $\pi(y)=x$,
the morphism $\pi \colon \widetilde{Q}\to Q$ gives a bijection from the set of arrows starting (respectively, ending) at $y$
to the set of arrows starting (respectively, ending) at $x$.
Thus the assertion follows.
\end{proof}

We have the following key observation.

\begin{lemma} \label{crucial2}
Fix $x_0\in Q_0$ and $\widetilde{x}_0\in\widetilde{Q}_0$ such that $\pi(\widetilde{x}_0)=x_0$.
For any walks $p$ and $q$ in $Q_C$ satisfying $\mathfrak{s}(p)=\mathfrak{s}(q)=x_0$ and $\mathfrak{e}(p)=\mathfrak{e}(q)$,
then $\widetilde{p}$ and $\widetilde{q}$ as given in Lemma~\ref{lemma.lifting} satisfy $\mathfrak{e}(\widetilde{p})=\mathfrak{e}(\widetilde{q})$.
\end{lemma}

\begin{proof}
By our definition of $\Phi$, we have that $\phi_i(\widetilde{p})$ counts the number of arrows of type $i$ appearing in $\widetilde{p}$.
Since we have $\phi_i(\widetilde{p})=\phi_i(\widetilde{q})$ by Proposition \ref{crucial}, we have that the number of arrows of type $i$
appearing in $\widetilde{p}$ is equal to that in $\widetilde{q}$.
Since $\mathfrak{s}(\widetilde{p})=\mathfrak{s}(\widetilde{q})$, we have $\mathfrak{e}(\widetilde{p})=\mathfrak{e}(\widetilde{q})$.
\end{proof}

Now Theorem~\ref{new_theorem_C} follows from the following result, which allows us to construct slices from cuts.

\begin{proposition} \label{section}
Let $C$ be a cut in $Q$. Fix a vertex $x_0\in Q_0$ and $\widetilde{x}_0\in\pi^{-1}(x_0)$.
\begin{enumerate}
\item There exists a unique morphism $\iota \colon Q_C\to \widetilde{Q}$ of quivers satisfying the following conditions.
\begin{itemize}
\item $\iota(x_0)=\widetilde{x}_0$,
\item the composition $\pi\circ\iota \colon Q_C\to Q$ is the identity on $Q_C$.
\end{itemize}
\item $\iota(Q_C)$ is a slice in $\widetilde{Q}$.
\end{enumerate}
\end{proposition}

\begin{proof}
(1) To give the desired morphism $\iota \colon Q_C\to\widetilde{Q}$ of quivers, we only have to give a map $\iota \colon Q_0\to\widetilde{Q}_0$ between the sets of vertices, satisfying the following conditions.
\begin{itemize}
\item $\iota(x_0)=\widetilde{x}_0$,
\item the composition $\pi\circ\iota \colon Q_C\to Q$ is the identity on $Q_0$,
\item for any arrow $a \colon x\to y$ in $Q_C$, there is an arrow $\iota(x)\to\iota(y)$ in $\widetilde{Q}$.
\end{itemize}

We define $\iota \colon Q_0\to\widetilde{Q}_0$ as follows:
Fix any $x\in Q_0$. We take any walk $p$ in $Q_C$ from $x_0$ to $x$.
By Lemma~\ref{lemma.lifting}, there exists a unique walk $\widetilde{p}$ in $\widetilde{Q}$ such that $\mathfrak{s}(\widetilde{p})=\widetilde{x}_0$ and $\pi(\widetilde{p})=p$.
Then we put $\iota(x):=\mathfrak{e}(\widetilde{p})$. By Lemma \ref{crucial2}, $\iota(x)$ does not depend on the choice of the walk $p$.

We only have to check the third condition above.
Fix an arrow $a \colon x\to y$ in $Q_C$. Take any walk $p$ in $Q_C$ from $x_0$ to $x$.
The walk $pa \colon x_0 \leadsto y$ in $Q_C$ gives the corresponding walk $\widetilde{pa} \colon \widetilde{x}_0 \leadsto \iota(y)$ in $\widetilde{Q}$.
Then $\widetilde{pa}$ has the form $\widetilde{p}\,b$ for an arrow $b \colon \iota(x) \to \iota(y)$ and a walk $\widetilde{p} \colon \widetilde{x}_0 \leadsto \iota(x)$ in $\widetilde{Q}$. Thus the third condition is satisfied.

The uniqueness of $\iota$ is clear.

(2) Fix vertices $x,y\in\iota(Q_C)_0$ and a path $p$ in $\widetilde{Q}$ from $x$ to $y$.
We only have to show that $p$ is a path in $\iota(Q_C)$.

Since $Q_C$ is connected, we can take a walk $q$ on $\iota(Q_C)$ from $x$ to $y$.
Then we have $\Phi(\pi(p))=\Phi(\pi(q))$. We have
\[\phi_C(p)+(n+1)\ell_C(p) =\sum_{i=1}^{n+1}\phi_i(p) = \sum_{i=1}^{n+1}\phi_i(q)=\phi_C(q)+(n+1)\ell_C(q) = \phi_C(q) \]
by Lemma \ref{phi}.
Since we have $\phi_C(p)=\phi_C(q)$ by Lemma \ref{invariance}, we have $\ell_C(p)=0$.
Thus any arrow appearing in $p$ belongs to $\iota(Q_C)$.
\end{proof}

This completes the proof of Theorem~\ref{new_theorem_C}.

In the remainder of this subsection we give a purely combinatorial proof of Theorem~\ref{new_theorem_D}.

For a slice $S$, we denote by $S_0^+$ the subset of $\widetilde{Q}_0$ consisting of sources in $S$.

\begin{lemma}\label{sources determine slices}
The correspondence $S\mapsto S_0^+$ is injective.
\end{lemma}

\begin{proof}
We denote by $S_0'$ be the set of vertices $x$ of $\widetilde{Q}$ satisfying the following conditions.
\begin{itemize}
\item there exists a path in $\widetilde{Q}$ from some vertex in $S_0^+$ to $x$,
\item there does not exist a path in $\widetilde{Q}$ from any vertex in $S_0^+$ to $\nu_n x$.
\end{itemize}
To prove the assertion, we only have to show $S_0=S_0'$.
It is easily seen from the definition of $S_0'$ that each $\nu_n$-orbit in $\widetilde{Q}_0$ contains at most one vertex in $S'_0$. Since $S_0$ is a slice, we only have to show $S_0 \subset S_0'$.

For any $x\in S_0$, there exists a path in $S$ from some vertex in $S_0^+$ to $x$ since $S$ is a finite acyclic quiver.
Assume that there exists a path $p$ in $\widetilde{Q}$ from $y\in S_0^+$ to $\nu_n x$.
Since there exists a path $q$ in $\widetilde{Q}$ from $\nu_n x$ to $x$, we have a path $pq$ from $x$ to $y$.
Since $S$ is convex, we have $\nu_n x \in S_0$, a contradiction to $x\in S_0$.
\end{proof}

For a slice $S$ of $\widetilde{Q}$, define the full subquiver $\widetilde{Q}_S^{\ge0}$ by
\[ (\widetilde{Q}_S^{\ge0})_0 := \bigcup_{\ell\ge0}\nu_n^{\ell}S_0 . \]
Clearly we have $(\widetilde{Q}^{\ge0}_{\mu^+_x(S)})_0= (\widetilde{Q}^{\ge0}_S)_0\cup\{\nu_n^- x\}$.

\begin{lemma} \label{new_lemma_E} Let $S$ be a slice in $\widetilde{Q}$. Then there exists a numbering $S_0=\{x_1,\cdots,x_N\}$ of vertices of $S$ such that the following conditions are satisfied.
\begin{enumerate}
\item $x_{i+1}$ is a source in
$\mu^+_{x_i}\circ\cdots\circ\mu^+_{x_1}(S)$ for any $0\le i<N$.
\item We have $\mu^+_{x_N}\circ\cdots\circ\mu^+_{x_1}(S)=\nu_n^- S$.
\end{enumerate}
\end{lemma}

\begin{proof} When we have $x_1,\cdots,x_{i-1}\in S_0$, then we define
$x_i$ as a source of the quiver \\ $S \setminus \{x_0,\cdots,x_{i-1}\}$.
It is easily checked that the desired conditions are satisfied.
\end{proof}

For slices $S$ and $T$ in $\widetilde{Q}$, we write $S\le T$ if $(\widetilde{Q}_S^{\ge0})_0\subseteq(\widetilde{Q}_T^{\ge0})_0$.
In this case, we put
\[d(S,T):=\#((\widetilde{Q}_T^{\ge0})_0 \setminus (\widetilde{Q}_S^{\ge0})_0).\]

Now we are ready to prove Theorem~\ref{new_theorem_D}.

Let $S$ and $T$ be slices.
We can assume $S\le T$ by Lemma~\ref{new_lemma_E}. We use the induction on $d(S,T)$.
If $d(S,T)=0$, then we have $S=T$.
Assume $d(S,T)>0$.
By Lemma \ref{sources determine slices}, there exists a source $x$ of $S$ such that $x \notin T_0$.
Then we have $\mu_x^+(S)\le T$ and $d(\mu_x^+(S),T)=d(S,T)-1$.
By our assumption on induction, $\mu_x^+(S)$ is obtained from $T$ by a successive mutation.
Thus $S$ is obtained from $T$ by a successive mutation.
\qed

\subsection{Proof of Proposition~\ref{prop.generator}} \label{subsect.cycle_lemma}

We complete the proof of Theorem~\ref{theorem.main_typeA} by filling the remaining gap, that is by proving Proposition~\ref{prop.generator}.

For a walk $p$, we denote by $|p|$ the length of $p$. For $x,y\in Q_0$, we denote by $d(x,y)$ the minimum of the length of walks on $Q$ from $x$ to $y$.

It is easily checked (similarly to the proof of Lemma~\ref{addition}) that $d(x,y) = d(x', y')$ whenever $x-y = x'-y'$.

\begin{lemma}\label{shortest cycle}
Let $p$ be a cyclic walk.
Assume that, for any decomposition $p=p_1p_2p_3$ of $p$,
\[d(\mathfrak{s}(p_2),\mathfrak{e}(p_2))=\min\{|p_2|,|p_3 p_1| \}\]
holds. Then one of the following conditions holds.
\begin{enumerate}
\item $p$ or $p^{-1}$ is an $(n+1)$-cycle.
\item $p$ has the form $p=a_1^{\epsilon_1}\cdots a_\ell^{\epsilon_\ell} b_1^{-\epsilon_1}\cdots b_\ell^{-\epsilon_\ell}$ with an injective map $\sigma \colon \{1,\cdots,\ell\}\to\{1,\cdots,n\}$, arrows $a_i$ and $b_i$ of type $\sigma(i)$ and $\epsilon_i \in \{\pm1\}$.
\end{enumerate}
\end{lemma}

\begin{proof}
(i) Assume that $p$ contains an arrow of type $i$ and an inverse arrow of type $i$ at the same time.
Take any decomposition $p=q_1aq_2b^{-1}$ with arrows $a,b$ of type $i$ and walks $q_1$ and $q_2$.
If $|q_2| < |q_1|$, then we have
\[d(\mathfrak{s}(q_1),\mathfrak{e}(q_1)) = d(\mathfrak{e}(q_2),\mathfrak{s}(q_2)) <\min\{ |q_1|, |aq_2b^{-1}|\},\]
a contradiction. Similarly, $|q_1| < |q_2|$ cannot occur.
Consequently, we have $|q_1| = |q_2|$.
This equality also implies that $q_1$ and $q_2$ do not contain arrows or inverse arrows of type $i$.
Consequently, $p$ satisfies Condition~(2).

(ii) In the rest, we assume that $p$ does not satisfy Condition~(2).
By (i), we have that $p$ does not contain an arrow of type $i$ and an inverse arrow of type $i$ at the same time. Without loss of generality we may assume $\phi_i(p) > 0$. Then $p$ contains exactly $\phi_i(p)$ arrows of type $i$ for each $i$, and does not contain inverse arrows.

Since $p$ is a cyclic walk, we have an equality $\sum_{i=1}^{n+1} \phi_i(p) f_i=0$. This implies $\phi := \phi_1(p)=\cdots=\phi_{n+1}(p)$. We shall show that $\phi = 1$. Then Condition~(1) is satisfied.

Assume that $\phi >1$ holds.

Assume that $|p|$ is odd, so $n+1$ is also odd.
We write $p=ap_1p_2$ with an arrow $a$ and $|p_1| = |p_2|$.
By our assumption, we have $d(\mathfrak{s}(p_1),\mathfrak{e}(p_1))= |p_1| = |p_2| = d(\mathfrak{s}(p_2),\mathfrak{e}(p_2))$.
This implies that less than $\frac{n+1}{2}$ types of arrows appear in $p_1$ (respectively, $p_2$).
Since $\phi >1$, either $p_1$ or $p_2$ contains an arrow of same type with $a$.
Hence $p$ contains less than $\frac{n+1}{2}+\frac{n+1}{2}=n+1$ kinds of arrows, a contradiction.

Assume that $|p|$ is even.
We write $p=ap_1bp_2$ with arrows $a,b$ and $|p_1| = |p_2|$.
By our assumption, we have 
\[d(\mathfrak{s}(ap_1),\mathfrak{e}(ap_1))=|ap_1|=|bp_2|=d(\mathfrak{s}(bp_2),\mathfrak{e}(bp_2)).\]
This implies that at most $\frac{n+1}{2}$ types of arrows appear in $ap_1$ (respectively, $bp_2$).
Since all kinds of arrows appear in $p$, we have that $ap_1$ and $bp_2$ contain exactly $\frac{n+1}{2}$ types of arrows,
and there is no common type of arrows in $ap_1$ and $bp_2$.
By the same argument, we have that $p_1b$ and $p_2a$ contain exactly $\frac{n+1}{2}$ types of arrows,
and there is no common type of arrows in $p_1b$ and $p_2a$.

Since $\phi>1$, either $p_1$ or $p_2$ contains an arrow of same type with $a$.
Assume that $p_1$ contains an arrow of same type with $a$.
Then $p_1b$ contains more than $\frac{n+1}{2}$ types of arrows, a contradiction.
Similarly, $p_2$ does not contain an arrow of same type with $a$, a contradiction.
\end{proof}

\begin{lemma}\label{commutator}
The cyclic walk in Lemma~\ref{shortest cycle}(2) belongs to $G$ if $\epsilon_1=\cdots=\epsilon_\ell=1$.
\end{lemma}

\begin{proof}
By Lemma \ref{addition}, $a_1 \cdots a_{\ell}$ extends to an $(n+1)$-cycle $a_1 \cdots a_{n+1}$ in $Q$ with $a_i$ an arrow of type $\sigma(i)$ ($\sigma\in\mathfrak{S}_{n+1}$ an extending the original $\sigma$). Since
\[a_1\cdots a_\ell b_1^{-1}\cdots b_\ell^{-1}\sim (a_1\cdots a_{n+1})(b_\ell\cdots b_1a_{\ell+1}\cdots a_{n+1})^{-1}\in G,\]
we have the assertion.
\end{proof}

\begin{lemma}\label{commutator2}
Let $p a^{\epsilon} b^{\epsilon'} q$ and $p c^{\epsilon'} d^{\epsilon} q$ be cyclic walks on $Q$, with $\epsilon, \epsilon' \in \{\pm 1\}$, such that $a$ and $d$ are arrows of the same type, and $b$ and $c$ are arrows of the same type. Then one of them belongs to $G$ if and only if the other does.
\end{lemma}

\begin{proof}
We have the equivalences
\begin{align*}
p a b q & \sim (p (a b d^{-1} c^{-1}) p^{-1}) (p c d q) && (\epsilon = \epsilon' = 1), \\
p a b^{-1} q & \sim (p c^{-1} (c a b^{-1} d^{-1}) c p^{-1}) (p c^{-1} d q) && (\epsilon = 1, \epsilon' = -1),
\end{align*}
and similar for the remaining cases. The claim now follows from Lemma~\ref{commutator}.
\end{proof}

\begin{lemma}\label{addition2}
Let $x\in Q_0$, $\sigma \colon \{1,\ldots,\ell\}\to \{1, \ldots, n+1\}$ be an injective map and $\epsilon_i \in \{\pm 1\}$ for any $1\leq i\leq \ell$.
Assume that $x+\sum_{j=1}^{i} \epsilon_j f_{\sigma(j)}$ and $x+\sum_{j=i}^{\ell} \epsilon_j f_{\sigma(j)}$ belong to $Q_0$ for any $0\le i\le\ell$.
Then, for any subset $I$ of $\{1, \ldots, \ell\}$, we have that $x+\sum_{j\in I}\epsilon_j f_{\sigma(j)}$ belongs to $Q_0$.
\end{lemma}

\begin{proof}
We only have to show that
\[0\le x_{\sigma(i)} - \epsilon_i < s \mbox{ and }\ 0\le x_{\sigma(i)+1} + \epsilon_i < s\]
hold for any $i\in\{1,2,\cdots,\ell\}$.

If $\sigma(i)-1\notin\{\sigma(1),\cdots,\sigma(i-1)\}$, then the $\sigma(i)$-th entry of $x+\sum_{j=1}^{i} \epsilon_j f_{\sigma(j)}$ is equal to $x_{\sigma(i)} - \epsilon_i$. If $\sigma(i)-1\notin\{\sigma(i+1),\cdots,\sigma(\ell)\}$, then the $\sigma(i)$-th entry of $x+\sum_{j=i}^\ell \epsilon_j f_{\sigma(j)}$ is equal to $x_{\sigma(i)} - \epsilon_i$. In each case we have the former inequality.

The latter inequality can be shown similarly.
\end{proof}

We now look at the following special case of Proposition~\ref{prop.generator}.

\begin{lemma}
Any cyclic walk satisfying the condition in Lemma~\ref{shortest cycle}(2) belongs to $G$.
\end{lemma}

\begin{proof}
Let $p$ be the cyclic walk in Lemma~\ref{shortest cycle}(2), and $x=\mathfrak{s}(p)$. It follows from Lemma~\ref{addition2} that for any $\varrho \in \mathfrak{S}_{\ell}$, $\widetilde{Q}$ contains the cyclic walk
\[p_{\varrho}:=(\leftsub{\varrho}{a}_{1}^{\epsilon_{\varrho(1)}}) \cdots (\leftsub{\varrho}{a}_{\ell}^{\epsilon_{\varrho(\ell)}}) b_1^{-\epsilon_1}\cdots b_\ell^{-\epsilon_\ell}\]
starting from $x$, where $\leftsub{\varrho}{a}_i$ is an arrow of type $\sigma(\varrho(i))$. When $\varrho$ is given by $\varrho(i)=\ell + 1 - i$, the cyclic walk $p_{\varrho}$ is
\[ p_{\varrho} = b_\ell^{\epsilon_\ell}\cdots b_1^{\epsilon_1}
b_1^{-\epsilon_1}\cdots b_\ell^{-\epsilon_\ell},\]
which clearly belongs to $G$.
Using Lemma~\ref{commutator2} repeatedly, we see that all $p_{\varrho}$ lie in $G$, so in particular $p = p_{\id} \in G$.
\end{proof}

Now we are ready to prove Proposition~\ref{prop.generator}.

\begin{proof}[Proof of Proposition~\ref{prop.generator}]
We use the induction on $|p|$.
Assume that $p$ does not satisfy the conditions (1) and (2) in Lemma \ref{shortest cycle}.
Then we can write $p=p_1p_2p_3$ with
\[d(\mathfrak{s}(p_2),\mathfrak{e}(p_2))<\min\{ |p_2|, |p_3 p_1| \}.\]
Take a walk $q$ from $\mathfrak{s}(p_2)$ to $\mathfrak{e}(p_2)$ with $|q| = d(\mathfrak{s}(p_2),\mathfrak{e}(p_2))$.
Then we have
\[p \sim (p_1qp_3)(p_3^{-1}(q^{-1}p_2)p_3)\]
and $|p_1qp_3| = |p_1| + |q| + |p_3| < |p|$ and $|q^{-1}p_2| = |q| + |p_2| < |p|$.
By our assumption of induction, $p_1qp_3$ and $q^{-1}p_2$ belong to $G$.
Thus $p$ also belongs to $G$.
\end{proof}

\subsection{$(n+1)$-preprojective algebras} \label{subsect.preproj}

We end this paper by showing that the algebras $\widehat{\Lambda}^{(n,s)}$ have the following properties:

\begin{theorem} \label{theorem.preproj_n-repfin}
$\widehat{\Lambda}^{(n,s)}$ is self-injective weakly $(n+1)$-representation-finite, and we have a triangle equivalence $\stabmod \widehat{\Lambda}^{(n,s)} \approx \mathcal{C}_{\Lambda^{(n+1,s-1)}}^{n+1}$.
\end{theorem}

We remark that this proof relies heavily on a results from \cite{IO2} (also see Remark~4.17 in that paper). We need the following observation.

\begin{proposition} \label{prop.preproj}
For any cut $C$ of $Q^{(n,s)}$, the $(n+1)$-preprojective algebra of the $n$-representation-finite algebra $\Lambda^{(n,s)}_C$ is $\widehat{\Lambda}^{(n,s)}$.
\end{proposition}

\begin{proof}
The quiver morphism $\pi \colon \widetilde{Q}^{(n,s)} \to Q^{(n,s)}$ gives an equivalence $\mathcal{U} / \nu_n \approx \proj \widehat{\Lambda}^{(n,s)}$ of categories, which sends $\Lambda^{(n,s)}_C$ to $\widehat{\Lambda}^{(n,s)}$. Thus the $(n+1)$-preprojective algebra of $\Lambda^{(n,s)}_C$ is
\[ \End_{\mathcal{U} / \nu_n}(\Lambda^{(n,s)}_C)^{\op} \iso \widehat{\Lambda}^{(n,s)}. \qedhere \]
\end{proof}

\begin{proof}[Proof of Theorem~\ref{theorem.preproj_n-repfin}]
By Proposition~\ref{prop.preproj} the algebra $\widehat{\Lambda}^{(n,s)}$ is the $(n+1)$-preprojective algebra of the $n$-representation finite algebra $\Lambda^{(n,s)}_C$ for any cut $C$. Thus, by \cite[Corollary~3.4]{IO2}, $\widehat{\Lambda}^{(n,s)}$ is self-injective.

Moreover, by \cite[Theorem~1.1]{IO2}, we have
\[ \stabmod \widehat{\Lambda}^{(n,s)} \approx \mathcal{C}_{\Gamma}^{n+1}, \]
where $\Gamma$ is the stable $n$-Auslander algebra of $\Lambda^{(n,s)}_C$. In particular for $C = C_0$ we have $\Gamma$ is the stable $n$-Auslander algebra $\Lambda^{(n,s)}$, which is $\underline{\End}_{\Lambda^{(n,s)}}(M^{(n,s)}) \iso \Lambda^{(n+1,s-1)}$.

The fact that $\widehat{\Lambda}^{(n,s)}$ is weakly $(n+1)$-representation finite now follows from the existence of an $(n+1)$-cluster tilting object in $\mathcal{C}_{\Lambda^{(n+1,s-1)}}^{n+1}$ by work of Amiot (\cite{CC, C_PhD} -- also see \cite[Corollary~4.16]{IO2}).
\end{proof}

\newcommand{\etalchar}[1]{$^{#1}$}

\end{document}